\documentclass[12pt]{amsart}

\usepackage{a4wide}
\usepackage[latin9]{inputenc}
\usepackage{amsmath}
\usepackage{amsfonts}
\usepackage{amssymb}
\usepackage{amsthm}
\usepackage{subfigure}
\usepackage[usenames,dvipsnames]{color}
\usepackage{pstricks}
\usepackage{graphicx}
\usepackage[all]{xy}
\newtheorem{theo}{Theorem}[section]
\newtheorem{prop}[theo]{Proposition}
\newtheorem{coro}[theo]{Corollary}
\newtheorem{lemm}[theo]{Lemma}

\theoremstyle{definition}
\newtheorem{def1}[theo]{Definition}
\theoremstyle{remark}
\newtheorem{rema}[theo]{Remark}
\newcommand{\Op}{\operatorname{Op}}

\newcommand{\nwc}{\newcommand}
\nwc{\eps}{\epsilon}
\nwc{\ep}{\epsilon}
\nwc{\vareps}{\varepsilon}
\nwc{\Oph}{\operatorname{Op}_\hbar}
\nwc{\la}{\langle}
\nwc{\ra}{\rangle}

\nwc{\mf}{\mathbf} 
\nwc{\blds}{\boldsymbol} 
\nwc{\ml}{\mathcal} 

\nwc{\defeq}{\stackrel{\rm{def}}{=}}

\nwc{\cE}{\ml{E}}
\nwc{\cN}{\ml{N}}
\nwc{\cO}{\ml{O}}
\nwc{\cP}{\ml{P}}
\nwc{\cU}{\ml{U}}
\nwc{\cV}{\ml{V}}
\nwc{\cW}{\ml{W}}
\nwc{\tU}{\widetilde{U}}
\nwc{\IN}{\mathbb{N}}
\nwc{\IR}{\mathbb{R}}
\nwc{\IZ}{\mathbb{Z}}
\nwc{\IC}{\mathbb{C}}
\nwc{\IT}{\mathbb{T}}
\nwc{\IS}{\mathbb{S}}
\nwc{\tP}{\widetilde{P}}
\nwc{\tPi}{\widetilde{\Pi}}
\nwc{\tV}{\widetilde{V}}
\nwc{\supp}{\operatorname{supp}}
\nwc{\rest}{\restriction}

\begin{document}

\title[Spectral analysis of Morse-Smale gradient flows]{Spectral analysis of Morse-Smale gradient flows}

\author[Nguyen Viet Dang]{Nguyen Viet Dang}

\address{Institut Camille Jordan (U.M.R. CNRS 5208), Universit\'e Claude Bernard Lyon 1, B\^atiment Braconnier, 43, boulevard du 11 novembre 1918, 
69622 Villeurbanne Cedex }

\email{dang@math.univ-lyon1.fr}

\author[Gabriel Rivi\`ere]{Gabriel Rivi\`ere}

\address{Laboratoire Paul Painlev\'e (U.M.R. CNRS 8524), U.F.R. de Math\'ematiques, Universit\'e Lille 1, 59655 Villeneuve d'Ascq Cedex, France}

\email{gabriel.riviere@math.univ-lille1.fr}

\begin{abstract} 
On a smooth, compact and oriented manifold without boundary, we give a complete description of the correlation function of a Morse-Smale gradient flow 
satisfying a certain nonresonance assumption. This is done by analyzing precisely the spectrum of the generator of such a flow acting on certain
anisotropic spaces of currents. In particular, 
we prove that this dynamical spectrum is given by linear combinations with integer coefficients of the 
Lyapunov exponents at the critical points of the Morse function. Via this spectral analysis and in analogy with Hodge-de Rham theory, 
we give an interpretation of the Morse 
complex as the image of the de Rham complex under the spectral projector on the kernel of the generator of the flow. This allows us to 
recover classical results from differential topology such as the Morse inequalities and Poincar\'e duality.

\end{abstract}



\maketitle

\section{Introduction}

Consider a smooth ($\ml{C}^{\infty}$) flow $(\varphi^t)_{t\in\IR}$ acting on a smooth, compact, oriented manifold $M$ which has no boundary 
and which is of dimension $n\geq 1$. A natural question to raise is whether the limit 
$$\lim_{t\rightarrow +\infty}\varphi^{-t*}(\psi)$$
exists for any smooth function $\psi$ defined on $M$. This is of course very unlikely to happen in general, and a natural setting 
where one may expect some convergence is the class of dynamical systems with hyperbolic behaviour and for a nice enough 
reference measure. For instance, if $\varphi^t$ is a topologically transitive 
Anosov flow~\cite{Ano67} and if we study the weak limit with respect to a so-called Gibbs measure, it is known from the works of Bowen, Ruelle and Sinai that such 
a limit exists and is equal to the average of $\psi$ with respect to the Gibbs 
measure\footnote{Recall that a well-known example is the 
Liouville measure for the geodesic flow on a negatively curved manifold.}~\cite{Si72, BowRu75}. If one is able to show that this 
equilibrium state exists, a second natural question to raise is: can one describe the fluctuations? For instance, what is the rate of convergence to this state? 

These problems are naturally related to the study of the operator generating the flow:
$$\ml{L}:\psi\in\ml{C}^{\infty}(M)\mapsto -\frac{d}{dt}\left(\varphi^{-t*}(\psi)\right)|_{t=0}\in\ml{C}^{\infty}(M).$$
Note that, by duality, this operator acts on the space of distributions $\ml{D}'(M)$. In recent years, many progresses have been made in the 
study of such operators acting on suitable Banach spaces of distributions when the flow $\varphi^t$ enjoys the Anosov property. 
In~\cite{Li04}, Liverani defined  Banach spaces of distributions with ``anisotropic 
H\"older regularity'' for which he could make a precise spectral analysis of $\ml{L}$ in the case of \emph{contact Anosov flows},
and from which he could deduce that, for every $t\geq 0$ and for every $\psi_1,\psi_2$ in $\ml{C}^{\infty}(M)$,
\begin{equation}\label{e:liverani}C_{\psi_1,\psi_2}(t):=\int_M\varphi^{-t*}(\psi_1)\psi_2d\text{vol}_g=\int_M\psi_1d\text{vol}_g\int_M\psi_2d\text{vol}_g+\ml{O}_{\psi_1,\psi_2}(e^{-\Lambda t}),\end{equation}
where $\Lambda>0$ is some fixed positive constant related to the spectral gap of $\ml{L}$ and where $\text{vol}_g$ is the Riemannian volume. 
His construction was inspired by similar results for diffeomorphisms~\cite{BlKeLi02} and by a proof of 
Dolgopyat which holds in the $2$-dimensional case~\cite{Do98}. Introducing Banach spaces inside $\ml{D}'(M)$ contrasts
with earlier approaches to these questions where symbolic coding of Anosov flows was used to describe the 
weak convergence of $\varphi^{-t*}(\psi)$. 
For more general Anosov flows, Butterley and Liverani also showed how this direct approach 
allows to make a meromorphic extension for the Laplace transform of the correlation function $C_{\psi_1,\psi_2}(t)$ to the entire 
half plane~\cite{BuLi07}. This extended earlier works of Pollicott~\cite{Po85} and Ruelle~\cite{Ru87a} which were also based on the use of symbolic dynamics. 
Such poles (and their corresponding eigenstates) describe in some sense the fine structure of the long time dynamics and are often called \emph{Pollicott-Ruelle resonances}. 
Pushing further this direct approach~\cite{GiLiPo13}, Giulietti, Liverani and Pollicott extended this spectral analysis to anisotropic spaces of currents and they 
proved that, for any smooth Anosov flow, 
the Ruelle zeta function has a meromorphic extension to $\IC$. In the case of Anosov geodesic flows satisfying certain pinching assumptions, 
they also showed that~\eqref{e:liverani} also holds for the Bowen-Margulis measure (and not only with respect to the Riemannian volume). 
In parallel to this approach via spaces of anisotropic H\"older distributions, it was observed that the spectral analysis of Anosov 
flows can in fact be understood as a semiclassical problem which fits naturally 
in the theory of semiclassical resonances~\cite{HeSj86, DyZw16}. Building on earlier works for Anosov diffeomorphisms by Baladi-Tsujii~\cite{BaTs07,BaTs08} and 
Faure-Roy-Sj\"ostrand~\cite{FaRoSj08} involving microlocal tools, this kind of approach to Pollicott-Ruelle resonances was developped for Anosov flows
by Tsujii~\cite{Ts10, Ts12}, Faure-Sj\"ostrand~\cite{FaSj11}, Faure-Tsujii~\cite{FaTs15, FaTs13} and Dyatlov-Zworski~\cite{DyZw13}. We refer to the 
survey article of Gou\"ezel~\cite{Go15} for a recent account on these progresses.

Regarding the important steps made in the Anosov case, it is natural to understand to what extent these methods can be adapted to more 
general dynamical systems satisfying weaker chaotic features. A natural extension to consider is the class of Axiom $A$ systems~\cite{Sm67, BowRu75}. 
In the case of nonsingular Axiom $A$ flows, this was analyzed by Dyatlov and Guillarmou
who showed that Pollicott-Ruelle resonances can be defined \emph{locally} on a small 
neighborhood of any basic set of a given Axiom A flow with no critical points -- see also~\cite{Ru87a, BaTs07, BaTs08, GoLi08} 
for earlier results in the case of Axiom A diffeomorphisms. Here, we are aiming at analyzing the simplest class of Axiom $A$ flows, namely 
Morse-Smale gradient flows. Yet, compared with the above references, our objective is to give a global description of the 
correlation function and not only in a neighborhood of the basic sets (here the critical points). Recall that gradient flows associated with 
a Morse function are also interesting because of their deep connections with differential topology which first appeared in the pioneering works of Thom~\cite{Th49} and Smale~\cite{Sm60, Sm61}. Our microlocal 
approach to the properties of gradient flows will allow to give new (spectral) interpretations of some results of 
Laudenbach~\cite{Lau92, Lau12} and Harvey-Lawson~\cite{HaLa00, HaLa01} and to recover some classical facts from differential topology such as the 
finiteness of the Betti numbers, the Morse inequalities and Poincar\'e duality. Recall that an alternative spectral 
approach to Morse theory (based on Hodge theory) was introduced by Witten in~\cite{Wi82}. For a more detailed exposition on these 
relations between dynamical systems, topology and spectral theory, we refer to the classical survey article of 
Bott~\cite{Bo88}.

\section{Statement of the main results}

\subsection{Dynamical framework} We fix $f$ to be a smooth ($\ml{C}^{\infty}$) Morse function, meaning that $f$ has only finitely many critical points and 
that these points are non degenerate. We denote by $\text{Crit}(f)$ the set of critical points. For simplicity, we shall always assume that $f$ is excellent in 
the sense that, given $a\neq b$ in $\text{Crit}(f)$, one has $f(a)\neq f(b)$. If we consider a smooth Riemannian metric $g$ on $M$, we can define a vector 
field $V_f$ as follows
\begin{equation}
 \forall (x,v)\in TM,\ d_xf(v)=\la V_f(x),v\ra_{g(x)}. 
\end{equation}
This vector field generates a complete flow on $M$~\cite[Ch.~6]{Lau12} that we denote by $\varphi_f^t$. Given any point $a$ in $\text{Crit}(f)$, we can define 
its stable (resp. unstable) manifold, i.e.
$$W^{s/u}(a):=\left\{x\in M:\ \lim_{t\rightarrow +/-\infty}\varphi_f^{t}(x)=a\right\}.$$
One can show that $W^{s}(a)$ (resp. $W^u(a)$) is an embedded submanifold in $M$ of dimension $0\leq r\leq n$ (resp. $n-r$) where $r$ is the index of the critical 
point~\cite{Web06}. Note also that $W^u(a)\cap W^s(a)=\{a\}$. A remarkable property of these submanifolds is that \emph{they form a partition of 
the manifold} $M$~\cite{Th49}, i.e.
$$M=\bigcup_{a\in\text{Crit}(f)} W^s(a),\ \text{and}\ \forall a\neq b,\ W^s(a)\cap W^s(b)=\emptyset.$$
The above property also holds true for the unstable manifolds. This partition in stable (and unstable) leaves will play a central role in our analysis. Among these Morse gradient flows, Smale introduced 
a particular family of flows~\cite{Sm60}. Namely, given any $a$ and $b$ in $\text{Crit}(f)$, he required that \emph{$W^s(a)$ and $W^u(b)$ intersect transversally} whenever they 
intersect. This assumption also turns out to be a crucial ingredient to make our proofs work. We will use the 
terminology \textbf{Morse-Smale} for any gradient flow enjoying the above 
properties. Finally, for any point $a$ in $\text{Crit}(f)$, we define $L_f(a)$ as the only matrix satisfying
\begin{equation}\label{e:def-Lyapunov}\forall \xi,\eta\in T_aM,\ d^2_af(\xi,\eta)=g_a(L_f(a)\xi,\eta).\end{equation}
As $a$ is a nondegenerate critical point, $L_f(a)$ is symmetric with respect to $g_a$ and invertible. Its eigenvalues are called the \textbf{Lyapunov exponents} at the point $a$ and we write them as
$$\chi_1(a)\leq\ldots\leq \chi_r(a)<0< \chi_{r+1}(a)\leq\ldots\leq\chi_n(a),$$
where $r$ is the index of the critical point $a$. All along the article, we will often make the assumption that $(f,g)$ is a \emph{smooth Morse pair} inducing a Morse-Smale gradient flow. By smooth Morse pair, we 
roughly mean that there is a smooth linearizing chart for $V_f$ near any critical point. By the Sternberg-Chen Theorem~\cite{Ch63}, this is for instance
satisfied when, for every critical point $a$, the Lyapunov exponents $(\chi_{j}(a))_{1\leq j\leq n}$ are rationally independent -- see paragraph~\ref{sss:adapted} for more details.

\begin{rema} Let us fix some conventions. We will denote by $\IN^*$ the set of positive integers $\{1,2,\ldots\}$ while $\IN$ will be the set of nonnegative integers $\{0,1,2,\ldots\}$. 
We will use $\alpha=(\alpha_1,\ldots,\alpha_n)$ for a multi-index in $\IN^n$. Given any critical point $a$ of $f$, we denote by $|\chi(a)|$ the vector $(|\chi_1(a)|,\ldots,|\chi_n(a)|)$.
For any $0\leq k\leq n$, $\Omega^k(M)$ will be the space of smooth differential forms of degree $k$ and $\ml{D}^{\prime,k}(M)$ will be the topological dual of $\Omega^{n-k}(M)$, 
i.e. the space of currents of degree $k$ (or of dimension $n-k$). For an introduction to the theory of currents, we refer to~\cite{dRh80, Sch66}.
\end{rema}

\subsection{Correlation function}

The main concern of the article will be to perform a spectral analysis of the operator $\ml{L}_{V_f}$ acting on appropriate spaces of currents. As an application of our analysis, we will prove 
the following result on the asymptotic behaviour of $\varphi_f^{-t*}(\psi)$:
\begin{theo}\label{t:fullasympt-nonresonant} Let $\varphi_f^t$ be a Morse-Smale gradient flow all of whose Lyapunov exponents are rationally independent. 
Let $0\leq k\leq n$. 

Then, for every $a$ in $\operatorname{Crit}(f)$ and for every $\alpha$ in $\IN^n$, there exists a continuous linear operator,
$$\pi_{a,k}^{(\alpha)}:\Omega^k(M)\rightarrow \ml{D}^{\prime,k}(M),$$
such that
for every $\Lambda>0$, for every $\psi_1$ in $\Omega^k(M)$, for every $\psi_2$ in $\Omega^{n-k}(M)$, and for every $t\geq 0$,
$$\int_M\varphi_f^{-t*}(\psi_1)\wedge\psi_2=\sum_{a\in\operatorname{Crit}(f)}\sum_{\alpha\in \IN^n:\alpha.|\chi(a)|<\Lambda}e^{-t \alpha.|\chi(a)| }
\int_M \pi_{a,k}^{(\alpha)} (\psi_1)\wedge\psi_2+\ml{O}_{\psi_1,\psi_2}(e^{-\Lambda t})$$
where $\alpha.\vert\chi(a)\vert=\sum_{i=1}^n\alpha_i\vert\chi_i(a)\vert$.
Moreover, for every $a$ in $\operatorname{Crit}(f)$ and for every $\alpha$ in $\IN^n$, 
one has\footnote{We will in fact give a (rather combinatorial) explicit expression of $\operatorname{rk}(\pi_{a,k}^{(\alpha)} )$ in the proof -- see Remark~\ref{r:combinatorics}.}
\begin{itemize}
\item for every $\psi_1$ in $\Omega^k(M)$, the support of $\pi_{a,k}^{(\alpha)} (\psi_1)$ is contained in $\overline{W^u(a)}$,
\item $0\leq \operatorname{rk}(\pi_{a,k}^{(\alpha)})\leq 2^n$,
\item $\operatorname{rk}(\pi_{a,k}^{(0)})=\delta_{k,r}$ where $r$ is the index of $a$,
\item $\operatorname{rk}(\pi_{a,k}^{(\alpha)})=\frac{n!}{k!(n-k)!}$ for every $\alpha\in(\IN^{*})^n$.
\end{itemize}

\end{theo}

This Theorem gives us an asymptotic expansion at any order of the \emph{correlation function}\footnote{Here we make a small abuse of 
terminology as correlation functions are usually concerned with invariant measures.}
associated with a Morse-Smale gradient flow. As we shall see, 
we will also provide a more or less explicit expression of the operator 
$\pi_{a,k}^{(\alpha)}$ near the critical point $a$. These operators of course depend on the choice of the Riemannian metric used to define the 
gradient flow as well as the Lyapunov exponents appearing in the above asymptotic expansion. 
Each term appearing in the sum looks also very much like the expansion obtained by Faure and Tsujii in the case of linear models acting on $\IR^n$~\cite[Ch.~3-4]{FaTs15}. 
Here, one of the main difficulty will be to understand how these local models can be glued together in order to obtain a result valid on the whole manifold. We also 
mention that similar expansions appear via techniques from complex analysis in the case of analytic expanding circle maps arising from finite Blaschke products~\cite{BanSlJu15,Bannaud2016}.
Even if our flows are in some sense degenerate Axiom $A$ flows, this result is also closely related to the recent results of Dyatlov 
and Guillarmou on Pollicott-Ruelle resonances for open systems~\cite{DyGu14}. Maybe the main difference with this reference is that 
Theorem~\ref{t:fullasympt-nonresonant} holds \emph{globally} on $M$ and 
not only in a neighborhood of the critical points\footnote{We also note that~\cite{DyGu14} made the assumption that 
the vector field does not vanish.}. If we consider the time $1$ map of the flow, 
the induced diffeomorphism $h=\varphi_f^{-1}$ is probably one 
of the simplest example of an Axiom A (but not Anosov) diffeomorphism~\cite{Sm67}, 
even if it is a kind of ``trivial'' example as all the basic sets are reduced to fixed points. Resonances of general Axiom A diffeomorphisms 
were studied by Ruelle in~\cite{Ru87b} via methods of symbolic dynamics while the direct functional 
approach of these dynamical systems was developped by Baladi--Tsujii~\cite{BaTs07, BaTs08} and by Gou\"ezel--Liverani~\cite{GoLi08}. 
As for the case of Axiom A flows treated in~\cite{DyGu14}, these results focused on the dynamics near a convenient basic set, i.e. we 
restrict ourselves to $\psi_1$ and $\psi_2$ supported in a neighborhood of a fixed critical point.
Here, due to the simple structure of the diffeomorphism, the asymptotic expansion of the correlation function 
can be made without restrictions on $\psi_1$ and $\psi_2$.

We made a kind of (global) nonresonance assumption in the statement of Theorem~\ref{t:fullasympt-nonresonant}. This assumption ensures that the 
generator of the flow does not have any Jordan blocks. 
As far as we know, the above Theorem gives the first example of such diffeomorphisms where \emph{all} the Jordan blocks in 
the spectrum are trivial. In~\cite{FrenkelLosevNekrasov1}, Frenkel, Losev and Nekrasov were lead to similar problems in the context of quantum field theory. In particular, as a by product of their analysis, 
they obtain the complete asymptotic for the correlation function of the flow associated with the height function on the $2$-sphere endowed with its canonical metric. In that case, the 
Lyapunov exponents are all equal to $\pm 1$ (hence resonant) and they prove that there are indeed infinitely many polynomial factors (hence non trivial Jordan blocks) 
in the asymptotic expansion. In the general case (including the case of~\cite{FrenkelLosevNekrasov1}), we still obtain 
an asymptotic expansion for the correlation function which may involve polynomial factors in $t$ except for the peripheral eigenvalue 
$\lambda=0$. We just mention here the \emph{leading term of the asymptotics} in the general case:

\begin{theo}\label{c:leadingasympt} Let $(f,g)$ be a smooth Morse pair generating a Morse-Smale gradient flow $\varphi_f^t$. Let $0\leq k\leq n$.
 
 Then, for every $a$ in $\operatorname{Crit}(f)$ of index $k$, there exist
\begin{itemize}
 \item $U_{a}$ in $\ml{D}^{\prime,k}(M)$ whose support is equal to $\overline{W^u(a)}$,
 \item $S_a$ in $\ml{D}^{\prime,n-k}(M)$ whose support is equal to $\overline{W^s(a)}$,
\end{itemize}
such that, for every 
$$0<\Lambda<\min\left\{|\chi_j(b)|:\ 1\leq j\leq n,\ b\in\operatorname{Crit}(f)\right\},$$ 
for every $(\psi_1,\psi_2)$ in $\Omega^k(M)\times\Omega^{n-k}(M)$, and for every $t\geq 0$,
$$\int_M\varphi_f^{-t*}(\psi_1)\wedge\psi_2=\sum_{a\in\operatorname{Crit}(f):\ \operatorname{ind}(a)=k}\int_M \psi_1\wedge S_a\int_M U_a\wedge\psi_2
+\ml{O}_{\psi_1,\psi_2}(e^{-\Lambda t}).$$
\end{theo}

Except for the remainder term, this result was first proved by Harvey and Lawson via techniques from geometric measure theory and under 
slightly more restrictive assumptions on the gradient flows~\cite{HaLa00, HaLa01}.
This Theorem follows from Propositions~\ref{p:correlation} and~\ref{p:basis}. To our knowledge, the currents $U_a$ and $S_a$ 
appearing in the leading term of the asymptotic expansion were first constructed by Laudenbach~\cite{Lau92, Lau12} 
in the case of a ``locally flat metric adapted to the Morse coordinates'' 
-- see paragraph~\ref{sss:adapted} for the definition. In the following, 
we shall refer to them as Laudenbach's currents. The difficulty is that the submanifolds $W^u(a)$ and $W^s(a)$ are not a priori properly embedded and one has to justify that 
the currents of integration are well defined. Precisely, one can integrate on $W^u(a)$ a differential form $\psi$ whose support is included in a compact part of $W^u(a)$ but integration of a general form 
whose support may intersect the boundary needs to be justified. This can solved by analyzing the mass of the currents near the boundary 
of the unstable (resp. stable manifold) and this requires a careful description 
of the structure of the boundary of $W^u(a)$~\cite{Lau92, Lau12}. Even if it is in a different manner, similar 
difficulties involving the boundary will of course occur at some point in our analysis 
and we shall deal with this problem via dynamical techniques following the works of Smale~\cite{Sm60} -- 
see for instance Lemmas~\ref{l:compact} and~\ref{l:openneighbor}.

After properly defining the spectral framework of our problem, we will recover the existence of these currents as a consequence of our spectral analysis. They correspond 
to the kernel of the operator $\ml{L}_{V_f}$ acting on appropriate anisotropic spaces of currents. The advantage of this approach is that it allows to treat 
more general families of gradient flows and that it sheds a new (spectral) light on these natural dynamical objects. 
A difficulty may be that it relies on microlocal techniques which are maybe 
not as well-known as the geometric measure theory used by Harvey and Lawson
in~\cite{HaLa00, HaLa01} to give an interpretation of these currents as a limit of the correlation function   
under the assumption 
of finite volume -- see also~\cite{Mi15} for generalizations of this result. In any case, Theorem~\ref{t:fullasympt-nonresonant} generalizes 
this type of result in the sense that it does not only give the existence of the limit but also a rate of convergence to this equilibrium state and the full asymptotic expansion as $t\rightarrow+\infty$.

\subsection{Topological interpretation of the leading term}

One of the main application of this dynamical approach to Morse theory is that the partition of the manifold into unstable components 
has beautiful topological 
implications~\cite{Th49, Sm60, HaLa01, Lau12}. 
In section~\ref{s:topology}, we will explain how to recover some classical results from differential topology 
(e.g. finiteness of Betti numbers, Poincar\'e duality, Morse inequalities) via our spectral approach and via some
 analogies with Hodge-de Rham theory. For that purpose, we can set, for every $0\leq k\leq n$,
$$C^k(f):=\text{span}\left\{U_a: \text{ind}(a)=k\right\}.$$
In analogy with Hodge-de Rham theory where one uses the formula $\Delta=d\circ d^*+d^*\circ d$, we can write the Cartan formula
$$\ml{L}_{V_f}=d\circ i_{V_f}+i_{V_f}\circ d,$$
where $d$ is the coboundary operator and $i_{V_f}$ is the contraction by the vector field $V_f$. We will verify that $(C^*(f),d)$ induces a cohomological 
complex while $(C^*(f),i_{V_f})$ induces a homological complex. The first complex is known as the Morse complex (also sometimes called the Thom-Smale-Witten complex) 
and we will call the second one the \textbf{Morse-Koszul} complex. Using our analysis of Morse-Smale gradient flows, we will 
give a purely spectral proof of the following results:
\begin{theo}\label{t:topology} Let $(f,g)$ be a smooth Morse pair generating a Morse-Smale gradient flow $\varphi_f^t$. Then, the following holds.
\begin{enumerate}
 \item The maps
 $$\mathbb{P}^{(k)}:\psi\in \Omega^k(M)\mapsto \sum_{a\in\operatorname{Crit}(f):\ \operatorname{ind}(a)=k}\left(\int_M S_a\wedge\psi\right) U_a\in C^k(f)$$
 induce a quasi--isomorphism between the cohomology of the de Rham complex $(\Omega^*(M),d)$ and the cohomology of the Morse complex $(C^*(f),d)$.
 \item The homology in degree $k$ of the Morse-Koszul complex $(C^*(f),i_{V_f})$ is equal to $C^k(f)$.
\end{enumerate}
\end{theo}
The first part of the Theorem is due to Laudenbach in the case of a ``locally flat metric adapted to the Morse coordinates''~\cite{Lau92, Lau12}. It recovers the classical fact that the de Rham complex $(\Omega^*(M),d)$ 
is quasi-isomorphic to the Morse complex. From this, it is classical to deduce 
the finiteness of the Betti numbers and the so-called Morse inequalities -- see section~\ref{s:topology} for more details. 
The second part seems new.

\subsection{About the proof: spectral analysis of $\ml{L}_{V_f}$} 

One of the main difficulty one encounters when trying to describe 
this spectrum is to find good Banach spaces containing $\Omega^*(M)$ and where $\ml{L}_{V_f}$ has nice spectral properties such as discrete spectrum. 
Here, we will in fact closely follow the construction of Faure-Sj\"ostrand in~\cite{FaSj11} (see also~\cite{DyZw13, FaTs13} in the case of currents) and explain how to adapt it to our dynamical framework. One 
of the main issue we have to deal with is the 
asymptotic behaviour of the Hamiltonian lift of $\varphi_f^t$. In particular, we have to verify that the attractor and the repeller of the normalized Hamiltonian flow are compact subsets -- 
see Lemmas~\ref{l:compact} and~\ref{l:openneighbor}. This is one of the first place where we will strongly use our extra assumptions on the flow, namely 
the Smale transversality and the (smooth) linearizing property near every critical point.

After setting properly this dynamical framework and its asymptotic properties, we can closely follow the construction 
from~\cite{FaSj11} which requires minor (but necessary) modifications that will be 
described in section~\ref{s:anisotropic} -- see also appendices~\ref{a:escape} and~\ref{a:discrete}. Given any 
$\Lambda>0$ and any $0\leq k\leq n$, this procedure allows us to construct an anisotropic Sobolev space $\ml{H}_k^{m_{\Lambda}}(M)\subset\ml{D}^{\prime,k}(M)$ 
such that
$$-\ml{L}_{V_f}^{(k)}:\ml{H}_k^{m_{\Lambda}}(M)\rightarrow \ml{H}_k^{m_{\Lambda}}(M),$$
and such that the operator has only discrete spectrum with finite multiplicity in the half plane $\{\text{Re}(z)>-\Lambda\}.$ According to~\cite[Th.~1.5]{FaSj11}, these values are independent of the 
choice of our anisotropic space. These complex numbers are called the Pollicott-Ruelle resonances of $-\ml{L}_{V_f}^{(k)}$~\cite{Po85, Ru87a}, 
and they correspond to the poles of the meromorphic extension of 
$(-\ml{L}_{V_f}^{(k)}-z)^{-1}:\Omega^k(M)\rightarrow \ml{D}^{\prime,k}(M)$ to the complex plane. We denote these poles by $\ml{R}_k(f,g)$.

In the case where $(f,g)$ is a smooth Morse pair inducing a Morse-Smale gradient flow, we will obtain several results on their structure that we will now describe:
\begin{enumerate}
 \item Any element in $\ml{R}_{k}(f,g)$ is contained in $(-\infty,0]$ and is a linear combination with integer coefficients of the Lyapunov exponents at a fixed critical point $a$ (Proposition~\ref{p:ruelle-spectrum}).
  \item If all the Lyapunov exponents are rationally independent, we can determine the multiplicity of every element in $\ml{R}_{k}(f,g)$ and 
 the local expression of the eigenmodes near the associated critical point (Propositions~\ref{p:generalized-laudenbach} and~\ref{p:basis}). 
 \item The algebraic multiplicity of an eigenvalue is always equal to its geometric multiplicity (Proposition~\ref{p:strong-jordan}). 
 \item In particular, we can determine Weyl asymptotics in terms of the Lyapunov exponents (Proposition~\ref{p:weyl}) 
and we give a spectral version of the classical Lefschetz trace formula (Proposition~\ref{p:trace}).
\end{enumerate}
The combination of all these results allows to prove Theorem~\ref{t:fullasympt-nonresonant} and Theorem~\ref{c:leadingasympt}-- see Section~\ref{s:jordan}. 
The proofs of these different spectral results will heavily rely on the construction of the spaces $\ml{H}^{m_{\Lambda}}_k(M)$ that implies that our 
eigenmodes have a certain prescribed Sobolev regularity.

\subsection{Organization of the article} In section~\ref{s:dynamics}, we gather some crucial dynamical preliminaries and introduce some notations that will be used all along the article. In 
section~\ref{s:anisotropic}, we make use of our dynamical assumptions in order to construct anisotropic spaces of currents which are adapted to our problem. As our construction is 
very close to the one in~\cite{FaSj11}, we mostly focus on the differences, namely the construction of the escape function whose detailed proof is 
postponed to appendix~\ref{a:escape}. In section~\ref{s:spectrum}, we make use of the regularity properties of the eigenmodes to 
prescribe the values of the Pollicott-Ruelle eigenvalues. Section~\ref{s:jordan} gives a complete description of the spectrum 
(multiplicities of the eigenvalues, local structure of the eigenmodes). We explain in section~\ref{s:topology} how 
to deduce some classical results of differential topology from the results obtained 
in the previous sections. In appendix~\ref{a:escape}, we give the complete proof of the construction of the escape function. Appendix~\ref{a:discrete} is devoted to a brief reminder of~\cite{FaSj11} 
concerning the proof of Proposition~\ref{p:eigenvalues}. Finally, appendix~\ref{a:asymptotic} collects 
some facts on asymptotic expansions that we use at several stages of our work.

\subsection*{Acknowledgements} 
We would like to thank Fr\'ed\'eric Faure for many explanations about his recent works on transfer operators and for generously sharing his knowledge on these questions with us. 
We warmly thank Livio Flaminio, Damien Gayet, Colin Guillarmou, Patrick Popescu Pampu, Nicolas Vichery and Tobias Weich for useful discussions related to this work.
We are grateful to Serguei Barannikov, Lie Fu, Camille Laurent-Gengoux and Joan Mill\`es for explanations and comments related to 
the content of section~\ref{s:topology} and more specifically to Koszul complexes. Finally, we thank the anonymous referee for his detailed report 
and his suggestions. The second author is 
partially supported by the Agence Nationale de la Recherche through the Labex CEMPI (ANR-11-LABX-0007-01) and the ANR project GERASIC (ANR-13-BS01-0007-01).

\section{Morse-Smale gradient flows}\label{s:dynamics}

In this section, we briefly collect some facts on the dynamical properties of Morse-Smale gradient flows. The main new results of this section are 
Lemmas~\ref{l:compact} and~\ref{l:openneighbor} which are 
related to earlier works of Smale~\cite{Sm60}. These two lemmas are the crucial ingredients to develop 
the machinery of anisotropic Sobolev spaces of Faure and Sj\"ostrand~\cite{FaSj11}.
We also fix some conventions that we will use all along the article. 
For the well-known results, we follow the lines of~\cite{Lau12, Web06} and 
we refer to these references for a more detailed exposition. Recall that, in all 
the article, $M$ denotes a smooth ($\ml{C}^{\infty}$), oriented, compact manifold without boundary. 

\subsection{Gradient flows}

Let $f:M\rightarrow \IR$ be a smooth function on $M$. If we fix a Riemannian metric $g$ on $M$ (compatible with our orientation), then we can define the corresponding gradient vector field as follows:
$$\forall x\in M,\ \la \text{grad} f(x),.\ra_{g(x)}:=df(x).$$
In local coordinates, this can we written as
$$V_f(x):=\text{grad} f(x)=\sum_{i,j=1}^n g^{ij}(x)\partial_{x_j} f \partial_{x_j},$$
where $(g^{ij}(x))_{1\leq i,j\leq n}$ is the induced Riemannian metric on $T_x^*M$. Under our geometric assumptions (compactness of the manifold), one knows that the gradient vector field induces a complete flow that we denote by 
$$\varphi_f^t:M\rightarrow M.$$ 
If it does not create any particular confusion, we will sometimes use the convention $x(t)=\varphi^t_f(x_0)$ for a fixed $x_0$ in $M$. Note that, 
for any integral curve $t\mapsto x(t)$ of the gradient vector field, one has
\begin{equation}\label{e:gradient-lines}\forall t_1,t_2\in\IR,\ f(x(t_2))-f(x(t_1))=\int_{t_1}^{t_2}\|\text{grad} f(x(t))\|_{g(x(t))}^2dt.\end{equation}
In other words, $f$ can only increase  along the flow lines. Suppose now that $f$ is a \textbf{Morse function}. We denote by $\text{Crit} f$ its critical points. For the sake 
of simplicity, we will also assume our Morse function $f$ to be \textbf{excellent} which means that all critical values are distinct. Recall that 
such functions are dense in 
the topological space $\ml{C}^{\infty}(M,\IR)$. The Morse Lemma tells us
\begin{lemm}[Morse Lemma]
Let $f$ be a Morse function on a Riemannian manifold $(M,g)$.
Then, near any critical point $a$, there is a system of coordinates
$(z_i)_i$ such that the point $a$ is given by $z=0$, and such that
$$f(z)=f(a)-\frac{z_1^2}{2}-\ldots- \frac{z_r^2}{2}+\frac{z_{r+1}^2}{2}+\ldots +\frac{z_n^2}{2},$$ 
for some $0\leq r\leq n$. The integer $r$ is called the \textbf{index} of the critical point $a$. We will either denote it by $r(a)$ or 
$\operatorname{ind}(a)$.
\end{lemm}

An important property of the dynamical system $\varphi_f^t:M\rightarrow M$ is that, 
for any given $x_0$ in $M$, there exist two points $x_-$ and $x_+$ in $\text{Crit} f$ such that
\begin{equation}\label{e:limit-points}
 \lim_{t\rightarrow\pm\infty}\varphi^t_f(x_0)=x_{\pm}.
\end{equation}

\subsection{Stable and unstable manifolds}

Let $a$ be a critical point of $f$. The stable (resp. unstable) manifold $W^s(a)$ (resp. $W^u(a)$) is defined as the set of points $x$ in $M$ satisfying 
$\varphi_f^t(x)\rightarrow a$ as $t\rightarrow +\infty$ 
(resp. $t\rightarrow -\infty$). From \cite[Th.~2.7]{Web06}, one knows that $W^s(a)$ (resp. $W^u(a)$) is a \emph{smooth submanifold} 
of dimension $r(a)$ (resp. $n-r(a)$) where $0\leq r(a)\leq n$ is the index of the critical point $a$. Note that, for more general vector fields 
with an hyperbolic point, the stable (resp. unstable) manifold 
is a priori only injectively immersed in $M$. The fact that we consider a gradient flow allows to show that it is also embedded~\cite[Th.~2.7]{Web06}, 
even if it is not a priori properly embedded.

We will say that the gradient flow $\varphi_f^t$ satisfies the \textbf{Morse-Smale assumption} if for every pair of critical points $(a,b)$, 
the submanifolds $W^s(a)$ and $W^u(b)$ are transversal. Note that, in the case where $a=b$, the intersection of the tangent spaces is in fact reduced to 
$\{0\}$. This transversality assumption ensures the following important property:
\begin{lemm}\label{l:decrease-dimension} Let $\varphi_f^t$ be a Morse-Smale gradient flow. If $a\neq b$ and if $W^s(a)\cap W^u(b)$ is non empty, then
 $$r(b)<r(a).$$
\end{lemm}

Let us now fix some conventions. Given any $x_0$ in $M$, there exists a unique pair of critical points $(x_-,x_+)$ such that $x_0$ belongs to $W^u(x_-)\cap W^s(x_+)$. We define
$$E^u(x_0):=T_{x_0}W^u(x_-)\ \text{and}\ E^s(x_0):=T_{x_0}W^s(x_+).$$
Note that, whenever $x_0$ is not a critical point of $f$, the intersection of these two subspaces is not reduced to $0$ as they both contains the direction of the flow. 
From our transversality assumption, one has $T_{x_0} M=E^u(x_0)+E^s(x_0)$. We refer to paragraph~\ref{ss:Lyapunov} for a more detailed description of 
these subspaces at the critical points of $f$. We can also introduce 
the dual spaces $E_u^*(x_0)$ and $E_s^*(x_0)$ which are defined as the annihilators of these unstable and stable spaces, i.e. 
$E_{u/s}^*(x_0)(E^{u/s}(x_0))=0$. From the Morse-Smale transversality assumption, one can verify that, for any $x_0$ in $M$,
\begin{equation}\label{e:transversality}
 E_u^*(x_0)\cap E_{s}^*(x_0)=\{0\}.
\end{equation}

\subsection{Lyapunov exponents}\label{ss:Lyapunov}

Given every point $x$ in $M$, we define $L_f(x)$ as 
the unique matrix satisfying
$$\forall \xi,\eta\in T_xM,\ \la L_f(x)\xi,\eta\ra_{g(x) }=d_x^2f(\xi,\eta).$$
Let $a$ be an element in $\text{Crit} (f)$. The matrix $L_f(a)$ corresponds to the linearization of $V_f$ at 
the point $a$. It can be shown~\cite[Lemma~2.5]{Web06} that
$$\forall t\in\IR,\ d\varphi_f^t(a)=\exp(tL_f(a)).$$
Moreover, from the definition and from the Morse assumption, one can verify that $L_f(a)$ is an invertible matrix which is symmetric with 
respect to the Riemannian metric $g$. In particular, it is diagonalizable and we denote by $(\chi_j(a))_{j=1,\ldots, n}$ its eigenvalues. 
These nonzero real numbers are called the \textbf{Lyapunov exponents} of the system. They depend both on $f$ and on the metric $g$. We will always 
suppose that $\chi_i(a)<0$ for $1\leq i\leq r$ and $\chi_i(a)>0$ for $r+1\leq i\leq n$. Moreover, there 
exists a basis of eigenvectors which is orthonormal with respect to the metric $g$. According to~\cite[Th.~2.7]{Web06}, the stable (resp. unstable) 
space is in fact equal to the direct sum of eigenspaces corresponding to the negative (resp. positive) eigenvalues of $L_f(a)$.

\subsection{Lift to the cotangent space}\label{ss:Hamiltonian} We will now explain how one can lift this gradient flow to the cotangent space $T^*M$. 
We associate to the vector field $V_f$ an Hamiltonian function $H_f$ which can be written as follows:
$$\forall (x,\xi)\in T^*M,\ H_f(x,\xi):=\xi\left(V_f(x)\right).$$
This Hamiltonian function also induces an Hamiltonian flow that we denote by $\Phi_f^t:T^*M\rightarrow T^*M$. We note that, by construction, 
$$\Phi_f^t(x,\xi):=\left(\varphi_f^t(x),\left(d\varphi^t_f(x)^T\right)^{-1}\xi\right),$$
and that \emph{this flow induces a diffeomorphism between $T^*M-\{0\}$ and $T^*M-\{0\}$.} When it does not lead to any confusion, we will also write 
$\Phi_f^t(x,\xi)=(x(t),\xi(t))$. 
We note that this flow induces a smooth flow on the unit cotangent bundle $S^*M$, i.e.
$$\forall t\in\IR,\ \forall(x,\xi)\in S^*M,\ \tilde{\Phi}_f^t(x,\xi)=\left(\varphi_f^t(x),\frac{\left(d\varphi^t_f(x)^T\right)^{-1}\xi}{\left\|\left(d\varphi^t_f(x)^T\right)^{-1}\xi\right\|}\right).$$
We denote by $\tilde{X}_{H_f}$ the induced smooth vector field on $S^*M$.

\subsection{Adapted coordinates}\label{sss:adapted} 

We will say in the following that $(f,g)$ is a \textbf{smooth Morse pair} if, given any critical point $a$ of $f$, one can find an open neighborhood 
$V_a$ of $a$ and a system of \emph{smooth} (meaning $\ml{C}^{\infty}$) local coordinate charts 
$(z_j)_{j=1,\ldots,n}=(x,y)$ such that, in this coordinate chart, the vector field $V_f$ reads
\begin{equation}\label{e:good-coord}V_f:=\sum_{j=1}^n\chi_j(a)z_j\partial_{z_j}=-\sum_{j=1}^r|\chi_j(a)|x_j\partial_{x_j}+\sum_{j=r+1}^n|\chi_j(a)|y_j\partial_{y_j}.\end{equation}
The key point for us is that this change of coordinates is smooth which will allow us to take as many derivatives as we want in the 
following sections where we aim at using microlocal techniques. Requiring that there exists a smooth change of coordinates 
for which the gradient vector field can be linearized may a priori look as a strong 
assumption. We will briefly discuss below two situations where this assumption is satisfied, the second one being in some sense rather general.

When this assumption is satisfied, we shall say that we have an \textbf{adapted system of coordinates}. We note that the function $f$ may not have a nice expression in these coordinates, meaning that $f$ may a priori not have 
a Morse-type expression. In such a coordinate chart, the gradient flow reads
\begin{eqnarray*}
\varphi_f^t(z)& = &(e^{t\chi_1(a)}z_1,\ldots,e^{t\chi_n(a)}z_n)\\
 & = &(e^{-|\chi_1(a)|t}x_1,\ldots,e^{-|\chi_r(a)|t}x_r, e^{|\chi_{r+1}(a)|t}y_{r+1},\ldots, e^{|\chi_n(a)|t}y_n). 
\end{eqnarray*}
Let us fix some conventions that we will use in the following. For every critical point $a$, we denote the change of coordinates by
$$\kappa_a: w\in V_a\subset M\rightarrow (x,y)\in W_a=(-\delta_a,\delta_a)^n\subset\IR^n,$$ 
where $\delta_a>0$ is some small enough parameter. 
\begin{rema}
Let $(U,\kappa)=(u^i)$ be local coordinates on $M$. Whenever the chart is of class $\ml{C}^1$, one can lift in a canonical way these coordinates into coordinates $(u^i,v_j)$ on the cotangent space $T^*M$ 
by using $(\ml{K},T^*U)$, where $\ml{K}(x,\xi)=(\kappa(x),(d\kappa(x)^T)^{-1}\xi)$. When we make a change of coordinates $(\tilde{u}^i,\tilde{v}_j)$, one can verify that $\tilde{v}$ is the image of $v$ 
under a linear transformation (depending only on the coordinate charts $(u^i)$ and $(\tilde{u}^i)$). 
\end{rema}
We can also write the expression of the Hamiltonian flow in the corresponding 
adapted coordinate chart near a critical point $a$. In such a chart, one can write
$$H_{f}(z,\zeta)=\sum_{j=1}^n\chi_j(a)z_j\zeta_j=-\sum_{j=1}^r|\chi_j(a)|x_j\xi_j+\sum_{j=r+1}^n|\chi_j(a)|y_j\eta_j.$$
In particular, the map $\Phi_f^t$ can be written in this local coordinate chart as
$$\Phi_f^t(z,\zeta)=(e^{\chi_1(a)t} z_1,\ldots e^{\chi_n(a)t} z_n; 
e^{-\chi_1(a)t} \zeta_1,\ldots e^{-\chi_n(a)t} \zeta_n).$$

\subsubsection{Locally flat metrics}

The vector field $V_f$ is a priori not linear in the chart of the Morse Lemma. In fact, there might be no Morse chart in which 
$g_{ij}(0)$ is diagonal and the vector field linear. We will call the metric $g$ \textbf{locally flat with respect} to $f$ if, 
for any critical point of $f$, there exists a \emph{smooth Morse chart} $(z_i)$ in which the vector field $V_f$ has the linear form
$$V_f(z)= \sum_{j=1}^n\chi_j(a) z_j\partial_{z_j}. $$ 
Such flows are also sometimes refered as tame flows. This type of locally flat metrics appears for instance 
in~\cite{HaLa00, Lau12}. It is shown in~\cite{HaLa00} that, given any Morse function, one can find an adapted metric $g$ such that the Morse-Smale transversality assumption is satisfied. 
Moreover, they proved that this property is more or less generic among such metrics.

\subsubsection{Sternberg-Chen Theorem}
We would like to justify that asking for a smooth change of coordinates which linarizes the gradient flow is in some sense generic. 
For that purpose, we just recall Sternberg-Chen's Theorem on the linearization of vector fields near hyperbolic critical points~\cite{Ch63} (see also~\cite[Th.~9, p.50]{Ne69}):
\begin{theo}[Sternberg-Chen]\label{t:Sternberg-Chen} Let $X(x)=\sum_ja_j(x)\partial_{x_j}$ be a smooth vector field defined in a neighborhood of $0$ in $\IR^n$. Suppose that $X(x)=0$. Denote by $(\chi_j)$ the 
eigenvalues of $L:=(\partial_{x_k}a_j(0))_{k,j}$. Suppose that the eigenvalues satisfy the \textbf{non resonant assumption},
$$\forall\ k_1,\ldots,k_n\in\IZ\ \text{s.t}\ k_1,\ldots k_n\geq 2,\ \forall\ 1\leq i\leq n,\ \chi_i\neq\sum_{j=1}^nk_j\chi_j.$$
Then, there exists a smooth diffeomorphism $h$ which is defined in a neighborhood of $0$ such that
$$X\circ h(x)=dh\circ (Lx.\partial_x).$$
\end{theo}
The classical Grobman-Hartman Theorem~\cite{Hart60} ensures the existence of a conjugating homeomorphism. The crucial point of the 
Sternberg-Chen Theorem is that the conjugating map is smooth provided some non resonance assumption is made. Applying this Theorem locally near the 
critical points of $f$ allows to show the existence of a smooth and adapted system of coordinates. Note that this non-resonant assumption on the 
eigenvalues is for instance satisfied if, for every $a$ in $\text{Crit} (f)$, the Lyapunov exponents $(\chi_j(a))_{j=1,\ldots, n}$ 
are rationally independent. In section~\ref{s:jordan}, we will in fact make the assumption that \emph{all} the 
Lyapunov exponents are rationally independent.

\subsection{Attractor and repeller of the Hamiltonian flow}

We now introduce the following subsets of $T^*M$:
$$\Gamma_+=\bigcup_{x\in M} E_s^*(x),\ \Gamma_-=\bigcup_{x\in M} E_u^*(x),\ \text{and}\ \Gamma=M\times\{0\}.$$
We have then
\begin{lemm}\label{l:trappedset} Suppose that $(f,g)$ is a smooth Morse pair which generates a Morse-Smale gradient flow $\varphi_f^t$. One has, for every 
$(x,\xi)$ in $T^*M$ with $\xi\neq 0$,
$$ (x,\xi)\in \Gamma_{\pm}\ \Longrightarrow\  \lim_{t\rightarrow\pm\infty}\|\xi(t)\|_{x(t)}=0,$$
 and
 $$(x,\xi)\notin\Gamma_{\pm}\ \Longrightarrow\  \lim_{t\rightarrow\pm\infty}\|\xi(t)\|_{x(t)}=+\infty,$$
 where $(x(t),\xi(t))=\Phi_f^t(x,\xi)$.
\end{lemm}
This Lemma tells us that the trapped set of the Hamiltonian flow is reduced to the zero section of $T^*M$. The proof of this 
Lemma will crucially use the fact that we made the Morse-Smale assumption and that we have a smooth (at least $\ml{C}^1$) change of coordinates 
which linearizes the vector field.

\begin{proof} We only consider the case where $t\rightarrow +\infty$ (the other case can be obtained by replacing $f$ by $-f$). Let $(x,\xi)$ be an element 
in $T^*M$ with $\xi\neq 0$. There exists a critical point $x_+$ of $f$ such that $\lim_{t\rightarrow+\infty}\varphi_f^t(x)=x_+$. In particular, for 
$t>0$ large enough, $\varphi_f^t(x)$ belongs to the adapted chart around $x_+$ which was defined in paragraph~\ref{sss:adapted}. Up to a 
translation of time, one can write that, in this system of adapted coordinates and for every $t\geq 0$ 
$$\Phi_f^t(x,\xi)=(e^{-|\chi_1|t}x_1,\ldots, e^{-|\chi_{r}|t}x_r,0,\ldots, 0,e^{|\chi_1|t}\xi_1,\ldots, e^{|\chi_{r}|t}\xi_r,
e^{-|\chi_{r+1}|t}\eta_{r+1},\ldots, e^{-|\chi_{n}|t}\eta_n).$$
Hence, as all the norms can be made uniformly equivalent to the Euclidean norm in a small neighborhood of $x_+$, one can find two positive constants 
$0<C_1<C_2$ such that
$$C_1 \left(\sum_{j=1}^re^{2|\chi_j|t}\xi_j^2+ \sum_{j=r+1}^ne^{-2|\chi_j|t}\eta_j^2\right)\leq\|\xi(t)\|_{x(t)}^2\leq 
C_2 \left(\sum_{j=1}^re^{2|\chi_j|t}\xi_j^2+ \sum_{j=r+1}^ne^{-2|\chi_j|t}\eta_j^2\right).$$
The fact that $(x,\xi)$ belongs to $\Gamma_+$ is exactly equivalent to the fact that $\xi_1=\ldots =\xi_r=0$ from which one can easily conclude the 
expected property.
\end{proof}

Introduce now the two following disjoint subsets of $S^*M$:
$$\Sigma_u:=S^*M\cap\Gamma_+,\ \text{and}\ \Sigma_s:=S^*M\cap\Gamma_-.$$
Then, one has:
\begin{lemm}\label{l:attract}Suppose that $(f,g)$ is a smooth Morse pair which generates a Morse-Smale gradient flow $\varphi_f^t$. One has
\begin{equation}\label{e:repeller}
 \forall (x,\xi)\in S^*M-\Sigma_s,\ \lim_{t\rightarrow-\infty} d_{S^*M}\left(\tilde{\Phi}_f^t(x,\xi),\Sigma_u\right)=0,
\end{equation}
and
\begin{equation}\label{e:attractor}
 \forall (x,\xi)\in S^*M-\Sigma_u,\ \lim_{t\rightarrow+\infty} d_{S^*M}\left(\tilde{\Phi}_f^t(x,\xi),\Sigma_s\right)=0.
\end{equation}
\end{lemm}
This lemma tells us that $\Sigma_u$ and $\Sigma_s$ are in a certain weak sense repeller and attractor of the flow $\tilde{\Phi}_f^t$. A 
stronger version of this fact will be given in Lemma~\ref{l:openneighbor}.

\begin{proof} We proceed as in the proof of Lemma~\ref{l:trappedset}, and we just treat the case where $t\rightarrow+\infty$. Let $(x,\xi)$ be an element 
in $S^*M-\Sigma_u$. In other words, $(x,\xi)$ does not belong $\Gamma_+$. Using the notations of the proof of Lemma~\ref{l:trappedset}, it means that
$$\Phi_f^t(x,\xi)=(e^{-|\chi_1|t}x_1,\ldots, e^{-|\chi_{r}|t}x_r,0,\ldots, 0,e^{|\chi_1|t}\xi_1,\ldots, e^{|\chi_{r}|t}\xi_r,
e^{-|\chi_{r+1}|t}\eta_{r+1},\ldots, e^{-|\chi_{n}|t}\eta_n),$$
 with $\xi_j\neq 0$ for some $1\leq j\leq r$. By letting $t\rightarrow+\infty$, we find that that any accumulation point of $\tilde{\Phi}_f^t(x,\xi)$ is of 
 the form $(0,\ldots, 0,\tilde{\xi}_1,\ldots,\tilde{\xi}_r,0,\ldots,0)$. Equivalently, every accumulation point belongs to $\Sigma_s$.
\end{proof}

\subsection{Compactness}

In order to make the machinery of anisotropic Sobolev space work, it will first be important for us that $\Sigma_u$ and $\Sigma_s$ are compact subsets of $S^*M$. 
This assumption is verified for gradient flow satisfying the Morse-Smale assumption: 
\begin{lemm}\label{l:compact} Suppose that $(f,g)$ is a smooth Morse pair which induces a gradient flow with the Morse-Smale property. 
Then, the subsets $\Sigma_u$ and $\Sigma_s$ are compact in $S^*M$.
\end{lemm}

This property combined with Lemma~\ref{l:openneighbor} will be crucial in our construction of anisotropic Sobolev spaces. In particular, 
they are necessary to prove Lemma~2.1 from~\cite{FaSj11} which is at the heart of this construction. We note that the proof of Lemma~\ref{l:compact} 
requires both the Morse-Smale assumption and the existence of a (at least $\ml{C}^1$) linearizing chart for the flow.

\begin{rema}\label{r:weber} Even if this Lemma sounds natural, the proof is a little bit subtle and it is related to the so-called Whitney regularity 
condition~\cite[Ch.~7]{Ni10} -- see also the appendix of~\cite{Lau12} for related results in the case of locally flat metrics. Here, we are aiming 
at weaker results than in these references and we shall give a proof of our Lemma which is based on purely ``dynamical arguments''. 
Our argument is in fact very close to the proof of 
the compactness of the space of connecting orbits of Weber in~\cite[Th.~3.8]{Web06} -- see also~\cite{Sm60} for earlier related results of Smale. 
In this reference, it was proved that the space of connecting orbits between two critical points $a$ and $b$ is ``compact up to broken orbits''. 
It means that, for a fixed sequence $(x_m)_{m\geq 1}$ in $W^u(a)\cap W^s(b)$, there exists (up to extraction) a sequence of critical points 
$a=b_l,b_{l-1},\ldots, b_1=b$ and a finite sequence of points $z^{(p)}$ in $W^{u}(b_{p+1})\cap W^s(b_{p})$ such 
that
$$\forall\epsilon_1>0,\exists m_0,\forall m\geq m_0,\ d\left(\ml{O}(x_m),\cup_{1\leq p\leq l-1}\ml{O}(z^{(p)})\right)<\epsilon_1,$$
where $\ml{O}(x)$ denotes the orbit of $x$ under the flow $\varphi_f^t$. The key ``dynamical argument'' in the proof of Weber was to use the Grobman-Hartman 
linearization Theorem around the critical points of $f$. Here, we want to prove something slightly stronger in the sense that we will have to keep track 
of the behaviour of the cotangent vectors in the phase space $S^*M$ and not only of the points in the position space $M$. For that purpose, we will crucially 
make use of the fact that we have a smooth (at least $\ml{C}^1$) chart where the vector field can be linearized. In some sense, our compactness statement 
on $S^*M$ requires the 
Sternberg-Chen's Theorem while Weber's compactness statement on $M$ only required the Grobman-Hartman's Theorem.
 
\end{rema}

\begin{proof} We only treat the case of $\Sigma_s$ as the case of $\Sigma_u$ can be obtained by replacing $f$ by $-f$. In order to prove compactness, we 
will just prove that $\Sigma_s$ is closed (as $S^*M$ is compact). Before starting the proof, we note that, for a given critical point $b$ of $f$, the 
conormal to the unstable manifold $W^u(b)$ can be written in local coordinates as
\begin{equation}\label{e:coord-unstable-conormal}\left\{(z,\zeta)=(0,y,\xi,0):\ y\in\IR^{n-r(b)},\ \xi\in\IR^{r(b)}-\{0\}\right\}.\end{equation}
We will now show that $\Sigma_s$ is closed via a contradiction argument. We fix $(z_m,\zeta_m)$ a sequence in $\Sigma_s$ which converges to $(z_{\infty},\zeta_{\infty})\in S^*M$ and we assume that the limit $(z_{\infty},\zeta_{\infty})$
does not belong to $\Sigma_s$. We know that there exists $b_1$ in $\text{Crit}(f)$ such that $z_{\infty}$ belongs 
to $W^u(b_1)$. Without loss of generality, by extracting a subsequence, we can suppose that there exists a single point $a$ in $\text{Crit}(f)$ such that $z_m$ belongs to $W^u(a)$ 
for every $m\geq 1$.

Let us first suppose that $a=b_1$. In that case, for $T>0$ large enough, $\tilde{\Phi}^{-T}(z_{\infty},\zeta_{\infty})$ 
will belong to the linearizing chart near $a$. Hence, for $m\geq 1$ large enough, $\tilde{\Phi}^{-T}(z_{m},\zeta_{m})$ also belongs to this 
linearizing chart by continuity of $\tilde{\Phi}^{-T}$. Up to applying the flow in backward time, we thus know from~\eqref{e:coord-unstable-conormal} that $(z_m,\zeta_m)$ and $(z_{\infty},\zeta_{\infty})$ are of the form
$$(z_m,\zeta_m)=(0,y_m,\xi_m,0)\quad\text{and}\quad(z_{\infty},\zeta_{\infty})=(0,y_{\infty},\xi_{\infty},\eta_{\infty}),$$
in the local coordinates near $a$. As we supposed that the limit point does not belong to $\Sigma_s$, we note that 
$\eta_{\infty}\neq 0$ and it gives us the expected contradiction as $(z_m,\zeta_m)\rightarrow(z_{\infty},\zeta_{\infty})$ as $m$ tends to $+\infty$.

Suppose now that $a\neq b_1$ and let us explain how we can get a contradiction too.
Note that this implies $r(b_1)>0$ as $S^*V_{b_1}\cap\Sigma_s=\emptyset$ when $r(b_1)=0$. 
For that purpose, we will verify that we can construct a new sequence 
$((z_m^{(1)},\zeta_m^{(1)}))_{m\geq 1}$ in $\Sigma_s$ 
such that the following holds~:
\begin{itemize}
 \item for every $m\geq 1$, $z_m^{(1)}$ belongs to $W^u(a)$,
 \item $(z_m^{(1)},\zeta_m^{(1)})\rightarrow (z_{\infty}^{(1)},\zeta_{\infty}^{(1)})$ as $m\rightarrow+\infty$,
 \item $z_{\infty}^{(1)}\in W^s(b_1)-\{b_1\}$ and $(z_{\infty}^{(1)},\zeta_{\infty}^{(1)})\notin\Sigma_s$.
\end{itemize}
Then, we know that $z_{\infty}^{(1)}\in W^u(b_2)$ for some critical point $b_2$ 
verifying $r(b_2)<r(b_1)$ -- see Lemma~\ref{l:decrease-dimension}. If $b_2=a$, we are in the first situation for the
sequence $(z_m^{(1)},\zeta_m^{(1)})_{m\in \mathbb{N}}$ and we get the expected contradiction. 
If not, we can reproduce the same construction. This yields a new sequence $(z_m^{(2)},\zeta_m^{(2)})\in\Sigma_s$ converging to some point 
$(z_{\infty}^{(2)},\zeta_{\infty}^{(2)})\notin\Sigma_s$ and such that $z_m^{(2)}$ belongs to $W^u(a)$ for every $m\geq 1$. 
Again, we can ensure that $z_{\infty}^{(2)}\in W^u(b_3)$ with $r(b_3)<r(b_2)$. In the end, this 
gives a sequence of critical points $b_1,b_2,\ldots, b_l$ with $r(b_{p+1})<r(b_{p})$ for every $i\geq 1$. As $r(b_p)\geq 0$ for every $p\geq 1$, there will necessarily 
be some $l\geq 1$ such that $b_l=a$ and in that case, we already saw how to get the contradiction.

Hence, all that remains to be proved is the existence of a new sequence $((z_m^{(1)},\zeta_m^{(1)}))_{m\geq 1}$ with the above requirements when $a\neq b_1$. As above, we 
can suppose without loss of generality that both $(z_m,\zeta_m)$ and $(z_{\infty},\zeta_{\infty})$ belong to the linearizing chart near $b_1$. In the local system of coordinates near $b_1$, 
these two points read
$$(z_m,\zeta_m)=(x_m,y_m,\xi_m,\eta_m)\quad\text{and}\quad(z_{\infty},\zeta_{\infty})=(0,y_{\infty},\xi_{\infty},\eta_{\infty}),$$
with $\|x_m\|\rightarrow 0$ as $m\rightarrow+\infty$ as $(z_m,\zeta_m)$ converges to $(z_{\infty},\zeta_{\infty})$. As the limit point does not belong to $\Sigma_s$, we also know that 
$\eta_{\infty}\neq 0$. Hence, there exists some $\delta_1>0$ such that, for $m$ large enough, $\|\eta_m\|\geq\delta_1$. Let us now apply the backward flow to $(z_m,\zeta_m)$, i.e.
$$\Phi^{-T}(z_m,\zeta_m):=(z_m(-T),\zeta_m(-T)),$$
where
$$z_m(-T)=(e^{|\chi_1|T}x_{m,1},\ldots,e^{|\chi_r|T}x_{m,r},e^{-|\chi_{r+1}|T}y_{m,r+1}
,\ldots, e^{-|\chi_n|T}y_{m,n})=(x_m(-T),y_m(-T)),$$
and
$$\zeta_m(-T)=(e^{-|\chi_1|T}\xi_{m,1},\ldots,e^{-|\chi_r|T}\xi_{m,r},e^{|\chi_{r+1}|T}\eta_{m,r+1}
,\ldots, e^{|\chi_n|T}\eta_{m,n})=(\xi_m(-T),\eta_m(-T)).$$
Note that this expression is only valid when $z_m(-T)$ belongs to the linearizing chart near $b_1$. For every $m$ large enough, we now pick $T_m$ large enough to ensure that there exists $1\leq j\leq r$ such 
that $0<\delta_2<|e^{|\chi_j|T_m}x_{m,j}|<2\delta_2$ 
for some fixed $\delta_2>0$ smaller than the size of the linearizing chart. Note that, as 
$\|x_m\|\rightarrow0$, $T_m\sim|\log\|x_m\||$ tends to $+\infty$ as $m\rightarrow+\infty$. We now set
$$(z_m^{(1)},\zeta_m^{(1)}):=\tilde{\Phi}^{-T_m}(z_m,\zeta_m)=\left(z_m(-T_m),\frac{\zeta_m(-T_m)}{\|\zeta_m(-T_m)\|}\right)\in\Sigma_s,$$
and we will verify that, up to extraction, it has the expected properties. First of all, up to extraction, we can suppose that the sequence converges to a limit point 
$(z_{\infty}^{(1)},\zeta_{\infty}^{(1)})$ belonging to $S^*M$ which is our second requirement. 
By construction, the points $z_m^{(1)}=\varphi^{-T_m}(z_m)$ belong to $W^u(a)$ which is our first 
requirement. It remains to check the last properties. From the expression of $z_m(-T_m)$ in local coordinates and as $T_m\rightarrow+\infty$, we can verify that 
$z_{\infty}^{(1)}$ is of the form $(x_{\infty},0)$ with $\|x_{\infty}^{(1)}\|\gtrsim \delta_2.$ Hence, $z_{\infty}^{(1)}\in W^s(b_1)-\{b_1\}$ as expected. It remains to consider 
the cotangent component. As $\|\eta_m\|\geq\delta_1$ and as $T_m\rightarrow+\infty$, we know that $\|\eta_m(-T_m)\|\rightarrow +\infty$ and that $\|\xi_m(-T_m)\|\rightarrow 0$ as 
$m$ tends to $+\infty$. This implies that $\zeta_{\infty}^{(1)}$ is of the form $(0,\eta_{\infty}^{(1)})\neq 0$. Hence, from~\eqref{e:coord-unstable-conormal}, $(z_{\infty}^{(1)},\zeta_{\infty}^{(1)})$ belongs 
to the conormal of $W^s(b_1)$ which is a subset of $\Sigma_u$. From the transversality assumption, it cannot belong to $\Sigma_s$ which was our last requirement on the limit point. 
This concludes the proof of the Lemma.

\end{proof}

\subsection{Invariant neighborhoods}

We conclude this dynamical section with the following Lemma which states that $\Sigma_u$ and $\Sigma_s$ are repeller and attractor in a slightly 
stronger sense than in Lemma~\ref{l:attract}.

\begin{lemm}\label{l:openneighbor} Let $(f,g)$ be a smooth Morse pair that induces a Morse-Smale gradient flow. Let $\eps>0$. Then, there exists 
an open neighborhood $V^s$ of $\Sigma_s$ which is of size $\leq\eps$ and such that
$$\forall t\geq 0,\ \tilde{\Phi}_f^t(V^s)\subset V^s.$$
The same property holds true for $\Sigma_u$ in backward times. 
\end{lemm}

One more time, the proof makes use of the existence of a (at least $\ml{C}^1$) linearizing chart for the flow. We 
also make use of the fact that the critical values of $f$ are distinct.

\begin{proof} Again, we only treat the case of $\Sigma_s$ (the case of $\Sigma_u$ can be obtained by replacing $f$ by $-f$). Recall that
we assumed our Morse function $f$ to be excellent, meaning that all its critical values are distinct. Thus, we can define the following 
\textbf{total order} relation between critical points. We say that $a< b$ if $f(a)<f(b)$. This relation allows to order the critical points
as $a_1< a_2<\dots <a_K$.

The proof of this lemma requires one more time a delicate analysis of the flow. We construct the neighborhood in a progressive manner. First, we 
build a small neighborhood of the projection of $\Sigma_s$ on $M$ which is equal to $\cup_{1\leq j\leq K:r(a_j)\neq 0}W^u(a_j)$. Then, we adjust the construction to be able to lift this open set into a small open 
neighborhood of $\Sigma_s$ inside $S^*M$. 
In order to construct the neighborhood in $M$, 
we fix, for every $1\leq j\leq K$ for which $a_j$ has positive Morse index (i.e. $r(a_j)>0$), 
the following open neighborhood of
$a_j$ inside $M$:
$$R(a_j,\eps_j,\eps_j'):=\{(x,y): \forall j,\ |x_j|<\eps_j',\ |y_j|<\eps_j\},$$
where $\eps_j,\eps_j'>0$ are small enough parameters to ensure that this defines an $\eps$ neighborhood of $a_j$. We will now adjust the values of these 
parameters to construct a neighborhood of the projection of $\Sigma_s$ which is invariant by $\varphi_f^t$ for every $t\geq 0$. 
For that purpose, we proceed 
by induction starting from the largest values of $f$.

First, we observe that every point whose trajectory enters $R(a_j,\eps_j,\eps_j')$ will either stay in this open set for every 
$t\geq 0$, or escape this open set (in a maybe arbitrarly large time) by crossing the following subset of $M$:
$$F(a_j,\eps_j,\eps_j'):=\{(x,y): \forall j,\ |x_j|<\eps_j',\ \exists j,\ |y_j|=\eps_j\}$$
which is one of the face of the boundary of $R(a_j,\eps_j,\eps_j')$.
We will inductively construct from the maximum of $f$ a system of open neighborhoods
$R(a_j,\eps_j,\eps_j')_j,r(a_j)>0$
such that for every face $F(a_j,\eps_j,\eps_j')$ of $R(a_j,\eps_j,\eps_j')$, 
there exists a finite time $T_j>0$, such that for every $x\in F(a_j,\eps_j,\eps_j')$, the trajectory 
$t\mapsto\varphi_f^t(x)$ meets $\cup_{j<i}R(a_i,\eps_i,\eps_i')$ for some $t\in (0,T_j)$.
For $j=K$, one can verify that the neighborhood is invariant by the flow in positive time 
provided that we pick $\eps_K=\eps_K'>0$ small enough to ensure that we are in the neighborhood of adapted coordinates 
defined in paragraph~\ref{sss:adapted}. 
Suppose now that we have fixed the values of $\eps_i$ and $\eps_i'$ for every $i>j$ with $r(a_i)\neq 0$ and that $r(a_j)\neq 0$. 
We will explain how to fix the value of 
$\eps_j$ and $\eps_j'$. We claim that the forward trajectory of every point inside $F(a_j,\eps_j,\eps_j')$ will reach 
$$\bigcup_{i>j:r(a_i)\neq 0}R(a_i,\eps_i,\eps_i')$$ in a finite 
time $0<t<T_j$ where $T_j$ depends only on $\eps_i,\eps_i'$ with $i>j$ satisfying $r(a_i)\neq 0$ and on $\eps_j$. In particular, this time can be made 
uniform in $\eps_j'$.
Assume by contradiction that, for every $m>0$ and for every $T>0$, there exists $x_{m,T}$ in $F(a_j,\eps_j,1/m)$ such that the orbit 
$t\in [0,T]\mapsto \varphi_f^t(x_{m,T})$ does not meet the subset 
$\bigcup_{i>j:r(a_i)\neq 0}R(a_i,\eps_i,\eps_i')$. We fix $T>0$, and, by compactness, one can extract a subsequence such that 
$x_{m,T}\rightarrow x_{\infty,T}$ as $m\rightarrow+\infty$ where $x_{\infty,T}$ belongs to $W^{u}(a_j)$ is at distance $>\ml{O}(\eps_j)$ of 
$a_j$. We now extract another subsequence (as $T\rightarrow+\infty$) and we obtain a point $x_{\infty}\neq a_j$ in $W^u(a_j)$ that would not 
reach $\bigcup_{i>j:r(a_i)\neq 0}R(a_i,\eps_i,\eps_i')$ in finite time. This contradicts the fact that $\lim_{t\rightarrow+\infty}
\varphi_f^{t}(x_{\infty})$ is equal to $a_i$ for some $i>j$ satisfying $r(a_i)\neq 0$.

Recall now that the distance between two 
trajectories can grow at most exponentially under the flow~\cite[Lemma~11.11]{Zw12}. 
Hence, if we choose $\eps_j'>0$ small enough, we can ensure that, the forward trajectory of every 
point inside $F(a_j,\eps_j,\eps_j')$ will remain $\eps$ close to $W^u(a_j)$ up the finite time $t\leq T_j$ where it will enter one 
of the neighborhood $R(a_i,\eps_i,\eps^\prime_i)$ with $i>j$ and $r(a_i)>0$.
This construction defines a family of open neighborhood of the critical points $a_j$ of index $>0$ whose forward trajectory under the flow 
will remain within a distance $\eps>0$ of $\cup_{1\leq j\leq K:r(a_j)\neq 0}W^u(a_j)$ which is exactly the projection of $\Sigma_s$ on $M$. 
Then, we set
$$\ml{N}:=\bigcup_{t\geq 0}\bigcup_{1\leq j\leq K:r(a_j)\neq 0}\varphi_f^t(R(a_j,\eps_j,\eps_j')).$$
By construction, this set is invariant by $\varphi_f^t$. Moreover, it defines a neighborhood of $\cup_{1\leq j\leq K:r(a_j)\neq 0}W^u(a_j)$ which is 
of size $\leq\eps$.

It now remains to verify that we can lift this neighborhood into a neighborhood of size $\eps$ of $\Sigma_s$. 
For that purpose, we rely on the fact that, our 
smooth system of coordinate chart allows to linearize also the Hamiltonian flow $\Phi_f^t$. Hence, we fix another positive 
parameter $\eps_j''>0$ and we consider above each neighborhood $R(a_j,\eps_j,\eps_j')$ an open neighborhood $\tilde{R}(a_j,\eps_j,\eps_j',\eps_j'')$ 
in $S^*M$ made of unit covectors which 
are within a distance $<\eps_j''$ of $\xi_1=\ldots=\xi_r=0$. For every fixed choice of $\eps_j>0$ and $\eps_j''>0$, we can use the 
compactness of $\Sigma_s$ to fix $\eps_j'>0$ small enough to ensure that this defines indeed a neighborhood of size $<\eps$ of 
$\Sigma_s\cap S^*R(a_j,\eps_j,\eps_j')$. Using the fact that the distance between two trajectories can grow at most exponentially under the 
flow $\tilde{\Phi}_f^t$, we can argue by induction as in the case of $M$. More precisely, at each step of the induction, 
we can fix $\eps_j''>0$ small enough in a way that depends only on the values of $\eps_j$ and of $\eps_i^{(*)}$ with $i>j$ and $r(a_i)>0$ and such that
$$\tilde{\ml{N}}:=\bigcup_{t\geq 0}\bigcup_{1\leq j\leq K:r(a_j)\neq 0}\tilde{\Phi}_f^t(\tilde{R}(a_j,\eps_j,\eps_j',\eps_j''))
$$
defines a forward invariant open neighborhood of $\Sigma_s$ of size $<\eps$.
\end{proof}

\section{Spectral properties of the transfer operator acting on currents}\label{s:anisotropic}

This section is organized as follows. First, we state the existence of a nice escape function enjoying the dynamical features of~\cite{FaSj11, DyZw13}. 
This allows us to define some Sobolev spaces of anisotropic currents following these references. Finally, we recall the 
spectral properties of $-\ml{L}_{V_f}^{(k)}$ acting on 
these spaces. The main difference with the above references is the construction of the escape function which 
requires modifications compared with the setting from~\cite[Lemma~2.1]{FaSj11} where the authors made use of the 
Anosov property. Lemmas~\ref{l:compact} and~\ref{l:openneighbor} will in fact ensure that the construction of Faure and Sj\"ostrand can be 
extended to our framework. From this point on, we will always assume that $(f,g)$ is a \textbf{smooth Morse pair generating a Morse-Smale gradient flow}.

\subsection{Construction of anisotropic Sobolev spaces}
\label{s:anisotropicsobo}
\subsubsection{Escape function} The key ingredient in the construction of~\cite{FaSj11} is the following Lemma which will allow us to define 
appropriate Sobolev spaces where the operator $\ml{L}_{V_f}$ has nice spectral properties.

\begin{lemm}[Escape function]\label{l:escape-function} Let $N_0,N_1>4\|f\|_{\ml{C}^0}$ be two elements in $\IR$. Then, there exist 
$c_0>0$ (depending on $(M,g)$ but not on $N_0$ and $N_1$) and a smooth function $m(x,\xi)$ in $\ml{C}^{\infty}(T^*M)$ with bounded derivatives and which 
\begin{itemize}
 \item takes values in $[-2N_0,2N_1]$,
 \item is $0$ homogeneous for $\|\xi\|_x\geq 1$,
 \item is $\leq -\frac{N_0}{2}$ on a conic neighborhood of $\Gamma_-$ (for $\|\xi\|_x\geq 1$),
 \item is $\geq \frac{N_1}{2}$ on a conic neighborhood of $\Gamma_+$ (for $\|\xi\|_x\geq 1$),
 \item is $\geq \frac{N_1}{4}-2N_0$ outside a conic neighborhood of $\Gamma_-$ (for $\|\xi\|_x\geq 1$),
\end{itemize}
and such that there exists $R_0>0$ for which the escape function
$$G_{m}(x,\xi):=m(x,\xi)\log(1+\|\xi\|_x^2)$$
verifies, for every $(x,\xi)$ in $T^*M$ with $\|\xi\|_x\geq R_0$,
$$X_{H_f}.(G_m)(x,\xi)\leq -C_N:=-c_0\min\{N_0,N_1\}.$$
\end{lemm}

Now that we have settled the dynamical framework precisely in section~\ref{s:dynamics}, the construction of the function $G_m$ closely follows the one from~\cite{FaSj11}. 
For the sake of exposition, we postpone the detailed proof of this result to appendix~\ref{a:escape}, and we just mention the key ingredients: (1) $f$ is strictly decreasing along the flow, 
(2) there exists a $\ml{C}^1$ chart of adapted coordinates (see paragraph~\ref{sss:adapted}), (3) the attractor and 
repeller of the Hamiltonian flow (Lemmas~\ref{l:compact} and~\ref{l:openneighbor}) are compact. Lemma~\ref{l:escape-function} is in fact the only step in the construction of the anisotropic Sobolev space 
where one uses the dynamical properties of the flow under consideration.

\subsubsection{Anisotropic Sobolev spaces}
Let us now define the corresponding anisotropic Sobolev spaces. We fix $N_0,N_1>4\|f\|_{\ml{C}^0}$ large and we set
\begin{equation}\label{A_mweightsobo}
A_m(x,\xi):=\exp G_m(x,\xi),
\end{equation}
where $G_m(x,\xi)$ is given by Lemma~\ref{l:escape-function}. Following paragraph~1.1.2 in~\cite{FaSj11}, one can 
define the following anisotropic Sobolev 
space
$$H^{m}(M):=\Op(A_m)^{-1}(L^2(M)),$$
where $\Op(A_m)$ is an essentially selfadjoint pseudodifferential operator\footnote{Note that this requires to deal with symbols of variable orders 
whose symbolic calculus was described in Appendix~A of~\cite{FaRoSj08}. This can be done as the symbol $m(x,\xi)$ 
belongs to the standard class of symbols $S^0(T^*M)$.} with principal symbol $A_m$. 

We now briefly collect some facts concerning these spaces and we refer to~\cite[Sect.~3.2]{FaRoSj08} for more properties of these spaces. The 
space $H^m(M)$ is endowed with a Hilbert structure induced by the Hilbert structure on $L^2(M)$. The space $$H^{-m}(M)=\Op(A_m)L^2(M)$$ 
is the topological dual of $H^m(M)$. The anisotropic Sobolev space $H^m(M)$ is a reflexive space. Finally, one has
$$C^\infty(M) \subset H^m(M) \subset \mathcal{D}^\prime(M),$$ 
where the injections are continuous.

\subsubsection{Anisotropic Sobolev spaces of currents}

Let $0\leq k\leq n$. We consider the vector bundle $\Lambda^k(T^*M)$ of exterior $k$ forms. 
We define $\mathbf{A}_m^{(k)}(x,\xi):=A_m(x,\xi)\textbf{Id}$ belonging to $\text{Hom}(\Lambda^k(T^*M))$. 
We fix the inner product $\la,\ra_{g^*}^{(k)}$ on $\Lambda^k(T^*M)$ which is induced by the metric $g$ on $M$.

This allows to define the Hilbert space $L^2(M,\Lambda^k(T^*M))$ and to introduce an 
anisotropic Sobolev space of currents by setting
$$\ml{H}^m_k(M)=\Op(\mathbf{A}_m^{(k)})^{-1}L^2(M,\Lambda^k(T^*M)),$$
where $\Op(\mathbf{A}_m^{(k)})$ is a pseudodifferential operator with principal symbol $\mathbf{A}_m^{(k)}$. We refer to~\cite[App.~C.1]{DyZw13} for a brief reminder of pseudodifferential operators 
with values in vector bundles. In particular, adapting the proof of~\cite[Cor.~4]{FaRoSj08} to the vector bundle valued framework, one can verify that $\mathbf{A}_m^{(k)}$ is an elliptic symbol, and 
thus $\Op(\mathbf{A}_m^{(k)})$ can be chosen to be invertible on $\Omega^k(M)$. Mimicking the proofs of~\cite{FaRoSj08}, we can deduce some properties of these spaces of currents. First of all, they are endowed with a Hilbert 
structure inherited from the $L^2$-structure on $M$. The space $$\ml{H}^{m}_k(M)^{\prime}=\Op(\mathbf{A}_m^{(k)})L^2(M,\Lambda^k(T^*M))$$
is the topological dual of $\ml{H}^m_k(M)$ which is in fact reflexive. We also note that the space $\ml{H}_k^{m}(M)$ can be identified with $H^m(M)\otimes_{\ml{C}^{\infty}(M)}\Omega^k(M)$. Finally, one has
$$\Omega^k(M) \subset \ml{H}_k^{m}(M) \subset \mathcal{D}^{\prime,k}(M),$$ 
where the injections are continuous.

\subsection{Identifying the dual} \label{l:hodegstar}
Recall that the Hodge star operator is the unique isomorphism 
$\star_k :\Lambda^k(T^*M)\rightarrow \Lambda^{n-k}(T^*M)$ such that, for every $\psi_1$ in $\Omega^k(M)$ and $\psi_2$ in $\Omega^{n-k}(M)$, 
$$\int_M\psi_1\wedge\psi_2=\int_M\la \psi_1,\star_k^{-1}\psi_2\ra_{g^*(x)}^{(k)}\omega_g(x),$$
where $\la .,.\ra_{g^*(x)}^{(k)}$ is the induced Riemannian metric on $\Lambda^k(T^*M)$. In particular, 
$\star_k$ induces an isomorphism from $\ml{H}^{m}_k(M)^{\prime}$ to $\ml{H}^{-m}_{n-k}(M)$, 
whose Hilbert structure is given by the scalar product
$$(\psi_1,\psi_2)\in\ml{H}^{-m}_{n-k}(M)^2\mapsto \la \star_k^{-1}\psi_1,\star_k^{-1}\psi_2\ra_{\ml{H}_k^m(M)'}.$$
Thus, the topological dual of $\ml{H}_k^m(M)$ can be identified with $\ml{H}^{-m}_{n-k}(M)$, where, for every $\psi_1$ in $\Omega^k(M)$ 
and $\psi_2$ in $\Omega^{n-k}(M)$, one has the following duality relation:
$$\la\psi_1,\psi_2\ra_{\ml{H}_k^m\times\ml{H}_{n-k}^{-m}}=\int_{M}\psi_1\wedge\psi_2
=\la \Op(\mathbf{A}_m^{(k)})\psi_1,\Op(\mathbf{A}_m^{(k)})^{-1}\star_k^{-1}\overline{\psi_2}\ra_{L^2}=\la\psi_1,\star_k^{-1}\psi_2
\ra_{\ml{H}_k^m\times(\ml{H}_{k}^m)^{\prime}}.$$

\subsection{Discrete spectrum}

The main result on the spectral properties of $-\ml{L}_{V_f}^{(k)}$ acting on these anisotropic spaces is the following Proposition:
\begin{prop}[Discrete spectrum]\label{p:eigenvalues} The operator
$-\ml{L}_{V_f}^{(k)}$ defines a maximal closed unbounded operator on $\ml{H}^m_k(M)$, 
$$-\ml{L}_{V_f}^{(k)}:\ml{H}^m_k(M)\rightarrow\ml{H}^m_k(M),$$
with domain given by $\ml{D}(-\ml{L}_{V_f}^{(k)}):=\{u\in\ml{H}^m_k(M):-\ml{L}_{V_f}^{(k)} u\in\ml{H}^m_k(M)\}.$ It coincides with the 
closure of $-\ml{L}_{V_f}^{(k)}:\Omega^k(M)\rightarrow\Omega^k(M)$ in the graph norm for operators. Moreover, there exists a constant 
$C_0$ in $\IR$ (that depends on the choice of the order function 
$m(x,\xi)$) such that $-\ml{L}_{V_f}^{(k)}$ has empty spectrum for $\operatorname{Re}(z)>C_0$. Finally,
the operator
 $$-\ml{L}_{V_f}^{(k)}:\ml{H}_k^m(M)\rightarrow \ml{H}_k^m(M),$$
 has a discrete spectrum with finite multiplicity in the domain
 $$\operatorname{Re} (z)>-C_N+C,$$
 where $C>0$ depends only the choice of the metric (and the local coordinate charts) and $C_N>0$ is the constant from Lemma~\ref{l:escape-function}.
\end{prop}
Up to some adaptations to deal with the case of currents, we can in fact follow the proof of~\cite{FaSj11}
which only requires the existence of an escape function as in Lemma~\ref{l:escape-function} -- 
see appendix~\ref{a:discrete} for a brief account on the proof of Faure and Sj\"ostrand.

We now list some properties of this spectrum. As in~\cite[Th.~1.5]{FaSj11}, we can show that the eigenvalues (counted with their algebraic multiplicity) 
and the eigenspaces 
of $-\ml{L}_{V_f}^{(k)}:\ml{H}_k^m(M)\rightarrow \ml{H}_k^m(M)$ are in fact independent of the choice of escape function. 
For every $0\leq k\leq n$, we call the eigenvalues the \textbf{Pollicott-Ruelle resonances of index $k$}. For later use, we will write
$$\ml{R}_k(f,g):=\left\{\text{Pollicott-Ruelle resonances of index}\ k\right\}\subset\IC.$$ 
In other words, these 
complex numbers are the poles of the meromorphic extension of the resolvent
$$\left(-\ml{L}_{V_f}^{(k)}-z\right)^{-1}:\Omega^k(M)\rightarrow\ml{D}^{\prime,k}(M).$$
Finally, we note that, by duality, the same spectral properties holds for the dual operator
\begin{equation}\label{e:dual}(-\ml{L}_{V_f}^{(k)})^{\dagger}=-\ml{L}_{V_{-f}}^{(n-k)}:\ml{H}^{-m}_{n-k}(M)\rightarrow\ml{H}^{-m}_{n-k}(M).\end{equation}

\section{Upper bound on the set of eigenvalues}\label{s:spectrum}

We aim at describing completely the eigenvalues and the eigenmodes in great detail. 
To begin with, we shall first show an upper bound on the spectrum~:
\begin{prop}\label{p:ruelle-spectrum} Suppose that $(f,g)$ is a smooth Morse pair inducing a Morse-Smale gradient flow. Then, one has, for every $0\leq k\leq n$,
$$\ml{R}_k(f,g)\subset\ml{I}_k:=\bigcup_{a\in\operatorname{Crit}(f)}\ml{I}_k(a),$$
where, for every $a$ in $\operatorname{Crit}(f)$ of index $0\leq r\leq n$,
\begin{equation}\label{e:spectrum-indexk}\ml{I}_k(a):=\bigcup_{(*)}\left\{-\sum_{j=1}^n\alpha_j|\chi_j(a)|:\ \forall j\in I\cup J,\ \alpha_j\geq 1\right\}.\end{equation}
where $(*)$ means $\alpha\in\IZ_+^n$ and\footnote{Note that $I$ and $J$ may be empty.} $I\subset\{1,\ldots, r\},\ J\subset\{r+1,\ldots,n\}$ with $|J|-|I|=k-r$.
\end{prop}
Building on the construction of section~\ref{s:anisotropic}, we shall now give the proof of this Proposition by making use of 
our a priori information on the regularity of the resonant states. 
Among other things, we will use the following property of transport equations:
\begin{lemm}[Propagation Lemma]
Let $0\leq k\leq n$, let $z$ in $\IC$ and let $u\in \mathcal{D}^{\prime,k}(M)$ be a solution of $\ml{L}_{V_f}^{(k)}u=zu$. If $u|_U=0$ where $U\subset M$ is some open subset then $u$ vanishes on the larger
open subset $\bigcup_{t\in\mathbb{R}} \varphi^t_f(U)$. 
\end{lemm}
The proof of this Lemma follows from the fact that $\varphi^{t*}_fu=e^{tz}u$.

\subsection{Proving $\ml{R}_k(f,g)\subset\ml{I}_{k}$}\label{sss:general} We let $0\leq k\leq n$. Let $z=\lambda+i\gamma$ be an element in $\ml{R}_k(f,g)$. We will prove 
that $z$ belongs to $\ml{I}_{k}$ by ordering $a_1<a_2<\ldots<a_K$ by increasing order
as in the proof of Lemma~\ref{l:openneighbor} and we will treat our problem using the pull-back Theorem of H\"ormander~\cite{Ho90}.
Assume that, for all $i\leqslant j$, the eigenfunction $u$ vanishes near $a_i$ and assume the germ of $u$ near $a_{j+1}$ is non vanishing (see Lemma~\ref{l:support-Lemma} below). In 
that case, 
we would like to prove that  $z$ is of the form:
$$z=-\sum_{l=1}^n\alpha_l|\chi_l(a_{j+1})|,$$
with some restrictions on the coefficients $\alpha_l$. For that purpose, we start with the following central observation:
\begin{lemm}\label{l:support-Lemma}
Let $u\in \mathcal{D}^{\prime,k}(M)$ be some eigencurrent of $-\ml{L}_{V_f}^{(k)}$ acting on $\ml{H}_k^m(M)$. 

If $u$ vanishes in some neighborhood
of all $a_i$ for $i\leqslant j$, then $u$ restricted to the level $f^{-1}(<f(a_{j+1}))$ vanishes. Moreover, if the germ 
$u|_{V_{a_{j+1}}}\neq 0$ (for the adapted chart 
$\kappa_{a_{j+1}}:V_{a_{j+1}}\rightarrow W_{a_{j+1}}$ defined in paragraph~\ref{sss:adapted}), then the germ $u|_{V_{a_{j+1}}}$ 
is supported in the germ of unstable manifold $W^u(a_{j+1})\cap V_{a_{j+1}}$.
\end{lemm}
\begin{rema}
A first consequence of this Lemma is that there is necessarily a critical point $a$ in a neighborhood of which $u$ does not vanish.
\end{rema}
\begin{proof} Assume without loss of generality that $f(a_{j+1})=0$.
The level $f^{-1}(<0)$ contains only the critical points $\{a_1,\dots,a_j\}$. Moreover, since the value of the potential $f$ is monotonic along the flow 
it follows that the level $f^{-1}(<0)$ is contained in the union of unstable manifolds $\bigcup_{i\leq j} W^u(a_i)$. Hence, by the propagation Lemma, 
$u|_{\{f^{-1}(<0)\}}=0$. Now consider some open set $V$ in $V_{a_{j+1}}$ which does not interset $W^u(a_{j+1})$. Using the facts that $f$ is excellent
that $W^u(a_{j+1})$ is an embedded submanifold and that $f$ must increase along the flow, one knows that, for every $x$ in $V$, $x_-=\lim_{t\rightarrow -\infty}\varphi_f^t(x)$ belongs to $\{a_1,\ldots, a_j\}$. Using the propagation Lemma one more time, one can then deduce that 
$u_{|V}=0$. This is valid for any open set $V\subset V_{a_{j+1}}$ which does not intersect $W^u(a_{j+1})$. In other words, $u_{|U}$ is supported 
in the germ of unstable manifold $W^u(a_{j+1})\cap V_{a_{j+1}}$, which concludes the proof of the Lemma.
\end{proof}
We now let $\psi$ be a smooth test form in $\Omega^{n-k}(M)$ 
which is compactly supported in $V_{a_{j+1}}$. We can then write, for every $s$ in $(0,1]$,
$$\left\la (\varphi_f^{\ln s})^*u,\psi\right\ra=s^{-(\lambda+i\gamma)}\la u,\psi\ra.$$
As we made the assumption that $u$ is not identically vanishing in the neighborhood of $a_{j+1}$, we can choose $\psi$ 
compactly supported near $a_{j+1}$ such that $\la u,\psi\ra\neq 0$. Then, we would like to prove that the left-hand side of the 
equality admits a polyhomogeneous expansion in $s$ which is indexed by the set $\ml{I}_k(a_{j+1})$. Combining this to Lemma~\ref{l:mellin-transform-constraint}, 
we would then deduce that $z$ is of the expected form. Thus, our last task is to prove that $\left\la (\varphi_f^{\ln s})^*u,\psi\right\ra$ admits a 
polyhomogeneous expansion indexed by $\ml{I}_k(a_{j+1})$. For that purpose, we shall work using the local coordinates $(x,y)$ defined in 
paragraph~\ref{sss:adapted}. We denote by $\tilde{u}$ the image of $u$ in this chart. From Lemma~\ref{l:support-Lemma}, this defines 
a current which is carried in $(-\delta_a,\delta_a)^n\cap\{x=0\}$. This can be extended into a current defined on 
$\tilde{W}_a:=\IR^{r(a)}\times(-\delta_a,\delta_a)^{n-r(a)}$ by setting $\tilde{u}=0$ outside $(-\delta_a,\delta_a)^n$. Then, we introduce the 
following map
$$\Phi:(s_1,\dots,s_n;x,y) \in (-1,1)^n\times \tilde{W}_a \mapsto (s_1^{-1}x_1,\ldots, s_r^{-1}x_r,s_{r+1}y_{r+1},\ldots, s_ny_n)\in\tilde{W}_a . $$
Note that this is well defined except if $s_i=0$ for some $1\leq i\leq r(a)$. We also define the partial maps:
$$\Phi^1:(s_1,\dots,s_n;x,y) \in (-1,1)^n\times \tilde{W}_a \mapsto (x_j,s_jy_j)\in\tilde{W}_a , $$
and
$$\Phi^2:(s_1,\dots,s_n;x,y) \in (-1,1)^n\times \tilde{W}_a \mapsto (s_jx_j,y_j)\in\tilde{W}_a. $$
Contrary to $\Phi$, these two maps are well defined for $\mathbf{s}$ belonging to the whole set $(-1,1)^n$. Let $\mathbf{s}$ be a 
point in $(-1,1)^n$ with all entries which are non vanishing. In that case, we can write
\begin{equation}\label{e:pullback-multivariable}\la \Phi(\mathbf{s})^*\tilde{u},\tilde{\psi}\ra= 
\la \Phi^2(\mathbf{s})_*\Phi^1(\mathbf{s})^*\tilde{u},\tilde{\psi}\ra=\la\Phi^1(\mathbf{s})^*\tilde{u},\Phi^2(\mathbf{s})^*\tilde{\psi}\ra.
\end{equation} 
This is valid as long as $s_j\neq 0$ for every $1\leq j\leq n$. Our next step is to show that this extends as a 
smooth function on $(-1,1)^n$. From the previous expression, one can observe that the main concern is to be able to
study the smoothness of $\Phi^1(\mathbf{s})^*\tilde{u}$ in the variable $(s_j)_j\in (-1,1)^{n}$. Recall that $u$ is an eigenvector of $-\ml{L}_{V_f}^{(k)}$ 
acting on a certain anisotropic Sobolev space $\ml{H}^m_k(M)$. According to~\cite[Th.~1.5]{FaSj11}, the eigenmodes are independent of the choice of the order 
function $m(x,\xi)$ satisfying the assumptions of Lemma~\ref{l:escape-function}. Letting the parameter $N_1\rightarrow+\infty$ in 
Lemma~\ref{l:escape-function} and using 
Lemma~\ref{l:support-Lemma}, one finds that the wave front set $WF(\tilde{u})$ of $\tilde{u}$ satisfies the following
\begin{equation}\label{e:wavefront}
WF(\tilde{u})\subset\left\{(0,y,\xi,0)\in T^*\tilde{W}_a:\xi\neq 0\right\}.
\end{equation}
We would now like to define the pull-back of $\tilde{u}$ under the map $\Phi^1$, and, for that purpose, we shall apply 
H\"ormander's pullback Theorem~\cite[Th.~8.2.4]{Ho90} -- see also~\cite{BrDaHe14}. Hence, we have to compute the normal $N^*_{\Phi^1}\subset T^*\tilde{W}_a$ of the map $\Phi^1$, 
\begin{eqnarray*}
N^*_{\Phi^1}&=&\{(x,s_j y_j;\xi,\eta)\text{ such that } (0,0,0)= (\xi,\eta)\circ d_{(\mathbf{s},x,y)}\Phi^1, (\xi,\eta)\neq (0,0)\}\\
&=&\{(x,s_j y_j;\xi,\eta)\text{ such that } (0,0,0)= (\xi,\eta)\circ 
\left(\begin{array}{ccc} 
0&1&0\\
y&0&(s_j)_j
\end{array}
\right),(\xi,\eta)\neq (0,0)\}\\
&=& \left\{(x,s_j y_j;\xi,\eta)\text{ such that } \left(\sum_j y_j\eta^j,\xi,\sum_j s_j\eta^j\right)=(0,0,0) ,(\xi,\eta)\neq (0,0)\right\}\\
&=&\left\{(x,s_j y_j;0,\eta)\text{ such that }\eta\neq 0, \sum_js_j\eta^j=\sum_jy_j\eta^j=0 \right\}.
\end{eqnarray*}
In particular, from~\eqref{e:wavefront}, $N^*_{\Phi^1}\cap WF(\tilde{u})$ is empty. Hence, we can apply H\"ormander's pullback Theorem, i.e.
$(\Phi^1)^*\tilde{u}$ is well defined and its wave front set
is contained in 
\begin{eqnarray*}
(\Phi^1)^*WF(\tilde{u})&=&\left\{\left(\mathbf{s},x,y;\sum_jy_j\eta^j,\xi,\sum_js_j\eta^j\right)\text{ such that }\left(x,\sum_js_j y_j;\xi,\eta\right)\in WF(\tilde{u})\right\}\\
&\subset&\{(\mathbf{s},0,y;0,\xi,0)\text{ such that }\xi\neq 0 \}.
\end{eqnarray*}
As $\tilde{\psi}$ is a smooth test form, we note that $(\Phi^1)^*\tilde{u}\wedge(\Phi^2)^*\tilde{\psi}$ is a current of degree $n$ 
on $(-1,1)^n\times \tilde{W}_a$ whose wave front set is included in $(\Phi^1)^*WF(\tilde{u})$. 
Consider now the pushforward of this current under the map~:
$$\mathbf{p}:(x,y,\mathbf{s})\in\tilde{W}_a\times(-1,1)^n \mapsto \mathbf{s} \in (-1,1)^n.$$
By the push-forward Theorem~\cite{Ho90, BrDaHe14}, the wave front set of the pushforward distribution is included in
$$\mathbf{p}_*\left((\Phi^1)^*WF(\tilde{u})\right)=
\left\{(\mathbf{s};\mathbf{\sigma})\text{ such that }(\mathbf{s},x,y;\mathbf{\sigma},0,0)\in (\Phi^1)^*WF(\tilde{u}),\ \sigma\neq 0\right\}=\emptyset.$$
In other words, the pushforward distribution is a smooth function in the variable $\mathbf{s}\in(-1,1)^n.$ In particular, 
according to~\eqref{e:pullback-multivariable}, $\la \Phi(\mathbf{s})^*\tilde{u},\tilde{\psi}\ra$ 
has a well--defined Taylor expansion in $\mathbf{s}$ around $0$. Then, we can combine 
Lemma~\ref{l:compo-smooth-becomes-polyhomogeneous} to the fact that, in our system of adapted coordinates, the reparametrized flow 
$\varphi^{\ln s}_f$ can be written $(s^{\chi_j(a)}x_j,s^{\chi_j(a)}y_j).$ From that, we deduce the expected property, i.e. 
$\left\la (\varphi_f^{\ln s})^*u,\psi\right\ra$ admits a polyhomogeneous expansion indexed by $(\alpha.|\chi(a)|)_{\alpha\in\IN^n}$. In order to conclude the proof, we should observe 
that $u(x,y,dx,dy)$ is of degree $k$ and $\psi(x,y,dx,dy)$ of degree $n-k$. This forces that some of the $\alpha_j$ do not vanish when we express $z$ as a combination of the Lyapunov exponents, 
i.e. $z$ must in fact belong to the set $\ml{I}_k(a_{j+1})$. This concludes the proof of the inclusion $\ml{R}_k(f,g)\subset\ml{I}_k$.

\begin{rema}\label{r:smoothness-pullback}
We shall use this kind of arguments several times in the following. We observe that we have just been able to prove that $\la \varphi_f^{\ln s *} u,\psi\ra$ 
has a polyhomogeneous expansion indexed by the set $\ml{I}_k(a_{j+1})$ and that our proof only made use of the facts that the support of $u$ near $a_{j+1}$ 
was included in $W^u(a_{j+1})$ and that its wave front is included in $E_u^*$ near $a_{j+1}$.
\end{rema}

In particular, we have implicitely shown the following useful statement:
\begin{prop}\label{p:indexeigencurrentversusdegree-general}
Suppose that $(f,g)$ is a smooth Morse pair generating a Morse-Smale gradient flow. Let $u\neq 0$ be an element of $\ml{H}^m_k(M)$ satisfying $-\mathcal{L}_{V_f}^{(k)}u=\lambda u$. 
Let $a$ be the critical point of $f$ satisfying the following properties:
\begin{itemize}
\item $u$ does not vanish in any neighborhood of $a$,
\item for every $a'$ in $\operatorname{Crit} (f)$ satisfying $f(a')<f(a)$, $u$ identically vanishes near $a'$.
\end{itemize}
Then, $\lambda$ belongs to the index set $\ml{I}_k(a)$. Moreover, if $\lambda=0$, then the index of $a$ is equal to $k$.
\end{prop}

\subsection{Asymptotic expansion of the correlation function}\label{ss:correlation}

Let us now draw some consequences of the fact that $\ml{R}_k(f,g)\subset\ml{I}_{k}$ following the lines 
of~\cite{FaRoSj08}. From~\eqref{e:semigroup} in the appendix, we also know that $(\varphi_f^{-t})^*$
generates  
a strongly continuous semi-group from 
$\ml{H}^m_k(M)$ to $\ml{H}^m_k(M)$ for every $1\leq k\leq n$ whose norm is bounded by $e^{tC_0}.$ Fix now $\Lambda>0$. Suppose without loss of generality that $-\Lambda$ does not belong to 
$\ml{I}_{k}$. From Proposition~\ref{p:eigenvalues}, we now observe that, 
for every $\Lambda>0$, one can find a weight function $m(x,\xi)$ such
that the operator
$$-\ml{L}_{V_f}^{(k)}:\ml{H}^m_k(M)\rightarrow \ml{H}^m_k(M)$$
has only discrete spectrum with finite multiplicity in the half plane $\text{Re}(\lambda)>-\Lambda$. Moreover, from the fact that 
$\ml{R}_k(f,g)\subset\ml{I}_{k}$, the operator \emph{has only finitely many eigenvalues in this region} which are real and nonpositive. 
We denote by $-\lambda_i^{(k)}$ the eigenvalues of $-\ml{L}_{V_f}^{(k)}$ 
(counted with their algebraic multiplicities). Note that each eigenvalue may a 
priori be associated with a Jordan block of size $d_i^{(k)}\geq 1$. Following~\cite[App.~A]{HeSj86}, we fix a Jordan path in $\mathbb{C}$ which 
separates the eigenvalues in the half plane $\text{Re}(\lambda)>-\Lambda$ from the rest of the spectrum. Then, according to this reference, the 
\emph{spectral projector} associated with this finite part of the spectrum can be written as
$$\Pi_{\Lambda}^{(k)}:=\frac{1}{2i\pi}\int_{\gamma}(-\ml{L}_{V_f}^{(k)}-z)^{-1}dz.$$
We can then split the operator $-\ml{L}_{V_f}^{(k)}$ as follows:
$$-\ml{L}_{V_f}^{(k)}:=\Pi_{\Lambda}^{(k)}\circ(-\ml{L}_{V_f}^{(k)})\circ\Pi_{\Lambda}^{(k)}+(\text{Id}-\Pi_{\Lambda}^{(k)})\circ(-\ml{L}_{V_f}^{(k)})
\circ(\text{Id}-\Pi_{\Lambda}^{(k)}).$$
According to~\cite[p.~244-246]{EnNa00}, the spectrum of the operator $(\text{Id}-\Pi_{\Lambda}^{(k)})\circ(-\ml{L}_{V_f}^{(k)})\circ(\text{Id}-\Pi_{\Lambda}^{(k)})$ is 
contained in the half plane $\text{Re}(\lambda)<-\Lambda$ while the finite rank part can be written as
\begin{equation}\label{e:jordan-decomposition}
 \Pi_{\Lambda}^{(k)}\circ(-\ml{L}_{V_f}^{(k)})\circ\Pi_{\Lambda}^{(k)}=\sum_{i:\lambda_i^{(k)}\leq\Lambda}
 \left(\sum_{l=1}^{d_i^{(k)}}-\lambda_i^{(k)}|u_{i,l}^{(k)}\ra\la v_{i,l}^{(k)}|+
 \sum_{l=1}^{d_i^{(k)}-1}|u_{i,l}^{(k)}\ra\la v_{i,l+1}^{(k)}|\right)
\end{equation}
where
\begin{itemize}
 \item $(u_{i,l}^{(k)})_{\lambda_i^{(k)}\leq\Lambda, l=1,\ldots d_i^{(k)}}$ belongs to $\ml{H}^m_k(M)\subset\ml{D}^{\prime,k}(M)$, 
 \item $(v_{i,l}^{(k)})_{\lambda_i^{(k)}\leq\Lambda, l=1,\ldots d_i^{(k)}}$ belongs to $\ml{H}^{-m}_{n-k}(M)\subset\ml{D}^{\prime,n-k}(M)$, 
 \item $|u\ra \la v|: \psi\in \ml{H}^m_k(M)\mapsto \la v,\psi\ra u\in \ml{H}^m_k(M)$
\end{itemize}
Recall from~\cite[Th.~1.5]{FaSj11} that these ``generalized eigendistributions'' are intrinsic and that they do not depend on the choice of the 
order function $m$. We also note that the vectors $v_*^{(k)}$ give rise to a Jordan basis for the spectral decomposition of the dual operator acting 
on $\ml{H}^{-m}_{n-k}(M)$. We now want to relate this spectral decomposition 
to the correlation function from the introduction~:

\begin{prop}\label{p:correlation} Let $0\leq k\leq n$ and $\mathcal{I}_k$ be the subset defined in Proposition~\ref{p:ruelle-spectrum}. 
Then for every $i\geq 0$, there is an integer
$d_{i}^{(k)}\geq 1$ and $\lambda_i^{(k)}\in \mathcal{I}_k$ s.t. for any $\Lambda>0$, for every $(\psi_1,\psi_2)\in\Omega^k(M)\times\Omega^{n-k}(M)$ and for every $t\geq 0$,
\begin{equation}\label{e:correlator}
\la(\varphi_f^{-t})^*\psi_1,\psi_2\ra =\sum_{i:\lambda_{i}^{(k)}< \Lambda}e^{-\lambda_i^{(k)} t}\sum_{l=0}^{d_i^{(k)}-1} \frac{t^l}{l!}
\sum_{j=l+1}^{d^{(k)}_i}\la u_{i,j}^{(k)},\psi_2\ra
\la v_{i,j+l}^{(k)},\psi_1\ra
+\ml{O}_{\psi_1,\psi_2,\Lambda}(e^{-\Lambda t}).
\end{equation}
In fact, the result also holds for any $\psi_1$ in $\ml{H}^{m}_k(M)$ provided the parameters $(N_0,N_1)$ involved in the definition of $m$ are large enough. 
\end{prop}

Note that the sum is finite and that all the quantities involved in the sum are independent of the choice of the order function $m$. This 
expression gives us an asymptotic 
expansion for the correlation function at any order of precision. As was already explained, all the 
$\lambda_i^{(k)}$ appearing in the sum belong to the set $-\ml{I}_{k}\subset\IR_+$. The rest of the article is devoted to a more precise understanding of the terms appearing 
in this asymptotic expansion. Yet, before that, let us prove this Proposition.

\begin{proof} Fix $q\geq 1$. We first 
follow the arguments of~\cite[Th.~1]{FaRoSj08} applied to the hyperbolic diffeomorphism $\varphi_q:=\varphi_f^{-\frac{1}{q}}$ rather than to the generator $-\ml{L}_{V_f}^{(k)}$. Precisely, 
following this reference, we can verify that
the order function $m$ from Lemma~\ref{l:faure-sjostrand} 
satisfies the assumptions of~\cite[Lemma 2]{FaRoSj08}. Then, following almost verbatim~\cite[section 3.2]{FaRoSj08}, 
we can deduce that the transfer operator
$$\varphi_q^*: \psi\in \ml{H}^{m}_k(M)\rightarrow \varphi_f^{-\frac{1}{q}*}\psi\in \ml{H}^{m}_k(M)$$
defines a bounded operator on the anisotropic space $\ml{H}^{m}_k(M)$ 
which can decomposed as
\begin{equation}\label{e:decompositionFRS}
\varphi_q^*=\hat{r}_{m,q}+\hat{c}_{m,q},
\end{equation}
where $\hat{c}_{m,q}$ 
is a compact operator and the remainder $\hat{r}_{m,q}$ has small 
operator norm bounded as~: $\|\hat{r}_{m,q}\|\leq e^{\frac{C-\frac{C_N}{q}}{3}}$ 
(for some uniform $C$ that may be slightly larger than the one from Proposition~\ref{p:eigenvalues}). 
Note that, for every $q\in \mathbb{N}$, we can make $\|\hat{r}_{m,q}\|$ arbitrarily small 
by choosing $N$ large enough. The proof follows similar lines as for the definition of the 
spectrum of $-\ml{L}_{V_f}^{(k)}$ except that we deal with the 
propagator at discrete times instead of the generator. Again, we can verify that the spectrum is intrinsic, i.e. 
independent of the choice of order function. 
This is because the eigenvalues
and associated spectral projectors
correspond to the poles and residues
of 
a \emph{discrete resolvent}
defined as an operator from $\Omega^k(M)$ to $\ml{D}^{\prime, k}(M)$ as follows. Consider the series
$\sum_{l=0}^{+\infty}e^{-lz}\varphi_q^{l*}$. Then, by the direct bound 
$$ \Vert  \sum_{l=0}^{+\infty}e^{-lz}\varphi_q^{l*}\psi \Vert_{\ml{H}^{m}_k(M)}\leqslant \sum_{l=0}^{+\infty}e^{-l\text{Re}(z)}\Vert\varphi_q^{*}\Vert^l\Vert\psi \Vert_{\ml{H}^{m}_k(M)}, $$
we deduce that, for $\text{Re}(z)$ large enough, the series
$\sum_{l=0}^{+\infty}e^{-lz}\varphi_q^{l*}\psi $ converges absolutely in $\ml{H}^{m}_k(M)$ for every test form $\psi\in \Omega^k(M)$. Hence, 
by the continuous injections $\Omega^k(M)\hookrightarrow \ml{H}^{m_{N_0,N_1}}_k(M)\hookrightarrow  \mathcal{D}^{\prime,k}(M)$, the identity
$$ (\text{Id}-e^{-z}\varphi_q^{*})^{-1}:=\sum_{l=0}^{+\infty}e^{-lz}\varphi_q^{l*}:\Omega^k(M)\rightarrow \ml{D}^{\prime, k}(M) $$
holds true for $\text{Re}(z)$ large enough. 
A consequence of the decomposition~(\ref{e:decompositionFRS}) is that the resolvent of $\varphi^*_q$
$$(\lambda-\varphi_q^{*})^{-1}:\Omega^k(M)\rightarrow \ml{D}^{\prime, k}(M)$$
has a meromorphic extension from $|\lambda|>e^{C_0}$ to $\lambda\in \IC$ with poles of finite multiplicity
 which correspond to the eigenvalues of the operator $\varphi_q^*$~\cite[Corollary 1]{FaRoSj08}. In other words, 
 $(\text{Id}-e^{-z}\varphi_q^{*})^{-1}:\Omega^k(M)\rightarrow \ml{D}^{\prime, k}(M)$
 has a meromorphic extension from $\text{Re}(z)>C_0$ (with $C_0>0$ large enough) to $z\in \IC$ with poles of finite multiplicity. Denote by $\tilde{\pi}_{\lambda,q}^{(k)}$ the spectral projector
 of $\varphi_q^*$ associated to the eigenvalue $\lambda$ which is obtained from
 the contour integral formula~:
 $$\tilde{\pi}_{\lambda,q}^{(k)}=\frac{1}{2i\pi}\int_\gamma \left(\mu-\varphi_q^* \right)^{-1}d\mu$$  
  where $\gamma$ is a small circle around $\lambda$. This corresponds to the residues of the discrete resolvent at $e^z=\lambda$. As $\varphi_q^*$ commutes with $-\ml{L}_{V_f}^{(k)}$, we can deduce 
  that the range of $\tilde{\pi}_{\lambda,q}^{(k)}$ is preserved by $-\ml{L}_{V_f}^{(k)}$. In particular, any eigenvalue $z_0$ of $-\ml{L}_{V_f}^{(k)}$ on that space must verify $e^{\frac{z_0}{q}}=\lambda$. 
  As we know that any resonance of $-\ml{L}_{V_f}^{(k)}$ is real, we can deduce that the poles of $(\text{Id}-e^{-z/q}\varphi_q^{*})^{-1}$ 
  belong to $\ml{R}_k(f,g)\subset\IR$ modulo $2i\pi\IZ$. 
  Take now $z_0$ in $\ml{R}_k(f,g)$. We want to show that \begin{equation}\label{e:equality-projector}\tilde{\pi}_{e^{z_0},1}^{(k)}=\pi_{z_0}^{(k)},\end{equation} 
  where $\pi_{z_0}^{(k)}$ is the spectral projector of $-\ml{L}_{V_f}^{(k)}$ associated to the eigenvalue $z_0$. Equivalently, the spectral projectors are the same for both problems. 
  Once it will be done, the proposition is just a consequence of decomposition~\eqref{e:decompositionFRS} for $q=1$ when $t$ is an integer as a consequence of~\cite[Corollary 1]{FaRoSj08}. When $t$ is a positive real number, we can 
  conclude by writing $\varphi^{-t}=\varphi^{-[t]}\varphi^{-t+[t]}$ and by using~\eqref{e:semigroup} from the appendix.
  
  In order to show~\eqref{e:equality-projector}, we first observe that 
  $\tilde{\pi}_{e^{z_0/q},q}^{(k)}=\tilde{\pi}_{e^{z_0},1}^{(k)}$ and we decompose the resolvent $(z+\ml{L}_{V_f}^{(k)})^{-1}$ as follows:
  $$(z+\ml{L}_{V_f}^{(k)})^{-1}=\sum_{l=0}^{+\infty}e^{-\frac{z}{q}}\varphi_q^*\int_0^{\frac{1}{q}}e^{-zt}\varphi_f^{-t*}dt=(\text{Id}-e^{-\frac{z}{q}}\varphi_q^{*})^{-1}\int_0^{\frac{1}{q}}e^{-zt}\varphi_f^{-t*}dt.$$
  For $\text{Re}(z)$ large enough, this expression makes sense viewed as an operator from $\Omega^k(M)$ to $\ml{D}^{\prime,k}(M)$. We have seen that it can be meromorphically continued 
  to $\IC$ by using the fact that we have built 
  a proper spectral framework and that we may pick $N_0$ and $N_1$ arbitrarily large in the definition of $m$. Consider now a small contour $\gamma$ around $z_0$ containing no 
  other elements of $\ml{R}_k(f,g)$. Integrating over this contour tells us that, for every $q\geq 1$~:
  $$\pi_{z_0}^{(k)}=\tilde{\pi}_{e^{z_0/q},q}^{(k)}q\int_0^{\frac{1}{q}}e^{-z_0t}\varphi_f^{-t*}dt=\tilde{\pi}_{e^{z_0},1}^{(k)}\int_0^1e^{-t\frac{z_0}{q}}\varphi_f^{-\frac{t}{q}*}dt.$$
  As an operator on $\Omega^k(M)$, we can observe that $\int_0^1e^{-t\frac{z_0}{q}}\varphi_f^{-\frac{t}{q}*}dt$ converges to the identity as $q\rightarrow+\infty$. Hence, $\pi_{z_0}^{(k)}=\tilde{\pi}_{e^{z_0},1}^{(k)}$ 
  as expected.
\end{proof}

\section{Proof of Theorems~\ref{t:fullasympt-nonresonant} and~\ref{c:leadingasympt}}\label{s:jordan}

We now turn to the proofs of our main results. Thanks to Proposition~\ref{p:correlation}, it amounts to compute explicitely the eigenvalues and the eigenmodes of the operator 
$-\ml{L}_{V_f}^{(k)}$ on our families of anisotropic Sobolev spaces. We proceed in several steps. First, we show how to eliminate the Jordan blocks 
(Propositions~\ref{p:kernel} and~\ref{p:strong-jordan}). Then, we define a canonical basis of our generalized eigenspaces 
(Proposition~\ref{p:generalized-laudenbach}) and conclude the proof of 
Theorems~\ref{t:fullasympt-nonresonant} and~\ref{c:leadingasympt} (Proposition~\ref{p:basis}). Finally, we end this section with some extra comments 
related to Weyl's asymptotics and trace formulae.

\subsection{Jordan blocks} We distinguish the cases $\lambda=0$ and $\lambda\neq 0$ as we can eliminate Jordan blocks for reasons of different nature in both cases.
\subsubsection{The case $\lambda=0$}
Let us first show the absence of Jordan blocks in the kernel:
\begin{prop}\label{p:kernel} Suppose that $(f,g)$ is a smooth Morse pair which induces a Morse-Smale gradient flow. 
Let $0\leq k\leq n$. Then, when acting on a convenient\footnote{It means that there is a discrete 
spectrum for $\text{Re}(\lambda)<C$ if $C>0$.} anisotropic space $\ml{H}^m_k(M)$, one has
$$\operatorname{Ker}(\ml{L}_{V_f}^{(k)})=\operatorname{Ker}((\ml{L}_{V_f}^{(k)})^2).$$
\end{prop}

We start with the following Lemma:
\begin{lemm}\label{l:neighbor} Let $1\leq j\leq K$. Then, there exists an open neighborhood $V_{a_{j}}$ of $a_{j}$ such that, for every $i< j$ with $r(a_i)\geq r(a_{j})$, one has
$$W^u(a_i)\cap V_{a_{j}}=\emptyset.$$
\end{lemm}

\begin{proof} Let $1\leq j\leq K$. Let $i<j$ satisfying $r(a_i)\geq r(a_j)$. The closure of $W^u(a_i)$ is a compact subset. In order to prove this 
Lemma, we suppose by contradiction that $a_j$ 
belongs to $\overline{W^u(a_i)}$. It means that there exists a sequence $(x_m)_{m\geq 1}$ in $W^u(a_i)$ such that $x_m$ converges to $a_j$ as 
$p\rightarrow+\infty$. Without loss of generality, we can suppose that there exists a unique $b$ in $\text{Crit} (f)$ such that, for every $m\geq 1$,
$x_m\in W^s(b)$. Using the conventions of Remark~\ref{r:weber} for our fixed sequence $(x_m)_{m\geq 1}$, there exists a sequence of critical 
points $a_i=b_{l},b_{l-1},\ldots, b_1=b$ and a finite sequence of points $z^{(p)}$ in $W^{u}(b_{p+1})\cap W^s(b_{p})$ such 
that
$$\forall\epsilon_1>0,\exists m_0,\forall m\geq m_0,\ d\left(\ml{O}(x_m),\cup_{1\leq p\leq l-1}\ml{O}(z^{(p)})\right)<\epsilon_1,$$
where $\ml{O}(x)$ denotes the orbit of $x$ under the flow $\varphi_f^t$. Note from Lemma~\ref{l:decrease-dimension} that $r(b_{p+1})<r(b_{p})$ for every 
$1\leq p\leq l-1$. On the one hand, $x_m$ converges to a point belonging to $\cup_{1\leq p\leq l-1}\ml{O}(z^{(p)})$, and, on the 
other one, $x_m$ converges to $a_j$. Hence, there exists $0\leq p\leq l-1$ such that $a_j=b_p$, and one has $r(a_j)=r(b_p)>r(b_l)=r(a_i)$ 
which gives the contradiction.
\end{proof}

We can now give the proof of Proposition~\ref{p:kernel}. Suppose by contradiction that there exists a Jordan block associated to the eigenvalue $0$ for a certain degree of currents $k$. 
Then, it means that there exists $u_0\neq 0$ and 
$u_1\neq 0$ in our anisotropic Sobolev space of currents $\ml{H}^m_k(M)$ such that
$$\ml{L}_{V_f}^{(k)}u_0=0\ \text{ and } \ml{L}_{V_f}^{(k)}u_1= u_0.$$
Integrating these expressions, we find that, for all $t$ in $\IR$,
$$(\varphi_f^t)^*u_0=u_0\ \text{and}\ (\varphi_f^t)^*u_1=u_1+tu_0.$$
As in our computation of the spectrum, we let $t=\ln s$ with $0<s\leq 1$, 
\begin{equation}\label{e:jordan-block}(\varphi_f^{\ln s})^*u_1-u_1=(\ln s) u_0.\end{equation}
As above, we order our critical points $a_1<a_2<\ldots< a_K$ using the fact that the critical values of $f$ are distinct. 

We now use this Lemma to get the expected contradiction. We fix $j\geq0$. We suppose that 
$u_0$ is vanishing in a neighborhood of any critical point $(a_i)_{i\leq j}$ and that it does not 
vanish in a neighborhood of $a_{j+1}$. According to Lemma~\ref{l:support-Lemma}, we can deduce that 
$\text{supp}(u_0)\cap V_{a_{j+1}}$ is included in $W^u(a_{j+1})$. Arguing as in paragraph~\ref{sss:general} 
(i.e. via the pull-back Theorem of H\"ormander), we can verify that $\la(\varphi_f^{\ln s})^*u_0,\psi\ra=\la u_0,\psi\ra$ has a \emph{bounded} 
asymptotic expansion in $s$ for $\psi$ a smooth test form compactly supported in $V_{a_{j+1}}$. Moreover, we can choose $\psi$ such 
that the right hand side of the equality does not vanish. Hence, the leading order of this expansion must be of degree $0$. This implies that $a_{j+1}$ 
is a critical point of index $r(a_{j+1})=k$.

We would now like to prove that, near $a_{j+1}$, $u_1$ is also supported in $W^u(a_{j+1})$. We fix $V$ an 
open subset of $V_{a_{j+1}}$ which does not intersect $W^u(a_{j+1})$. From Lemma~\ref{l:neighbor}, one knows 
that, for every $x$ in $V$, there exists $i\leq j$ such that $a_i=\lim_{t\rightarrow-\infty}\varphi_f^t(x)$ and $r(a_i)<r(a_{j+1})$. 
Hence, we would conclude that $\text{supp}(u_1)\cap V_{a_{j+1}}$ is included in $W^u(a_{j+1})$ if we could show that, 
for every $i\leq j$ with $r(a_i)<k$, $u_1$ identically vanishes in an open neighborhood of $a_i$. 

Let $i_0\leq j$ be an index with $r(a_{i_0})<k$. Then, either $\text{supp}(u_1)\cap V_{a_{i_0}}$ is included in $W^u(a_{i_0})$, or, as $\ml{L}_{V_{f}}u_1=u_0=0$ 
on $f^{-1}(<f(a_{j+1}))$, we can deduce by propagation that there exists a critical point $a$ of smaller index such that $u_1$ 
does not vanish in a neighborhood of $a$. Hence, without loss of generality, we can suppose that $\text{supp}(u_1)\cap V_{a_{i_0}}$ is included in $W^u(a_{i_0})$. 
As $\ml{L}_{V_f}u_1=0$ in $V_{a_{i_0}}$, we can argue one more time as in paragraph~\ref{sss:general}. From that, we deduce that
$$\la(\varphi_f^{\ln s})^*u_1,\psi\ra=\la u_1,\psi\ra$$
has a bounded asymptotic expansion in $s$ for every choice of $\psi$ compactly supported in $V_{a_{i_0}}$. Using the fact that $r(a_{i_0})<k$, 
we conclude that the left hand side must go to $0$ as $s\rightarrow 0^+$. Thus, one has $\la u_1,\psi\ra =0$ as expected from which we deduce that 
 $\text{supp}(u_1)\cap V_{a_{j+1}}$ is included in $W^u(a_{j+1})$.

Thanks to the fact that $\text{supp}(u_1)\cap V_{a_{j+1}}$ is included in $W^u(a_{j+1})$ and to the fact that $u_1$ belongs to our family of 
anisotropic spaces, we can argue one more time as in paragraph~\ref{sss:general}. We find then that
$\la (\varphi_f^{\ln s})^*u_1,\psi\ra$ has a bounded asymptotic expansion as $s\rightarrow 0^+$ for any 
smooth test function $\psi$ supported near $a_{j+1}$. 
Using then that $u_1$ verifies equation~\eqref{e:jordan-block}, we can finally conclude that $\la u_0,\psi\ra=0$ for every $\psi$ supported near $a_{j+1}$ which 
gives the contradiction to the fact that there exists a nontrivial Jordan block in the kernel.
\subsubsection{The case $\lambda\neq 0$} When we make a slightly more restrictive assumption on the Lyapunov exponents, we can also eliminate Jordan blocks when $\lambda\neq 0$:

\begin{prop}\label{p:strong-jordan}  Let $\varphi_f^t$ be a Morse-Smale gradient flow all of whose Lyapunov exponents are rationally independent.
Let $0\leq k\leq n$. Then, when acting on a convenient anisotropic space $\ml{H}^m_k(M)$, one has, for every $\lambda\in\ml{R}_{k}(f,g)$,
$$\operatorname{Ker}((\ml{L}_{V_f}^{(k)}+\lambda))=\operatorname{Ker}((\ml{L}_{V_f}^{(k)}+\lambda)^2).$$
\end{prop}

\begin{proof} We only treat the case $\lambda\neq 0$. 
Suppose by contradiction that there exists a Jordan block associated to the eigenvalue $\lambda> 0$ 
for a certain degree $k$. Once again, it means that there exists $u_0\neq 0$ and 
$u_1\neq 0$ in our anisotropic Sobolev space of currents $\ml{H}^m_k(M)$ such that
$$\ml{L}_{V_f}^{(k)}u_0=\lambda u_0\ \text{ and } \ml{L}_{V_f}^{(k)}u_1=\lambda u_1+ u_0.$$
Integrating these expressions, we find that, for all $t$ in $\IR_-$,
$$(\varphi_f^{t})^*u_0=e^{\lambda t} u_0\ \text{and}\ (\varphi_f^{t})^*u_1=e^{\lambda t}\left(u_1+tu_0\right).$$
As in our computation of the spectrum, we let $t=\ln s$ with $0<s\leq 1$, 
\begin{equation}\label{e:jordan-block-nonzero}(\varphi_f^{\ln s})^*u_1-s^{\lambda}u_1=s^{\lambda}(\ln s) u_0.\end{equation}
Following the proof of paragraph~\ref{sss:general}, we denote by $j+1$ the index point such that, for every $i\leq j$, $u_0$ vanishes in a neighborhood of 
$a_i$ and such that $u_0$ does not vanish near $a_{j+1}$. This implies that $\text{supp} (u_0)\cap V_{a_{j+1}}$ is included in $W^u(a_{j+1})$ and that 
$\lambda$ is of the following form:
$$\lambda=\sum_{i=1}^n\alpha_i|\chi_i(a_{j+1})|.$$
As $\ml{L}_{V_f}^{(k)}u_1=\lambda u_1+ u_0$, we know that $\ml{L}_{V_f}^{(k)}u_1=\lambda u_1$ on the open set $f^{-1}(<f(a_{j+1}))$. Suppose now that there 
exists $i_0\leq j$ such that $u_1$ does not identically vanish near $a_{i_0}$. Without loss of generality, we may suppose that $i_0$ is minimal. 
In such a neighborhood, one has $\ml{L}_{V_f}^{(k)}u_1=\lambda u_1$ as $u_0$ 
vanishes near $a_{i_0}$. Arguing as in paragraph~\ref{sss:general} one more time, we would find that
$$\lambda=\sum_{i=1}^n\alpha_i|\chi_i(a_{i_0})|.$$
As we supposed that $\lambda\neq 0$ and that the Lyapunov exponents are rationally independent, this would lead to a contradiction. Hence, $u_1$ 
vanishes near any critical point $a_i$ with $i\leq j$. As $u_0$ is locally supported on $W^u(a_{j+1})$, we know that $u_1$ still 
satisfies the eigenvalue equation $\ml{L}_{V_f}^{(k)}u_1=\lambda u_1$ near $a_{j+1}$ and outside $W^u(a_{j+1})$. By propagation, we deduce 
that $u_1$ is locally supported on $W^u(a_{j+1})$. According to Remark~\ref{r:smoothness-pullback}, we are then able to infer that 
$\la (\varphi_f^{\ln s})^*u_1,\psi\ra$ has a (bounded) polyhomogeneous expansion in $s$ as $s\rightarrow 0^+$ for every smooth test form 
$\psi$ supported near $a_{j+1}$. From our assumption on $j$, one can find $\psi$ such that $\la u_0,\psi\ra\neq 0$ which gives the expected 
contradiction when we write~\eqref{e:jordan-block-nonzero}.
\end{proof}

\subsection{Background material on currents}\label{s:background-current}

In order to describe the eigenmodes, we start with some background material on the theory of currents. 
By a Theorem of Schwartz~\cite[Th.~37 p.~102]{Sch66}
whose adaptation to the case of currents is straightforward, we first recall that~:
\begin{theo}[Schwartz]
Let $u$ be a current of degree $k$ supported by a smooth submanifold $X$ embedded in $M$. 
Suppose that in a small neighborhood of $x\in X$, one has a system of coordinate functions
$(x_i,y_j)_{1\leqslant i\leqslant r,r+1\leqslant j\leqslant n}$ 
where the coordinates $(x_i)_{1\leqslant i\leqslant r}$ are transversal coordinates of $X$,
i.e. the submanifold $X$ is given by the equations $\{x_i=0,1\leqslant i\leqslant r\}$. 
Then the current $u$ reads locally as a finite sum~:
\begin{equation}
u(x,y)=\sum_{\alpha,|I|+|J|=k} u_{\alpha,I,J}(y)\partial^\alpha_x\delta^{\mathbb{R}^r}_{\{0\}}(x)dx^I\wedge dy^J 
\end{equation}
where $(\alpha,I,J)$ are multi--indices, the $u_{\alpha,I,J}$ 
are distributions in $\mathcal{D}^\prime(\mathbb{R}^{n-r})$.
\end{theo}
If we denote by $N^*X$ the conormal bundle of $X$, we also have the following property~\cite[Lemma~9.2]{Da16}:
\begin{lemm}\label{l:dang}
Suppose that the assumptions of the previous Theorem hold and use the same notations. If $WF(u)\subset N^*(X)$, then
the current $u$ reads 
\begin{equation}
u(x,y)=\sum_{\alpha,|I|+|J|=k} u_{\alpha,I,J}(y)\partial^\alpha_x\delta^{\mathbb{R}^r}_{\{0\}}(x)dx^I\wedge dy^J 
\end{equation}
where the $u_{\alpha,I,J}$ are smooth functions in $C^\infty(\mathbb{R}^{n-r})$.
\end{lemm}

\subsection{Description of the eigenmodes}

Let $0\leq k\leq n$. In this paragraph, we will construct a ``canonical'' basis for every eigenspace of the operator $\ml{L}_{V_f}^{(k)}$ acting on the anisotropic 
space $\ml{H}^m_k(M)$. We proceed in three steps. First, we determine the local shape of an eigenmode $u$ near the ``smallest'' critical point where it does 
not vanish -- recall that such a point exists from Lemma~\ref{l:support-Lemma}. Then, we prove that the germ of current defined near
the critical point $a$ can be 
extended into a current carried by $\overline{W^u(a)}$. Finally, we show that these currents form indeed a basis 
of the kernel.

Before starting the proof which is a little bit combinatorial, recall from Proposition~\ref{p:ruelle-spectrum} that any eigenvalue $\lambda$ of the operator
$\ml{L}_{V_f}^{(k)}:\ml{H}_k^m(M)\rightarrow\ml{H}_k^m(M)$ must be of the form
$$\lambda= \sum_{j\in I\cup J}(\alpha_j+1)|\chi_j(a)|+\sum_{j\in(I\cup J)^c}\alpha_j|\chi_j(a)|,$$
where
\begin{itemize}
 \item $a$ is a critical point of index $r$,
 \item for every $1\leq j\leq n$, $\alpha_j$ is a nonnegative integer,
 \item $I\subset\{1,\ldots, r\}$ and $J\subset\{r+1,\ldots,n\}$ such that $|J|-|I|=k-r$.
\end{itemize}

\subsubsection{Local form near the ``smallest'' critical point}\label{sss:local-irrational}

Let $u\neq 0$ be an element in $\ml{H}^m_k(M)$ such that $\ml{L}_{V_f}^{(k)}u=\lambda u$. As before, we 
denote by $j$ the index such that, for every $i<j$, $u$ vanishes in a neighborhood of $a_i$ and such that $u$ does 
not vanish near $a_j$. Recall that $\supp(u)\cap V_{a_j}$ is included in $W^u(a_j)$ for 
some small enough neighborhood $V_{a_j}$ of $a_j$. Thanks to Proposition~\ref{p:indexeigencurrentversusdegree-general}, 
$-\lambda$ belongs to $\ml{I}_k(a_j)$. In order to alleviate notations, we will write $a_j=a$ in the following.

Using Schwartz's Theorem and Lemma~\ref{l:dang}, we deduce that, in the adapted coordinates of paragraph~\ref{sss:adapted}, the current $u$ reads 
as a finite sum~:
\begin{equation}\label{e:sum-Dirac-irrational}
u(x,y,dx,dy)=\sum_{\alpha',|I'|+|J'|=k} u_{\alpha',I',J'}(y)\partial^{\alpha'}_x\delta^{\mathbb{R}^r}_{\{0\}}(x)dx^{I'}\wedge dy^{J'} 
\end{equation}
where the $u_{\alpha',I',J'}$ are smooth functions in $C^\infty(\mathbb{R}^{n-r})$. A direct calculation shows us that, in a small enough 
neighborhood of $a$, one has, 
for every $0<s\leq 1$,
$$(\varphi_f^{\ln s *}u)(x,y,dx,dy)=\sum_{\alpha',|I'|+|J'|=k} u_{\alpha',I',J'}((s^{\chi_j(a)}y_j)_j)
\partial^{\alpha'}_x\delta^{\mathbb{R}^r}_{\{0\}}(x) s^{\tilde{\lambda}_{I',J',\alpha'}}dx^{I'}\wedge dy^{J'},$$
where
$$\tilde{\lambda}_{I',J',\alpha'}:=\sum_{j=1}^r(\alpha_j'+1)|\chi_j(a)|-\sum_{j\in I'}|\chi_j(a)|+\sum_{j\in J'}|\chi_j(a)|.$$
On the other hand, as $u$ satisfies $\ml{L}_{V_f}^{(k)}u=\lambda u$, we know that, 
for every smooth test form $\psi$ of degree $n-k$ and for every $0<s\leq 1$,
$$\la \varphi_f^{\ln s*} u,\psi\ra=s^{\lambda}\la u,\psi\ra.$$
Combining this equality to the local form of $u$, we find
$$s^{\lambda}\la u,\psi\ra=\sum_{\alpha',|I'|+|J'|=k} s^{\tilde{\lambda}_{I',J',\alpha'}} \left\la 
\partial^{\alpha'}_x\delta^{\mathbb{R}^r}_{\{0\}}(x) ,u_{\alpha',I',J'}((s^{\chi_j(a)}y_j)_j)dx^{I'}\wedge dy^{J'}\wedge\psi(x,y,dx,dy)\right\ra.$$
Write now the Taylor expansion of $u_{\alpha',I',J'}$ (which is $\ml{C}^{\infty}$). From that, we find that 
$$u_{\alpha',I',J'}(y)=c_{\alpha',I',J'}y_{r+1}^{\alpha'_{r+1}}\ldots y_{n}^{\alpha'_{n}},$$
where $c_{\alpha',I',J'}$  is some fixed constant, $\alpha_j'$ belongs to $\IN$ for every $r+1\leq j\leq n$ and  
$$\tilde{\lambda}_{I',J',\alpha'}+\sum_{j=r+1}^n\alpha_j'|\chi_j(a)|=\lambda.$$
Equivalently, one has
$$\lambda=\sum_{j=1}^r(\alpha_j'+1)|\chi_j(a)|+\sum_{j=r+1}^n\alpha_j'|\chi_j(a)|-\sum_{j\in I'}|\chi_j(a)|+\sum_{j\in J'}|\chi_j(a)|.$$
To summarize, this shows that the current $u$ reads in the adapted cooordinates near $a$:
\begin{equation}\label{e:sum-Dirac-irrational-2}
u(x,y,dx,dy)=\sum_{\alpha,I,J:(*)} c_{\alpha,I,J}\left(y\partial_x\right)^{\alpha}\delta^{\mathbb{R}^r}_{\{0\}}(x)\left(\wedge_{j\notin I}dx_j\right)\wedge
\left(\wedge_{j\in J} dy_j\right), 
\end{equation}
where $c_{\alpha,I,J}$ are some fixed constant and where $(*)$ means that $(\alpha, I,J)$ satisfies 
\begin{itemize}
\item for every $1\leq j\leq n$, $\alpha_j\in\IN$,
\item $I\subset \{1,\ldots, r\}$, $J\subset \{r+1,\ldots, n\}$,
\item $|J|-|I|=k-r$,
\item $\lambda= \sum_{j\in I\cup J}(\alpha_j+1)|\chi_j(a)|+\sum_{j\in(I\cup J)^c}\alpha_j|\chi_j(a)|,$
\end{itemize}

\subsubsection{Extension of the local form to $M$} We will now explain how the local form obtained 
in~\eqref{e:sum-Dirac-irrational-2} can be extended into a natural eigencurrent carried by the closure of $W^u(a)$. For a fixed triple $(\alpha, I, J)$ 
satisfying the conditions $(*)$, we define
\begin{equation}\label{e:local-germ}
 \tilde{U}_a^{\alpha, I,J}(x,y,dx,dy):=\theta(x,y)\left(y\partial_x\right)^{\alpha}\delta^{\mathbb{R}^r}_{\{0\}}(x)\left(\wedge_{j\notin I}dx_j\right)\wedge
\left(\wedge_{j\in J} dy_j\right),
\end{equation}
where $\theta$ is a smooth cutoff function supported near $a$ (in particular, it is equal to $1$ near $a$). By construction, one can verify that 
$$\ml{L}_{V_f}^{(k)}\tilde{U}_a^{\alpha, I,J}=\lambda \tilde{U}_a^{\alpha, I,J}$$
on the open neighborhood $(-\delta_a/4,\delta_a/4)^n$. Moreover, this current belong to the anisotropic space $\ml{H}^m_k(M)$ provided that we pick $N_0$ 
large enough (compared with $|\alpha|$) in the definition of the order function $m$. Using the conventions of~\eqref{e:jordan-decomposition}, we then set
$$U_a^{\alpha, I,J}=\sum_{i:\lambda^{(k)}_i=\lambda}\la \tilde{U}_a^{\alpha, I,J}, v_{i,1}^{(k)}\ra u_{i,1}^{(k)},$$
which obviously satisfies the eigenvalue equation:
$$\ml{L}_{V_f}^{(k)}U_a^{\alpha, I,J}=\lambda U_a^{\alpha, I,J}$$
Let us now describe some properties of this current. First, we let $\psi$ be a smooth $n-k$ form carried outside $\overline{W^u(a)}$. For such a form 
and for every $0<s\leq 1$, one has $\la \varphi_f^{-\ln s*}\tilde{U}_a^{\alpha, I,J},\psi\ra =0$. Hence, every term in the asymptotic expansion~\eqref{e:correlator} must vanish. 
In particular, one has $\la U_a^{\alpha, I,J},\psi\ra=0$ for every smooth test form 
supported outside $\overline{W^u(a)}$. Equivalently, one has 
$$\text{supp}\left(U_a^{\alpha, I,J}\right)\subset\overline{W^u(a)}.$$
By invariance under the gradient flow, the support is in fact equal to $\overline{W^u(a)}$. Like in the case of the kernel, we would like to verify 
that $\tilde{U}_a^{\alpha, I,J}$ and $U_a^{\alpha, I,J}$ coincide 
in a neighborhood of the critical point $a$. For that purpose, we let $\psi(x,y,dx,dy)$ 
be a some smooth test form carried in the neighborhood with adapted coordinates. Then, one finds that, for every $0<s\leq 1$,
$$ \la \varphi_f^{\ln s*}\tilde{U}_a^{\alpha, I,J},\psi\ra  =  s^{\lambda}(1+o(1))
   \left\la \left(y\partial_x\right)^{\alpha}\delta^{\mathbb{R}^r}_{\{0\}}(x)\left(\wedge_{j\notin I}dx_j\right)\wedge
\left(\wedge_{j\in J} dy_j\right),\psi\right\ra.$$
Using one more time the spectral expansion of the correlation function~\eqref{e:correlator} and using the fact that there is \emph{no Jordan blocks}, 
one can identify the term of order $s^{\lambda}$ in the asymptotic. In particular, this implies that
$$\la U_a^{\alpha, I,J},\psi\ra=\sum_{i:\lambda_i^{(k)}=\lambda} 
   \la \tilde{U}_a^{\alpha, I,J}, v_{i,1}^{(k)}\ra \la u_{i,1}^{(k)},\psi\ra=\left\la \left(y\partial_x\right)^{\alpha}\delta^{\mathbb{R}^r}_{\{0\}}(x)\left(\wedge_{j\notin I}dx_j\right)\wedge
\left(\wedge_{j\in J} dy_j\right),\psi\right\ra,$$
for every smooth test form $\psi$ compactly supported in a small enough neighborhood of $a$. To summarize, we have shown the following:

\begin{prop}\label{p:generalized-laudenbach} Let $\varphi_f^t$ be a Morse-Smale gradient flow all of whose Lyapunov exponents are rationally independent. 
Let $a$ be a critical point of index $r$, let $0\leq k\leq n$ and let $0\leq\theta\leq 1$ 
be a smooth cutoff function which is compactly supported in a small enough neighborhood $V_a$ of $a$, and equal to $1$ in an open neighborhood of $a$. Let 
$I$ be a subset of $\{1,\ldots, r\}$ and $J$ be a subset of $\{r+1,\ldots, n\}$ satisfying $|J|-|I|=k-r$. Let $\alpha$ be an element in $\IN^n$. 
Set $[W^u(a)]_{\alpha,I,J}$ to be the image in the adapted coordinate chart of
$$\left(y\partial_x\right)^{\alpha}\delta^{\mathbb{R}^r}_{\{0\}}(x)\left(\wedge_{j\notin I}dx_j\right)\wedge
\left(\wedge_{j\in J} dy_j\right).$$

Then, there exists an open neighborhood $\tilde{V}_a\subset V_a$ of $a$ such that the current
$$U_a^{\alpha,I,J}:=\sum_{i:\lambda_i^{(r)}=\lambda}\la \theta [W^u(a)]_{\alpha,I,J}, v_{i,1}^{(r)}\ra u_{i,1}^{(r)}$$
satisfies
\begin{itemize}
 \item $U_a^{\alpha,I,J}|_{\tilde{V}_a}=[W^u(a)]_{\alpha, I,J}|_{\tilde{V}_a}$,
 \item $\operatorname{supp}(U_a^{\alpha,I,J})= \overline{W^u(a)}$,
 \item $\ml{L}_{V_f}^{(k)}(U_a^{\alpha,I,J})=\lambda U_a^{\alpha,I,J} $ with
 $$\lambda=\sum_{j\in I\cup J}(\alpha_j+1)|\chi_j(a)|+\sum_{j\in(I\cup J)^c}\alpha_j|\chi_j(a)|.$$
 
\end{itemize}

\end{prop}

\begin{rema}\label{r:kernel}
 Note that, in the case $\lambda=0$, everything is well defined as soon as 
$(f,g)$ is a smooth Morse pair generating a Morse-Smale gradient flow. Note that the eigenvalue $0$ can only occur if the index $r$ 
of the point $a$ is equal to $k$. In that case, the current can be easily interpreted. In fact, 
using the result of Corollary D.4 in~\cite{DaRi15}, we recognize that
$\delta^{\mathbb{R}^k}_{\{0\}}(x)dx_1\wedge\ldots\wedge dx_k$ is the integral formula for the current of integration on
the germ of submanifold $W^u(a)=\{x_i=0,1\leqslant i\leqslant k\}$. This proposition shows how this germ of current can be extended to a current $U_a:=U_a^{0,\emptyset,\emptyset}$
defined on $M$. We call these currents \textbf{Laudenbach's currents of degree $k$}. Recall that the extension of this current 
is a delicate task which was first achieved by Laudenbach in~\cite{Lau92} 
in the case of a locally flat metric. 
\end{rema}

In any case, up to the linearization chart, the expression of the eigenmodes is more or less explicit. 
For every $\lambda$ in $\ml{I}_{k}(a)$, we define the ``multiplicity'' of $\lambda$ as
\begin{equation}\label{e:multiplicity}
 m_k(\lambda):=\left|\left\{(\alpha,I,J)\ \text{satisfying}\ (*)\right\}\right|,
\end{equation}
where $(*)$ means that $(\alpha, I,J)$ satisfies 
\begin{itemize} 
\item for every $1\leq j\leq n$, $\alpha_j\in\IN$,
\item $I\subset \{1,\ldots, r\}$, $J\subset \{r+1,\ldots, n\}$,
\item $|J|-|I|=k-r$,
\item $\lambda= \sum_{j\in I\cup J}(\alpha_j+1)|\chi_j(a)|+\sum_{j\in(I\cup J)^c}\alpha_j|\chi_j(a)|.$
\end{itemize}

\begin{rema}\label{r:combinatorics} In order to compute the Weyl's law for our eigenvalues, it will be convenient to rewrite things in 
a slightly different manner. More precisely, for any 
given $\alpha$ in $\IN^n$, we set
$$m_{k,a}(\alpha):=\left|\left\{(I\times J\subset\{1,\ldots ,r\}\times\{r+1,\ldots ,n\}:\ |J|-|I|=k-r,\ \text{and}\ \forall j\in I\cup J,\ \alpha_j\geq 1 \right\}\right|,$$
where $r$ is the index of $a$. With these conventions and thanks to the rational independence, any $\alpha.|\chi(a)|$ appears with 
multiplicity $m_{k,a}(\alpha)$ in $\ml{R}_k(f,g)$.
\end{rema}

\subsubsection{The generation Theorem} We conclude this section by showing that the currents we have just constructed generate a basis of 
$\text{Ker}(\ml{L}_{V_f}^{(k)}+\lambda)$, i.e.
\begin{prop}\label{p:basis} Let $\varphi_f^t$ be a Morse-Smale gradient flow all of whose Lyapunov exponents are rationally independent. Let $0\leq k\leq n$ and let $\lambda\neq 0$ be an element in $\ml{R}_k(f,g)$. 
The family of currents
$$\left\{U_a^{\alpha,I,J}: \sum_{j\in I\cup J}(\alpha_j+1)|\chi_j(a)|+\sum_{j\in(I\cup J)^c}\alpha_j|\chi_j(a)|=-\lambda \right\}$$
forms a basis of the kernel of the operator
$$\ml{L}_{V_f}^{(k)}+\lambda:\ml{H}_k^m(M)\rightarrow \ml{H}_k^m(M).$$
In particular, the kernel of this operator is of dimension $m_{k}(\lambda)$.
\end{prop}

Note that the proof of Theorem~\ref{t:fullasympt-nonresonant} is then just a combination of Proposition~\ref{p:correlation} with this statement and with the fact that there is 
no Jordan blocks. In the case $\lambda=0$, this statement is true without the rational independence assumption. Hence, except for the properties on the support of the dual basis 
$(S_a)_{a\in\text{Crit}(f)}$ (see paragraph~\ref{sss:correlation} below), Theorem~\ref{c:leadingasympt} is a consequence of Proposition~\ref{p:ruelle-spectrum}, Proposition~\ref{p:correlation}, Proposition~\ref{p:basis} (for $\lambda=0$) and
Proposition~\ref{p:kernel}.

\begin{proof} Again, we need to distinguish the case $\lambda=0$ and the case $\lambda\neq 0$. Let us start with the case $\lambda=0$. 
First, we show that this family of currents is linearly independent. For that purpose, we suppose that
$$\sum_{a\in\operatorname{Crit}(f):\operatorname{ind}(a)=k}\alpha_aU_a=0.$$
Let $a$ be the ``smallest'' point of index $k$, in the sense that, for every other point $a'\neq a$ of index $k$, $f(a')>f(a)$. We pick 
$\psi$ a smooth form which is compactly supported near $a$ and such that $\la [W^u(a)],\psi\ra \neq 0$. As the support of $U_b$ is 
contained in $\overline{W^u(b)}$ for any critical point $b$ of index $k$, we can deduce (provided that the support of $\psi$ is small enough) that
$$0=\sum_{b\in\operatorname{crit}(f):\operatorname{ind}(b)=k}\alpha_b\la U_b,\psi\ra=\alpha_a\la U_a,\psi\ra=\alpha_a\la [W^u(a)],\psi\ra.$$
From this, we deduce that $\alpha_a=0$. By induction, we can conclude that the family contains only linearly independent currents. It remains to verify that this family generates all the kernel. Let $u\neq 0$ be an element in $\ml{H}^m_k(M)$ in the kernel of $\ml{L}_{V_f}^{(k)}$. 
From Remark~\ref{r:kernel}, we know that $u$ must be equal to $c_{a}[W^u(a)]$ in a neighborhood of $a$ where $a$ is the 
``smallest'' critical point where $u$ does not identically vanish and where $c_{a}$ is a fixed constant. Set now 
$$u_1=u-c_{a}U_{a}.$$
We know that $u_1$ belongs to $\ml{H}^m_k(M)$ and that it satisfies $\ml{L}_{V_f}^{(k)}u_1=0$. Morever, by construction, we know that $u_1$ vanishes identically near 
every critical point $a'$ with $a'\leq a$. Repeating the process a finite number of times, we finally get that
$$u=\sum_{a\in\operatorname{Crit}(f):\operatorname{ind}(a)=k}c_aU_a,$$
for some $c_a$ in $\IR$.

Suppose now that $\lambda\neq 0$. Let us first show that it generates the kernel of $\ml{L}_{V_f}^{(k)}+\lambda$. This follows from the discussion from paragraph~\ref{sss:local-irrational}. 
If we take $u$ in $\ml{H}^m_k(M)$ 
satisfying $\ml{L}_{V_f}^{(k)}u=-\lambda u$, then~\eqref{e:sum-Dirac-irrational-2} gives us a family of constants $c_{\alpha,I,J}$. We then set $\tilde{u}=
u-\sum_{\alpha,I,J}c_{\alpha,I,J}U_a^{\alpha,I,J}$ where $a$ is the smallest critical point where $u$ does vanish in a neighborhood. Note that $\lambda$ belongs 
to $\ml{I}_{k}(a)$ from Proposition~\ref{p:indexeigencurrentversusdegree-general}. One still has that $\ml{L}_{V_f}^{(k)}\tilde{u}=-\lambda\tilde{u}$. From 
paragraph~\ref{sss:local-irrational}, we also know that $\tilde{u}$ vanishes near any critical point $b$ satisfying $f(b)\leq f(a)$. Then, combing the 
rational independence of the Lyapunov exponents with Proposition~\ref{p:indexeigencurrentversusdegree-general}, we conclude that $\tilde{u}=0$.
Let us now briefly verify that these elements are independent. Suppose that one can write
$$\sum_{(\alpha,I,J)\text{satisfying}\ (*)}\gamma_{\alpha,I,J}U_a^{\alpha,I,J}=0$$
We write this relation near the critical point $a$ and we use that the germs of current are (from proposition~\ref{p:generalized-laudenbach}) 
linearly independent near this critical point. This implies that $\gamma_{\alpha,I,J}=0$ for every $(\alpha,I,J)$ associated with $a$.

\end{proof}

\subsubsection{Support of the dual basis}\label{sss:correlation}

In order to conclude the proof of Theorem~\ref{c:leadingasympt}, it remains to say a word on the basis dual to $(U_a)_{a\in\text{Crit}(f)}$. Note that the 
same discussion would hold for other eigenvalues.
We have shown that there is no Jordan blocks for 
the eigenvalue $0$ (Prop.~\ref{p:kernel}), and that we can choose a basis of eigenmodes $(U_a)_{a}$ indexed by the critical points of index $k$. 
Moreover, all the elements in this basis can be chosen in such a way that the support of $U_a$ is equal to $\overline{W^u(a)}$. We denote by 
$S_a$ the corresponding dual basis. Following paragraph~\ref{ss:correlation}, we have then
$$\forall\psi_1\in\ml{H}^m_k(M),\ \forall\psi_2\in\ml{H}^{-m}_{n-k}(M),\ \la\varphi_f^{-t*}\psi_1,\psi_2\ra=\sum_{a:\text{ind}(a)=k}
\la U_a,\psi_2\ra\la \psi_1,S_a\ra+\ml{O}_{\psi_1,\psi_2}(e^{-\Lambda t}),$$
for every $t>0$. Applying the arguments of the previous paragraphs to the operator $\ml{L}_{V_{-f}}^{(n-k)}$ acting on the anisotropic 
space $\ml{H}^{-m}_{n-k}(M)$, we can construct a basis of the kernel that we denote by $(\overline{S}_a)_a$ indexed by the critical points of index $k$.
Mimicking the above procedure, we can impose that $\overline{S}_a$ has support contained in $\overline{W^s(a)}$ and that $\overline{S}_a$ coincides with $[W^s(a)]$ in 
a neighborhood of the critical point $a$. In particular, as $\overline{W^s(a)}\cap \overline{W^u(a)}=\{a\}$, we can use our local adapted coordinates near $a$ to find that 
$\la U_a,\overline{S}_a\ra=1$. Consider now $a'\neq a$ of index $k$. If we are able to show that $\la \overline{S}_a,U_{a'}\ra =0$ for every such $a'$, then we will have that 
$\overline{S}_a=S_a$ which would conclude the proof of Theorem~\ref{c:leadingasympt}. To prove this, we just need to observe that $\overline{W^s(a)}\cap\overline{W^u(a')}=\emptyset.$ In fact, according to 
Remark~\ref{r:weber} applied to $f$ and $-f$, we find that, if $x$ belongs to $\overline{W^s(a)}\cap\overline{W^u(a')}$, then $\text{ind}(x_-)\geq k$ and 
$\text{ind}(x_+)\leq k$, where $x\in W^u(x_-)\cap W^s(x_+)$. In other words, from the Morse-Smale assumption, $x_-=x_+$. From Lemma~\ref{l:neighbor}, we 
would then have $a=a'$ which gives the contradiction.

\subsection{Asymptotic formulas}

In order to conclude this section, we will give some nice asymptotic formulas that can be easily derived from our description of the spectrum.

\subsubsection{Weyl asymptotics}\label{sss:weyl}

Due to the fact that we obtained an explicit expression for the spectrum of the transfer operator, we can easily obtain some Weyl's formula. More precisely, 

\begin{prop}[Weyl Law]\label{p:weyl} Let $0\leq k\leq n$ and let $\varphi_f^t$ be a Morse-Smale gradient flow all of whose Lyapunov exponents are rationally independent.. Then, one has
$$\left|\left\{\lambda\in\ml{R}_k(f,g):|\lambda|\leq \Lambda\right\}\right|= \frac{\Lambda^n}{k!(n-k)!}\sum_{a\in \operatorname{Crit}(f)} \frac{1}{\prod_{j=1}^n\vert\chi_j(a)\vert}+\ml{O}(\Lambda^{n-1}),\ \text{as}\ \Lambda\rightarrow+\infty,$$
 where the elements in $\ml{R}_k(f,g)$ are counted with their algebraic multiplicities.
\end{prop}

\begin{proof} From Remark~\ref{r:combinatorics} and Propositions~\ref{p:basis} and~\ref{p:strong-jordan}, one knows that
$$\left|\left\{\lambda\in\ml{R}_k(f,g):|\lambda|\leq \Lambda\right\}\right|=\sum_{a\in \operatorname{Crit}(f)}\sum_{\alpha\in\IN^n:\alpha.|\chi(a)|\leq\Lambda}m_{k,a}(\alpha).$$
Hence, we can fix a critical point $a$ and compute $\sum_{\alpha\in\IN^n:\alpha.|\chi(a)|\leq\Lambda}m_{k,a}(\alpha)$. We write
$$\sum_{\alpha\in\IN^n:\alpha.|\chi(a)|\leq\Lambda}m_{k,a}(\alpha)=\left|\left\{\alpha\in\IN^n;I\subset\{1,\ldots,r\},J\subset\{r+1,\ldots,n\}:\alpha.|\chi(a)|\leq\Lambda\ \text{and}\ (**)\right\}\right|,$$
where $r$ is the index of $a$ and where $(**)$ means that $|J|-|I|=k-r$ and $\forall j\in I\cup J,\ \alpha_j\geq 1.$ We start by fixing a pair $(I,J)$ 
where $I\subset \{1,\dots,r\}$ and $J\subset \{r+1,\dots,n\}$ subject to the condition
$\vert J\vert-\vert I\vert=k-r$. We then want to compute $$\left|\left\{\alpha\in\IN^n:\forall j\in I\cup J,\ \alpha_j\geq 1\ \text{and}\ \alpha.|\chi(a)|\leq\Lambda\right\}\right|.$$
One can verify that
$$\left|\left\{\alpha\in\IN^n:\forall j\in I\cup J,\ \alpha_j\geq 1\ \text{and}\ \alpha.|\chi(a)|\leq\Lambda\right\}\right|=\left|\left\{\alpha\in\IN^n:\alpha.|\chi(a)|\leq\Lambda\right\}\right|+\ml{O}(\Lambda^{n-1}).$$
Then, one has
$$\left|\left\{\alpha\in\IN^n:\alpha.|\chi(a)|\leq\Lambda\right\}\right|=\text{Vol}\left(\left\{x\in(\IR_+)^n:|\chi(a)|.x\leq\Lambda\right\}\right)+\ml{O}(\Lambda^{n-1}),$$
which is the volume of a simplical domain. Hence, one has
$$\left|\left\{\alpha\in\IN^n:\forall j\in I\cup J,\ \alpha_j\geq 1\ \text{and}\ \alpha.|\chi(a)|\leq\Lambda\right\}\right|=\frac{\Lambda^n}{n!\vert\prod_{j=1}^n\chi_j(a)\vert}
+\ml{O}(\Lambda^{n-1}).$$
This is valid for any $I\subset\{1,\ldots, r\},J\subset\{r+1,\ldots,n\}$ subject to the condition $|J|-|I|=k-r$. 
One can remark that the number 
of such $I\times J$ is equal to the number of $I'\times J\subset\{1,\ldots, r\}\times\{r+1,\ldots,n\}$ subject to the 
condition $|J|+|I'|=k$. This is exactly equal to $\left(\begin{array}{c}
n\\k
\end{array} \right)$. This concludes the proof of the Proposition.
\end{proof}

\subsubsection{Trace formulas}
\label{ss:traceformulas}
In this paragraph, we discuss briefly some trace formulas related to our problem. For every $0\leq k\leq n$ and every $\lambda\geq 0$, we set
$$C^k(f,\lambda):=\text{Ker}(\ml{L}_{V_f}^{(k)}-\lambda),$$
where we mean the kernel of the operator in an appropriate anisotropic Sobolev space as above. We define then the spaces of even (bosonic) and 
odd (fermionic) eigenstates:
$$C^{\text{even}}(f,\lambda):=\bigoplus_{k\equiv 0(\text{mod} 2)}C^k(f,\lambda),\ \text{and}\ C^{\text{odd}}(f,\lambda):=\bigoplus_{k\equiv 1(\text{mod} 2)}C^k(f,\lambda)$$
The \textbf{fermion number} operator $(-1)^F$ acts on 
$$C(f,\lambda)=C^{\text{even}}(f,\lambda)\oplus C^{\text{odd}}(f,\lambda)$$ 
with eigenvalue $\pm 1$ depending on the parity of the state. Let now $\theta:\IR\rightarrow \IC$. We define the \textbf{super\footnote{One more time, the 
prefix super just emphasizes the fact that we are considering functions of odd $(dz_i)$ and even $(z_i)$ variables.}-trace} 
as follows:
\begin{eqnarray*}
\text{Str}\left(\theta\left(\mathcal{L}_{V_f}\right)\right)&=&\text{Tr}\left((-1)^F \theta\left(\mathcal{L}_{V_f}\right)\right)\\
&=&\text{Tr}\left(\theta\left(\mathcal{L}_{V_f}\right)\rceil_{C^{\text{even}}(f,\lambda)}\right)
-\text{Tr}\left(\theta\left(\mathcal{L}_{V_f}\right)\rceil_{C^{\text{odd}}(f,\lambda)}\right)\\
&:=&\sum_{\lambda\in\cup_{k=0}^n\ml{R}_k(f,g)} \theta(\lambda)\left(\dim C^{\text{even}}(f,\lambda)-\dim C^{\text{odd}}(f,\lambda)\right).
\end{eqnarray*}
This allows to define a notion of super-trace as soon as the last quantity is well-defined. In order to avoid too many complications that would be 
beyond the scope of this article, we take this as a definition of the trace in our framework.

The operators $d$ and $i_{V_f}$ both commute with $\mathcal{L}_{V_f}$. Hence, $Q=(d+i_{V_f})$ defines an operator
\begin{eqnarray*}
Q_{\lambda} :C^{\text{even}}(f,\lambda)\oplus C^{\text{odd}}(f,\lambda) \mapsto C^{\text{odd}}(f,\lambda)\oplus C^{\text{even}}(f,\lambda).
\end{eqnarray*}
which exchanges chiralities. We observe that, for every $\lambda>0$, $Q_{\lambda}$ is an isomorphism
since $Q_{\lambda}^2=\ml{L}_{V_f}=\lambda\text{Id}$. In particular, for every $\lambda>0$, one has
$$\dim C^{\text{even}}(f,\lambda)=\dim C^{\text{odd}}(f,\lambda).$$
Combined with Proposition~\ref{p:basis} (when $\lambda=0$), this implies
$$\text{Str}\left(\theta\left(\mathcal{L}_{V_f}\right)\right)=\theta(0)\sum_{k=0}^n(-1)^k|\{a\in\text{Crit}(f):\text{ind}(a)=k\}|=\theta(0)
\sum_{a\in\operatorname{Crit}(f)}(-1)^{\operatorname{ind}(a)}.$$
By the classical Morse inequalities, the right-hand side of this equality is equal to $\theta(0)\chi(M)$, where $\chi(M)$ is the 
Euler characteristic of $M$. We shall prove this property in section~\ref{s:topology}.

Let us now specialize this result when we take $\theta(\lambda)=e^{-\lambda t}\textbf{1}_{[0,\Lambda]}(\lambda)$ some fixed $\Lambda>0$. In that
case, we get the following spectral version of the Atiyah--Bott--Lefschetz fixed point Theorem~\cite{AtBo67}:
\begin{prop}\label{p:trace} Let $\varphi_f^t$ be a Morse-Smale gradient flow all of whose Lyapunov exponents are rationally independent. Then, one has, 
for every $\Lambda>0$, and for every $t> 0$,
\begin{equation}\label{e:lefschetz} 
\sum_{k=0}^n(-1)^{k}\operatorname{Tr}\left(\Pi_{\Lambda}^{(k)}\varphi_f^{-t*}\Pi_{\Lambda}^{(k)}\right)=\sum_{x=\varphi_f^{-t}(x)}
\frac{\operatorname{det}(\operatorname{Id}-d_x\varphi_f^{-t})}{|\operatorname{det}(\operatorname{Id}-d_x\varphi_f^{-t})|}, 
\end{equation}
where $\Pi_{\Lambda}^{(k)}$ is the spectral projector defined in paragraph~\ref{ss:correlation} and $\operatorname{Tr}$ is the standard trace.
\end{prop}
In the terminology of~\cite{AtBo67}, the left-hand side of~\eqref{e:lefschetz} is called the \textbf{Lefschetz number} of $\varphi_f^{-t*}$ (more precisely 
of $\Pi_{\Lambda}^{(k)}\varphi_f^{-t*}\Pi_{\Lambda}^{(k)}$). As was already mentionned, we will verify in the next section that 
the right-hand side is equal to the Euler characteristic $ \chi(M)$ of $M$. Note that, after integrating the previous equality against $t^{s-1}e^{-zt}$ between $0$ and $+\infty$, 
one can write the following expression for the spectral (super-)zeta function of $\ml{L}_{V_f}+z$:
$$\zeta(s,z):=\frac{1}{\Gamma(s)}\sum_{k=0}^n(-1)^{k}\int_0^{+\infty}t^se^{-zt}\operatorname{Tr}\left(\Pi_{\Lambda}^{(k)}\varphi_f^{-t*}\Pi_{\Lambda}^{(k)}\right)
\frac{dt}{t}=\frac{\chi(M)}{z^s}.$$
At $s=0$, this is formally equal to $\chi(M)$. If we differentiate this expression 
with respect to $s$ and evaluate it at $0$, we find that $e^{-\partial_s\zeta(0,z)}=z^{\chi(M)}.$ Equivalently, the 
super-determinant of $(\ml{L}_{V_f}+z)$ verifies:
\begin{coro}\label{c:determinant} Let $\varphi_f^t$ be a Morse-Smale gradient flow all of whose Lyapunov exponents are rationally independent. Then, one has, 
for every $\Lambda>0$ and for every $z$ in $\IC^*$,
\begin{equation}\label{e:superdeterminant}\prod_{k=0}^n\operatorname{det}\left(\Pi_{\Lambda}^{(k)}(\ml{L}_{V_f}^{(k)}+z)\Pi_{\Lambda}^{(k)}\right)^{(-1)^k}=z^{\chi(M)}.\end{equation}
where $\Pi_{\Lambda}^{(k)}$ is the spectral projector defined in paragraph~\ref{ss:correlation} and $\operatorname{det}$ is the standard determinant.
\end{coro}

\section{Topological considerations}\label{s:topology}

Studying Morse functions has deep connections with the topology of the manifold, and we will now describe 
some topological consequences of our spectral analysis of the operator $\ml{L}_{V_f}^{(*)}$. In all this section, we still suppose that $(f,g)$ is a smooth 
Morse pair inducing a Morse-Smale gradient flow but we do not suppose a priori that the Lyapunov exponents are rationally independent. 
The results presented here 
are in fact related to the description of the Morse complex given by Laudenbach in~\cite{Lau92, Lau12} and to the interpretation of Morse 
theory given by Harvey and Lawson in~\cite{HaLa01}. The main novelty here is the 
spectral interpretation of these results in analogy with Hodge-de Rham theory.

\subsection{De Rham cohomology}\label{ss:cohomology}

We start with a brief reminder on de Rham cohomology~\cite{dRh80, Sch66}. Recall that, for every $k\geq 0$, 
the coboundary operator $d$ sends any element in $\Omega^k(M)$ to an element in $\Omega^{k+1}(M)$, and that it satisfies $d\circ d=0$. 
In particular, one can define a cohomological complex $(\Omega^*(M),d)$ associated with $d$:
$$0\rightarrow \Omega^0(M)\rightarrow \Omega^{1}(M)\rightarrow \ldots\rightarrow \Omega^n(M)\rightarrow 0.$$
This complex is also called the de Rham complex. An element $\omega$ in $\Omega^*(M)$ such that $d\omega=0$ is called a \textbf{cocycle} while an 
element $\omega$ which is equal to $d\alpha$ for some $\alpha\in \Omega^*(M)$ is called a \textbf{coboundary}. We define then
$$Z^k(M)=\text{Ker}(d) \cap\Omega^k(M),\ \text{and}\ B^k(M)=\text{Im}(d) \cap\Omega^k(M).$$
Obviously, $B^k(M)\subset Z^k(M)$, and the quotient space $\mathbb{H}^k(M)= Z^k(M)/B^k(M)$ is called the \textbf{$k$-th de Rham cohomology}. 

According to~\cite[p.~344-345]{Sch66}, the coboundary operator $d$ can be extended into a map acting 
on the space of currents. This allows to define another 
cohomological complex $(\ml{D}^{\prime,*}(M),d)$:
$$0\rightarrow \ml{D}^{\prime,0}(M)\rightarrow \ml{D}^{\prime,1}(M)\rightarrow\ldots\rightarrow \ml{D}^{\prime,n}(M)\rightarrow 0,$$
where we recall that $\ml{D}^{\prime,k}(M)$ is the topological dual of $\Omega^{n-k}(M)$. One can similarly define the $k$-th cohomology  
of that complex. A remarkable result 
of de Rham is that these two cohomologies 
coincide~\cite[Ch.~4]{dRh80} -- see also~\cite[p.355]{Sch66} for a generalization of this result.
\begin{theo}[de Rham]\label{t:deRham} Let $u$ be an element in $\ml{D}^{\prime,k}(M)$ satisfying $du=0$. 
\begin{enumerate}
 \item There exists $\omega$ in $\Omega^k(M)$ such that $u-\omega$ belongs to $\operatorname{Im}(d)\cap\ml{D}^{\prime,k}(M)$.
 \item If $u=dv$ with $(u,v)$ in $\Omega^k(M)\times \mathcal{D}^{\prime,k-1}(M)$, then there exists $\omega$ in $\Omega^{k-1}(M)$ such that $u=d\omega$. 
\end{enumerate}
\end{theo}
\begin{rema}\label{r:derham} In the following, we will need something slightly more precise than (i). Namely 
suppose that $u$ is a cocycle in $\ml{H}^{m+k}_k(M)$. We denote by $\Delta_g^{(k)}$ the Laplace-Beltrami operator 
acting on $L^{2}(M,\Lambda^k(T^*M))$. We can find a pseudodifferential operator $A_k$ of order $-2$ such that
$u-\Delta_g^{(k)}A_ku$ belongs to $\Omega^k(M)$. As $u$ is a cocycle, one can deduce that $d\Delta_g^{(k)}A_ku\in\Omega^{k+1}(M)$. From the 
ellipticity of $\Delta_g$, we find that $dA_ku\in\Omega^{k+1}(M)$. This implies that $d^*dA_ku\in\Omega^{k}(M)$ and thus $u-dd^*A_k(u)$ 
belongs to $\Omega^k(M)$. As $u$ belongs to $\ml{H}^{m+k}_k(M)$, we get a refinement for point (i) in the sense that 
$u-\omega=dd^*A_k(u)$ belongs to $d(\ml{H}^{m+k-1}_{k-1}(M))$. 
\end{rema}

\subsection{Finite dimensional complexes}\label{ss:homology}

Proving that the $k$-th cohomology is finite dimensional requires more work -- see e.g.~\cite[Ch.~4]{dRh80}. 
Before deducing that result from our spectral analysis of $\ml{L}_{V_f}$, 
we start with some general considerations on finite dimensional 
cohomological complexes. Consider a cohomological complex 
$(C^*,d)$ associated with the coboundary operator:
$$0\rightarrow C^0\rightarrow C^1\rightarrow\ldots\rightarrow C^n\rightarrow 0,$$
where for every $0\leq k\leq n$, $C^k$ is a finite dimensional subspace of $\ml{D}^{\prime,k}(M)$. 
Consider now the complexes induced by the operator $\ml{L}_{V_f}$. For that purpose, we pick $N_0$ and $N_1$ large enough in the definition of the order function 
$m(x,\xi)$, and we define
$$C^{k}(f):=\text{Ker}\left(\ml{L}_{V_f}^{(k)}\right)$$
which is a finite dimensional space. Recall one more time from~\cite[Th.~1.5]{FaSj11} that these spaces are intrinsic in the sense that they do not depend on 
the choice of the order function $m$. As $d$ commutes with the Lie derivative $\ml{L}_{V_f}$, one can verify that, if $\ml{L}_{V_f}^{(k)}u=0$ with $u$ in $\ml{H}_k^m(M)$, 
then $\ml{L}_{V_f}^{(k+1)}(d u)=0$ with $d u$ 
belonging to $\ml{H}^{m+1}_{k-1}(M)$. Hence, the coboundary operator $d$ induces a \emph{finite dimensional cohomological complex} $(C^*(f),d)$
$$ 0\rightarrow C^0(f)\rightarrow C^1(f)\rightarrow \ldots\rightarrow C^n(f)\rightarrow 0.$$
We can now apply Propositions~\ref{p:generalized-laudenbach} and~\ref{p:basis} (for $\lambda=0$) and we find that $\text{dim}(C^k(f))$ 
is in fact equal to the number $c_k(f)$ of critical point of $f$ which are 
of index $k$.

\subsection{Morse type inequalities}\label{ss:morse}
Consider a finite dimensional complex $(C^*,d)$. We briefly recall how to obtain Morse type inequalities in that abstract framework arguing as in~\cite[Ch.~6]{Lau12}. For that purpose, we define
$$Z^k(C^*)=\text{Ker}(d) \cap C^k,\ \text{and}\ B^k(C^*)=\text{Im}(d) \cap C^k.$$
As above, we define the quotient space (or the $k$-th cohomology of the complex):
$$\mathbb{H}^k(C^*):=Z^k(C^*)/B^k(C^*).$$
We denote by $\beta_k(C^*)<\infty$ the dimension of that quotient space. We also introduce 
$$b_{k}(C^*)=\text{dim}\ B^k(C^*),\ c_{k}(C^*)=\text{dim}\ C^k,\ \text{and}\ z_{k}(C^*)=\text{dim}\ Z^k(C^*).$$
We observe that 
$$\beta_{k}(C^*)=z_k(C^*)-b_k(C^*)\ \text{and}\ c_k(C^*)=b_{k+1}(C^*)+z_k(C^*).$$
We now write that, for every $k\geq 0$,
$$0\leq b_{k+1}(C^*)=(c_{k}(C^*)-\beta_k(C^*))-(c_{k-1}(C^*)-\beta_{k-1}(C^*))+\ldots$$
From this expression, we can deduce the following \textbf{Morse type inequalities} associated with the complex $(C^*,d)$:
\begin{equation}\label{e:Morse-k}
 \forall 0\leq k\leq n,\ \sum_{j=0}^k(-1)^{k-j}c_{j}(C^*)\geq\sum_{j=0}^k(-1)^{k-j}\beta_{j}(C^*), 
\end{equation}
and using $b_{n+1}(C^*)=0$~:
\begin{equation}\label{e:Morse-n}
\sum_{j=0}^n(-1)^{n-j}c_{j}(C^*)=\sum_{j=0}^n(-1)^{n-j}\beta_{j}(C^*).
\end{equation}
In the case where we pick $C^*=C^*(f)$, inequalities~\eqref{e:Morse-k} and~\eqref{e:Morse-n} are exactly the Morse inequalities for the complex $C^*(f)$ which is nothing 
else but the Morse complex (also called Thom-Smale-Witten complex). 

\subsection{The Morse complex is isomorphic to the de Rham complex} Let $0\leq k\leq n$. 
We would like now to give \emph{a spectral proof} of the fact that the $k$-th cohomology of the Morse complex is isomorphic to the de Rham
cohomology of degree $k$. A proof of this result based on the theory of currents can be found in~\cite[Ch.~6]{Lau12} in the case of locally flat 
metrics -- see also~\cite{HaLa00, HaLa01}. 
Here, we give an alternative proof of that result based on our spectral analysis of the operator $\ml{L}_{V_f}$.

\subsubsection{Spectral decomposition}\label{sss:spectral-decomp} Let $m(x,\xi)$ be an order function with $N_0$ and $N_1$ sufficiently large to ensure that $0$ is an isolated eigenvalue with finite algebraic 
multiplicity (eventually equal to $0$). Introduce the \emph{spectral projector} associated with the eigenvalue $0$:
$$\mathbb{P}^{(k)}:=\int_{\gamma}\frac{dz}{(z-\ml{L}_{V_f}^{(k)})}:\ml{H}^m_k(M)\rightarrow C_k(f),$$
where $\gamma$ is a small Jordan path which separates $0$ from the rest of the spectrum of $\ml{L}_{V_f}^{(l)}$ acting on $\ml{H}^m_l(M)$ for every $0\leq l\leq n$ -- see~\cite[App.~A]{HeSj86}. 
This operator commutes with $\ml{L}_{V_f}^{(k)}$. According to~\cite[p.244-246]{EnNa00}, one knows that
$$\ml{L}_{V_f}^{(k)}:\left(\text{Id}_{\ml{H}_k^m}-\mathbb{P}^{(k)}\right)\ml{H}^m_k(M)\rightarrow \left(\text{Id}_{\ml{H}_k^m}-\mathbb{P}^{(k)}\right)\ml{H}^m_k(M)$$
does not contain $0$ in its spectrum. In particular, we can write the following decomposition:
$$\text{Id}_{\ml{H}_k^m}=\mathbb{P}^{(k)}+\ml{L}_{V_f}^{(k)}\circ\left((\ml{L}_{V_f}^{(k)})^{-1}\circ\left(\text{Id}_{\ml{H}_k^m}-\mathbb{P}^{(k)}\right)\right).$$ 
By Cartan's formula~\cite[p.~351]{Sch66}, one knows that $\ml{L}_{V_f}^{(k)}=i_{V_f}\circ d+d\circ i_{V_f}$. Hence,
\begin{equation}\label{e:resolution-id}\text{Id}_{\ml{H}_k^m}=\mathbb{P}^{(k)}+(d\circ i_{V_f}+i_{V_f}\circ d)\circ\left((\ml{L}_{V_f}^{(k)})^{-1}\circ\left(\text{Id}_{\ml{H}_k^m}-\mathbb{P}^{(k)}\right)\right).\end{equation} 
Moreover, $d$ commutes 
with $\ml{L}_{V_f}$ hence with $\mathbb{P}^{(k)}$ from the expression of the spectral projector. Hence, for every $u$ in 
$(\text{Id}-\mathbb{P}^{(k)})\ml{H}^m_k(M)$, $du\in(\text{Id}-\mathbb{P}^{(k+1)})\ml{H}_{k+1}^{m+1}(M)$, and 
one has $\ml{L}_{V_f}^{-1}\circ d u=\ml{L}_{V_f}^{-1}\circ d\circ\ml{L}_{V_f}\circ\ml{L}_{V_f}^{-1}(u)=
d\circ\ml{L}_{V_f}^{-1}(u)$. From that, we infer that $d$ also commutes with $\ml{L}_{V_f}^{-1}\circ\left(\text{Id}_{\ml{H}_k^m}-\mathbb{P}^{(k)}\right)$. 
Combining this last observation with the fact that $d$ commutes with $\mathbb{P}^{(k)}$ and with~\eqref{e:resolution-id}, we finally find that, 
for every $u$ in $\ml{H}_k^m(M)$, one has
\begin{equation}\label{e:hodge-type}
 u=\mathbb{P}^{(k)}(u)+d\circ R_{\infty}^{(k)} (u)+R_{\infty}^{(k+1)}\circ d(u),
\end{equation}
where $R_{\infty}^{(n+1)}=0$ and, for every $0\leq l\leq n$,
$$R_{\infty}^{(l)}:=i_{V_f}\circ (\ml{L}_{V_f}^{(l)})^{-1}\circ\left(\text{Id}_{\ml{H}_l^m}-\mathbb{P}^{(l)}\right).$$

\begin{rema}
Recall that, from our complete description of the spectrum of $\ml{L}_{V_f}$, one has, for every $u$ in $\ml{H}^m_k(M)$,
$$\mathbb{P}^{(k)}(u)=\lim_{t\rightarrow+\infty}\varphi_f^{-t*}u
=\sum_{a\in\text{Crit}(f):\ \text{ind}(a)=k}\la u,S_a\ra U_a\in C^k(f).$$
Up to its spectral interpretation, this type of  
``limit homotopy equation'' already appears in the works of Harvey and Lawson~\cite[Th.~2.3]{HaLa01}. 
\end{rema}


\subsubsection{Cohomological consequences}\label{sss:isomorphism} As the coboundary operator $d$ commutes with $\ml{L}_{V_f}$, it also commutes with $\mathbb{P}^{(k)}$. In particular, 
the map $\mathbb{P}^{(k)}$ induces a map from $Z^k(M)$ to $Z^k(C^*(f))$. We will now show (using our spectral approach) that it induces an isomorphism between the quotient spaces:
\begin{prop}\label{p:isomorphism} Let $0\leq k\leq n$. The map
$$\mathbb{P}^{(k)}:\Omega^k(M)\rightarrow C^k(f)$$
induces an isomorphism between the vector spaces $\mathbb{H}^k(C^*(f),d)$ and $\mathbb{H}^k(M)$.
\end{prop}

In particular, $\mathbb{H}^k(M)$ is \emph{a finite dimensional space} for every $0\leq k\leq n$, and its dimension is called the 
$k$-th Betti number that we will denote by $b_k(M)$. With the notations of paragraph~\ref{ss:homology}, we have $b_k(M)=\beta_k(C^*(f))$ for every 
$0\leq k\leq n$. In particular, if we apply~\eqref{e:Morse-k} and~\eqref{e:Morse-n} in the case of the complex $(C^*(f),d)$, we recover the classical Morse inequalities:
\begin{coro}[Morse inequalities]\label{c:Morse} Let 
$$c_k(f)=\vert\{a\in \operatorname{Crit}(f) \text{ s.t. } \operatorname{ind}(a)=k   \}\vert.$$
Then, for all $k\in \{0,\dots,n\}$, we have: 
\begin{eqnarray*}
\sum_{j=0}^k(-1)^{k-j}c_{j}(f)\geq\sum_{j=0}^k(-1)^{k-j}b_{j}(M),
\end{eqnarray*}
 with equality in the case\footnote{Recall that in that case, the sum is the Euler characteristic $\chi(M)$ of $M$.} $k=n$.
\end{coro}

\begin{proof}[Proof of Proposition~\ref{p:isomorphism}] Let us start with injectivity. Let $u$ be a cocycle in $\Omega^k(M)$ such that $\mathbb{P}^{(k)}(u)=0$. 
We use equality~\eqref{e:hodge-type}, and we find that
$$u=d\circ R_{\infty}^{(k)} (u),$$
which exactly says that $u$ is a coboundary for the complex $(\ml{D}^{\prime,*}(M),d)$. As $u$ is smooth, we know from de Rham Theorem~\ref{t:deRham} that $u$ is 
a coboundary in $\Omega^k(M)$.

Let us now consider the surjectivity. Fix $u$ a cocycle in $\text{Ker}(\ml{L}_{V_f}^{(k)}$. From Remark~\ref{r:derham}, we know that there exists 
$\omega\in\Omega^k(M)$ and $v$ in $\ml{H}_{k-1}^{m+k-1}(M)$ such that $u-\omega=dv$. Writing the cochain homotopy equation~\eqref{e:hodge-type} for 
$\omega$, we find that
$$\omega=\mathbb{P}^{(k)}(\omega)+d \circ R_{\infty}^{(k)}(\omega).$$
This implies that
$$u=\mathbb{P}^{(k)}(\omega)+d \left( R_{\infty}^{(k)}(\omega)+v\right).$$
By construction, $R_{\infty}^{(k)}(\omega)+v$ belongs to $\ml{H}_{k-1}^{m+k-1}(M)$. Hence, applying the spectral projector to the previous equality and as 
$d$ commutes with $\mathbb{P}^{(k)}$ (thanks to the integral expression of the spectral projector), we find that
$$u=\mathbb{P}^{(k)}(\omega)+d \circ\mathbb{P}^{(k)}\left( R_{\infty}^{(k)}(\omega)+v\right),$$
which proves the surjectivity.
\end{proof}



\subsection{Poincar\'e duality and $f\mapsto -f$}

By construction, one knows that the currents $(S_a)_{a\in\text{Crit}(f)}$ is the dual basis to $(U_a)_{a\in\text{Crit}(f)}$ for the duality 
bracket between $\ml{H}_k^m(M)$ and $\ml{H}_{n-k}^{-m}(M)$ which coincides (in the case of smooth forms) 
with the standard duality bracket between $\ml{D}^{\prime,*}(M)$ and $\Omega^{n-*}(M)$. Moreover, from paragraph~\ref{sss:correlation}, it 
is in fact a basis of the kernel of the operator $\ml{L}_{V_{-f}}^{(*)}$ acting on $\ml{H}^{-m}_{n-*}(M)$. We set
$$C^{n-k}(-f):=\text{Ker}(\ml{L}_{V_{-f}}^{(n-k)}).$$
We can then define the following complex associated with the coboundary operator $d$:
$$0\rightarrow C^0(-f)\rightarrow \ldots\rightarrow C^{n-1}(-f)\rightarrow C^n(-f)\rightarrow 0.$$
As was already explained, the two complexes $(C^*(f),d)$ and $(C^*(-f),d)$ are dual to each other via the duality between $\ml{H}_k^m(M)$ and $\ml{H}_{n-k}^{-m}(M)$, i.e.
$$\forall (u,v)\in C^k(f)\times C^{n-k}(-f),\ \la u,v\ra= \la u,v\ra_{\ml{H}_k^m(M),\ml{H}_{n-k}^{-m}(M)}=\int_Mu\wedge v.$$
Introduce now the following \textbf{Poincar\'e isomorphism} between $C^k(f)$ and the dual of $C^{n-k}(-f)$:
$$\ml{P}_0^{(k)}:u\in C^k(f)\mapsto \la u,.\ra\in C^{n-k}(-f)'.$$
We observe that $\la u,v\ra$ does not depend on the cohomology class of $u$ and $v$. Hence, $\ml{P}_0^{(k)}$ induces a linear map 
between $\mathbb{H}^k(C^*(f),d)$ and $\mathbb{H}^{n-k}(C^*(-f),d)'$. We now follow closely~\cite[Ch.~6]{Lau12} and verify that this 
is in fact an isomorphism betweeen the quotient spaces. Suppose that $\theta$ is a linear form on $\mathbb{H}^{n-k}(C^*(-f),d)$. This induces a linear form $\theta$ on 
$Z^{n-k}(C^*(-f),d)$ which vanishes on $B^{n-k}(C^*(-f),d)$. By the Hahn-Banach Theorem, we extend this linear form to $C^{n-k}(-f)$. From the 
duality between $C^{k}(f)$ and $C^{n-k}(-f)$, there exists a unique $u$ in $C^{k}(f)$ such that $\theta(v)=\la u,v\ra$ for every $v$ in $C^{n-k}(-f)$.
As $\theta$ vanishes on the image of $d$, we find that, for every $v$ in $C^{n-k-1}(-f)$, $\la u,dv\ra=0$ from which one can deduce that $du=0$. This shows 
surjectivity of the linear map induced by $\ml{P}_0^{(k)}$. If we intertwine the role of $f$ and $-f$, we find a linear surjection from $\mathbb{H}^{n-k}(C^*(-f),d)$ 
to $\mathbb{H}^{k}(C^*(f),d)'$. This implies that all the spaces have the same dimension. Hence, $\ml{P}_0^{(k)}$ induces an isomorphism 
between $\mathbb{H}^{k}(C^*(f),d)$ and $\mathbb{H}^{n-k}(C^*(-f),d)'$ for every $0\leq k\leq n$. Combined with Proposition~\ref{p:isomorphism} applied to both $f$ and $-f$, 
this implies the following well known result:
\begin{prop}\label{p:poincare} Let $M$ be a smooth, compact, oriented manifold without boundary. Then, for every $0\leq k\leq n$, $\ml{P}_0^{(k)}$ induces an isomorphism between 
$\mathbb{H}^k(C^*(f),d)$ and $\mathbb{H}^{n-k}(C^*(-f),d)'$. In particular, $b_k(M)=b_{n-k}(M)$ for every $0\leq k\leq n$.
 
\end{prop}


\subsection{Koszul complex associated with $i_{V_f}$}

In some sense, the Cartan formula
$$\ml{L}_{V_f}=d\circ i_{V_f}+i_{V_f}\circ d,$$
replaces in our context the formula $\Delta=d\circ d^*+d^*\circ d$ in Hodge theory. Hence, what plays the role in the Morse context
of the complex $(\Omega^*(M),d^*)$ from Hodge theory is the Koszul complex induced by the contraction operator $i_{V_f}$. 

We emphasize that the Cartan formula combined with our spectral decomposition yields an analogue of the Hodge 
decomposition in our framework:
$$u=\mathbb{P}^{(k)}(u)+d\left(i_{V_f}\circ(\ml{L}_{V_f}^{(k)})^{-1}\circ (\text{Id}-\mathbb{P}^{(k)})(u)\right)+i_{V_f}\left(d\circ(\ml{L}_{V_f}^{(k)})^{-1}\circ (\text{Id}-\mathbb{P}^{(k)})(u)\right).$$
 In other words, any $u$ in $\Omega^{k}(M)$ can be decomposed as the sum of an invariant current, of a coboundary (for $d$) and of a boundary (for $i_{V_f}$).

We now consider the Morse-Koszul homological complex $(C^*(f),i_{V_f})$
$$0\rightarrow C^n(f)\rightarrow C^{n-1}(f)\rightarrow\ldots \rightarrow C^0(f)\rightarrow 0.$$
Again, this is a well defined complex as $i_{V_f}$ commutes with the Lie derivative $\ml{L}_{V_f}$. 
Recall that the Euler characteristic of a homological complex $(C^*,i)$ is given by
$$\chi(C^*,i)=\sum_{j=0 }^n(-1)^j\text{dim}\left(\mathbb{H}_k(C^*,i)\right),$$
where $Z_j(C^*,i):=\text{Ker}(i)\cap C^j$, $\ B_j(C^*,i):=\text{Im}(i)\cap C^j$, and $Z_j(C^*,i)/B_j(C^*,i)$. We have the following property
\begin{prop}\label{p:Euler} Let $(f,g)$ be a smooth Morse pair inducing a Morse-Smale gradient flow. Then, one has
$$\mathbb{H}_k(C^*(f),i_{V_f})=C^k(f).$$
In particular, $\chi(C^*(f),i_{V_f})=\chi(M),$ where $\chi(M)$ is the Euler characteristic of the manifold.
\end{prop}
\begin{proof} Recall that $C^k(f)$ is equal to the vector space generated by the Laudenbach currents $U_a$ associated with critical points $a$ 
of index $k$. According to Proposition~\ref{p:generalized-laudenbach}, near a critical point $a$ of index $k$, $U_a$ can be written in the adapted coordinates of paragraph~\ref{sss:adapted} as 
$$U_a(x,y,dx,dy)=\delta_0^{\IR^k}(x)dx^1\wedge dx^2\wedge\ldots dx^k.$$
On the other hand, the vector field $V_f$ can be written in this system of coordinates:
$$V_f(x,y,\partial_x,\partial_y)=\sum_{j=1}^r\chi_j(a)x_j\partial_{x_j}+\sum_{j=r+1}^n\chi_j(a)y_j\partial_{y_j}.$$
 Hence, locally near $a$, one has
 $$i_{V_f}(U_a)(x,y,dx,dy)=\sum_{j=1}^r\chi_j(a)x_j\delta_0^{\IR^k}(x)dx_1\wedge\ldots \widehat{dx_j}\ldots\wedge dx_r=0.$$
As $U_a$ is supported in $\overline{W^u(a)}$, we can deduce that $i_{V_f}(U_a)$ is also carried by $\overline{W^u(a)}$. As we have just 
shown that it is equal to $0$ near $a$ and as $\ml{L}_{V_f}(i_{V_f}(U_a))=0$, we can deduce that the support of $i_{V_f}(U_a)$ is contained in 
$\overline{W^u(a)}-W^u(a)$. According to Remark~\ref{r:weber}, we can then deduce that the support of $i_{V_f}(U_a)$ is contained in the union of 
unstable manifold $W^u(b)$ with $\text{ind}(b)>k$. We now use Proposition~\ref{p:basis} (with $\lambda=0$) to write 
$$i_{V_f}(U_a)=\sum_{b':\text{ind}(b')=k-1}\alpha_{b'}U_{b'}.$$
Using~Proposition~\ref{p:generalized-laudenbach} and the fact that $i_{V_f}(U_a)$ is carried on a union of unstable manifold of index $>k$, we can deduce that $\alpha_{b'}=0$ 
for every critical point $b'$ of index $k-1$. In other words, $Z_k(C^*(f),i_{V_f})=C^k(f)$ and $B_k(C^*(f),i_{V_f})=\{0\}$. In particular, one has
$$\chi(C^*(f),i_{V_f})=\sum_{j=0 }^n(-1)^jc_{n-j}(f),$$
from which the last result follows thanks to the case of equality in the Morse inequalities.
\end{proof}

\appendix

\section{Proof of Lemma~\ref{l:escape-function}}\label{a:escape}

In this appendix, we give the proof of Lemma~\ref{l:escape-function}, i.e. construct of the escape function $G_m(x,\xi)$.
Let $N_0, N_1>4\|f\|_{\ml{C}^0}$ be some large parameters. As was already explained, up to some minor differences due to the special form of the dynamics, 
our construction is the one given in section~$2$ of~\cite{FaSj11}. Using the conventions of paragraph~\ref{ss:Hamiltonian}, we recall the following result~\cite[Lemma~2.1]{FaSj11}:
\begin{lemm}\label{l:faure-sjostrand} Let $V^u$ and $V^s$ be small open neighborhoods of $\Sigma_u$ and $\Sigma_s$ respectively, and let $\eps>0$. Then, 
there exist $\ml{W}^u\subset V^u$ and $\ml{W}^s\subset V^s$, $\tilde{m}$ in $\ml{C}^{\infty}(S^*M,[0,1])$, $\eta>0$ such that $\tilde{X}_{H_f}.\tilde{m}\geq 0$ on $S^*M$, 
$\tilde{X}_{H_f}.\tilde{m}\geq\eta>0$ on $S^*M-(\ml{W}^u\cup\ml{W}^s)$, $\tilde{m}(x,\xi)>1-\epsilon$ for $(x,\xi)\in \ml{W}^s$, 
$\tilde{m}(x,\xi)<\eps$ for $(x,\xi)\in \ml{W}^u$ and $\tilde{m}(x,\xi)<(1+\eps)/2$ for $(x,\xi)\notin V^s$. 
\end{lemm}

\begin{proof} Let us recall the main lines of the proof of this Lemma which relies only on the compactness and on the attracting properties of $\Sigma_u$ and $\Sigma_s$. First, we have 
to verify that, up to shrinking $V_u$ and $V_s$ a little bit, $V^u\cap V^s=\emptyset,$
\begin{equation}\label{e:inclusion}\forall t\geq 0,\ \tilde{\Phi}_f^t(V^s)\subset V^s,\ \text{and}\ \tilde{\Phi}_f^{-t}(V^u)\subset V^u.\end{equation}
This follows from Lemmas~\ref{l:compact} and~\ref{l:openneighbor}. Once we have this property, we can follow the proof of~\cite{FaSj11}. More precisely, we know that
$$\ml{I}(x,\xi):=\{t\in\IR:\ \tilde{\Phi}_f^t(x,\xi)\in S^*M-(V^u\cup V^s)\},$$
is a closed, connected interval whose length is uniformly bounded by some constant $\tau>0$. We then set $T>0$ such that $\tau/(2T)<\eps$ satisfying
$$\ml{W}^u:=\tilde{\Phi}_f^{-T} (S^*M-V^s)\subset V^u,\ \ml{W}^s:=\tilde{\Phi}_f^{-T} (S^*M-V^u)\subset V^s.$$
Once these parameters are fixed, one just has to verify that, if $m_0\in\ml{C}^{\infty}(S^*M,[0,1])$ is equal to $1$ on $V^s$ and to $0$ on $V^u$,
then the function
$$m_T(x,\xi):=\frac{1}{2T}\int_{-T}^Tm_0\circ\tilde{\Phi}_f^t(x,\xi)dt$$
satisfies the assumption of the Lemma -- see~\cite{FaSj11} for details. 
\end{proof}

We now use this Lemma with $V^u$, $V^s$ and $\eps>0$ small enough (to be precised). Thus, we have a function $\tilde{m}(x,\xi)$ defined on $S^*M$. We introduce a smooth function $m_1$ defined on $T^*M$ which satisfies
$$m_1(x,\xi)=N_1\tilde{m}\left(x,\frac{\xi}{\|\xi\|_x}\right)-N_0\left(1-\tilde{m}\left(x,\frac{\xi}{\|\xi\|_x}\right)\right),\ \text{for}\ \|\xi\|_x\geq 1,$$
and
$$m_1(x,\xi)=0,\ \text{for}\ \|\xi\|_x\leq \frac{1}{2}.$$
We set the order function of our escape function to be
$$m(x,\xi)=-f(x)+m_{1}(x,\xi).$$
Set now
$$\tilde{\Gamma}_{\mp}:=\left\{(x,\xi)\in T^*M:\ \xi\neq 0\ \text{and}\ \frac{\xi}{\|\xi\|_x}\in\ml{W}^{s/u}\right\}.$$
From the definition of $m$, $\tilde{\Gamma}_-$ (resp. $\tilde{\Gamma}_+$) is a small conical neighborhood of $\Gamma_-$ (resp. of $\Gamma_+$). Moreover, for every $(x,\xi)$ 
in $\tilde{\Gamma}_-$ (resp. $\tilde{\Gamma}_+$) satisfying $\|\xi\|_x\geq 1$, one has
$$m(x,\xi)\leq -N_0(1-\eps)+N_1\eps+\|f\|_{\ml{C}^0}\ \ \ (\text{resp.} \geq N_1(1-\eps)-N_0\eps-\|f\|_{\ml{C}^0}).$$
If we choose $\eps$ small enough, then the first items of Lemma~\ref{l:escape-function} are proved. We now set the following escape function:
$$G_{m}(x,\xi)=m(x,\xi)\log(1+\|\xi\|_x^2),$$
and we have to compute the derivative $X_{H_f}.G_m$ of $G_m$ along the Hamiltonian vector field $X_{H_f}$ associated with $H_f$. Note that
\begin{equation}\label{e:derivative-escape-function}X_{H_f}.G_m(x,\xi)=\log(1+\|\xi\|_x^2)X_{H_f}.m(x,\xi)+ m(x,\xi)\frac{X_{H_f}.\|\xi\|_x^2}{1+\|\xi\|_x^2}.\end{equation}
Let $r>0$ be a small parameter. We shall estimate the derivative of $G_m$ 
along the Hamiltonian function in $T^*B(a,r)$ for every critical points and in the complementary of this set.

Let us start with the case where $(x,\xi)$ belongs to $T^*M_{\text{reg}}$ 
where $M_{\text{reg}}$ is the complementary set of $\cup_{a\in\text{Crit} f}B(a,r/2)$. In that case, we fix $R_0>1$. Then, there exists $C_g$ depending only on the Riemannian metric and on $f$ such that, 
for every $(x,\xi)$ in $T^*M$ satisfying $\|\xi\|_x\geq R_0$,
$$X_{H_f}.G_m(x,\xi)\leq -X_{H_f}.f(x)\log(1+R_0^2)+C_g\left(N_0+N_1+\|f\|_{\ml{C}^0}\right),$$
as $X_{H_f}.m_1\leq 0$ for $\|\xi\|>1$ according to Lemma~\ref{l:faure-sjostrand}. As $x$ is far from the critical points of $f$, one knows from~\eqref{e:gradient-lines} that 
there exists a constant $c(r)>0$ (depending only on $r>0$) such that $X_{H_f}.f(x)\geq c(r)$. In particular, one has
$$X_{H_f}.G_m(x,\xi)\leq -c(r)\log(1+R_0^2)+C_g\left(N_0+N_1+\|f\|_{\ml{C}^0}\right)\leq -\min\{N_0,N_1\},$$
where the last equality holds if we choose $R_0>1$ large enough (in a way that depends on $r$, $N_0$ and $N_1$).

It now remains to analyse the behaviour in a ``neighborhood'' of a critical point $a$ in $\text{Crit} f$. In that case, we can as in~\cite{FaSj11} make use of the (local) 
hyperbolic structure of the flow. We fix 
$(x,\xi)$ in $T^*M$ such that $\|\xi\|_x\geq 1$ and $x$ in $B(a,r)$, and we use~\eqref{e:gradient-lines} to write
$$X_{H_f}.G_m(x,\xi)\leq X_{H_f}.m_1(x,\xi)\log(1+\|\xi\|_x^2)+ m(x,\xi)\frac{X_{H_f}.\|\xi\|_x^2}{1+\|\xi\|_x^2}.$$
We now distinguish three cases:
\begin{itemize}
 \item Suppose that we could show that, if $(x,\xi)$ belongs to $\tilde{\Gamma}_-$ and $\|\xi\|> 1$, then, one can find a constant\footnote{We note that we may have to take $V^s$ small enough.} 
 $c_->0$ \emph{depending only on} $f$ and $g$ such that
 \begin{equation}\label{e:unstable-cone}X_{H_f}.(\|\xi\|_x^2) > c_-\|\xi\|_x^2.\end{equation}
 In particular, one could infer 
 $$X_{H_f}.G_m(x,\xi)\leq -\frac{N_0}{2} c_-,$$
 where we used Lemma~\ref{l:faure-sjostrand} to bound $X_{H_f}.m_1(x,\xi)$. 
 \item Suppose that we could show that, if $(x,\xi)$ belongs to $\tilde{\Gamma}_+$ and $\|\xi\|> 1$, then, one can find a constant $c_+>0$ \emph{depending only on} $f$ and $g$ such that
 \begin{equation}\label{e:stable-cone}X_{H_f}.(\|\xi\|_x^2)<-c_+\|\xi\|_x^2.\end{equation}
 In particular, one could infer
 $$X_{H_f}.G_m(x,\xi)\leq -\frac{N_1}{2} c_+,$$
 where we used again Lemma~\ref{l:faure-sjostrand} to bound $X_{H_f}.m_1(x,\xi)$.
 \item If $(x,\xi)$ does not belong to $\tilde{\Gamma}_-\cup\tilde{\Gamma}_+$, then $\frac{\xi}{\|\xi\|}$ belongs to $S^*M-(\ml{W}^u\cup\ml{W}^s)$, and, by Lemma~\ref{l:faure-sjostrand}, one finds
 $$X_{H_f}.G_m(x,\xi)\leq -\eta (N_0+N_1)\log(1+\|\xi\|^2)+C_g\left(N_0+N_1+\|f\|_{\ml{C}^0}\right).$$
 Thus, if we choose $R_0>1$ large enough (in a way that depends on $N_0,N_1$), then one can ensure that $X_{H_f}.G_m(x,\xi)\leq -\min\{N_0,N_1\}$ whenever $\|\xi\|\geq R_0$ on this set.
\end{itemize}
This concludes the second part of the Lemma except for~\eqref{e:unstable-cone} and~\eqref{e:stable-cone} that are still to be proved. The proof is similar in both cases and we will only 
treat the first case. We note that the compactness of $\Sigma_u$ and $\Sigma_s$ will one more time play a crucial role in the proof. 

We start with the case where $(x,\xi)=(a,\xi)$ belongs to $E_s^*(a)$. In that case, we can make use of the fact that
we have a smooth linearizing chart -- see paragraph~\ref{sss:adapted}. Recall also that the linearized vector field $L_g(a)$ is diagonalizable in a basis of eigenvectors which 
is orthogonal for the metric $g^*(a)$ -- see paragraph~\ref{ss:Lyapunov}. This implies that, for $(a,\xi)=(0,\zeta)$ in $E_s^*(a)$,
$$X_{H_f}.(\|\xi\|^2)=2\sum_{j=1}^r|\chi_j(a)|\zeta_j^2.$$
In particular, provided that we take some small enough constant $c_-=c_-(f,g)$ 
depending on $g$ and $f$, inequality~\eqref{e:unstable-cone} holds in the case where $(x,\xi)=(a,\xi)$ belongs $E_s^*(a)$. We then define
$$U_r:=\left\{(x,\xi)\in\cup_{y\in B(a,r)}T_y^*M:\|\xi\|_x>1,\ \text{and}~\eqref{e:unstable-cone}\ \text{holds with}\ c_-=2c_-(f,g) \right\},$$
which is an open set in $T^*M-(M\times\{0\})$. Then, we have to prove that we can choose $V_s$ small enough to 
ensure that the neighborhood 
$$\tilde{V}_s^{(r)}:=\{(x,\xi)\in\cup_{y\in B(a,r)}T_y^*M:\|\xi\|>1\ \text{and}\ (x,\xi/\|\xi\|_x)\in V_s\}$$ 
is contained in $U_r$. We proceed by contradiction, and we suppose that, for every $r>0$ small enough, there exists $m_0\geq 1$ 
such that, for any $m\geq m_0$ and for any neighborhood $V_{m}$ of $\Sigma_s$ of size $1/m$, one can find $(x_m^{(r)},\xi_m^{(r)})\notin U_r$ belonging to $\tilde{V}_{m}^{(r)}$. 
Without loss of generality, we can suppose that $\|\xi_m^{(r)}\|\leq 2$. By compactness, we can then extract a 
subsequence such that $\lim_{m\rightarrow+\infty}(x_m^{(r)},\xi_m^{(r)})=(x^{(r)},\xi^{(r)})$. Moreover, as $\Sigma_s$ is compact, $(x^{(r)},\xi^{(r)})$ belongs to $\cup_{y\in \overline{B(a,r)}}(\Gamma_-\cap T_y^*M)$, 
and, by construction of the sequence,~\eqref{e:unstable-cone} does not hold at this point with the constant $c_-=2c_-(f,g)$. This holds for any $r>0$ small enough. We now extract a 
converging subsequence as $r\rightarrow 0^+$, and we find a point $(x,\xi)$ in $E^*_s(a)$ where we know that~\eqref{e:unstable-cone} holds with the constant $c_-=c_-(f,g)$. This gives the expected contradiction 
as $\xi\neq 0$.

\section{Proof of Proposition~\ref{p:eigenvalues}}\label{a:discrete}

The proof of this Proposition was given in great details in~\cite[Th.~1.4]{FaSj11} for the case $k=0$. The adaptation to the case $0\leq k\leq n$ 
is almost identical except that we have to deal with pseudodifferential operators with values in $\Lambda^k(T^*M)$. The main point is that 
the (pseudodifferential) operators under consideration have a scalar symbol. In fact, given any local basis $(e_j)_{j=1,\ldots J_k}$ 
of $\Lambda^k(T^*M)$ and any family $(u_j)_{j=1,\ldots J_k}$ of smooth functions $\ml{C}^{\infty}(M)$, one has
$$\ml{L}_{V_f}^{(k)}\left(\sum_{j=1}^{J_k}u_je_j\right)=\sum_{j=1}^{J_k}\ml{L}_{V_f}(u_j)e_j+\sum_{j=1}^{J_k}\ml{L}_{V_f}^{(k)}(e_j)u_j,$$
where the second part of the sum in the right-hand side is a lower order term (of order $0$). This scalar form allows to adapt the 
proofs of~\cite{FaSj11} to this vector bundle framework. For completeness, we briefly recall the main lines of the proof and just point a (minor) 
simplification due to the particular form of our flow. To make the 
comparison with that reference simpler, we shall consider the operator $-i\ml{L}_{V_f}$ instead of $-\ml{L}_{V_f}.$

\begin{rema} The case of currents was treated by Dyatlov and Zworski in~\cite{DyZw13} via a 
slightly different approach. Their method could also probably be adapted to deal with the case of Morse-Smale gradient flows. 
\end{rema}

The strategy is to consider the equivalent operator
\begin{equation}\label{e:conjugation}\widehat{\ml{L}}_{f}^{(k)}:=\Op(\mathbf{A}_m^{(k)})\circ (-i\ml{L}_{V_f}^{(k)})\circ\Op(\mathbf{A}_m^{(k)})^{-1},\end{equation} 
and to begin with, we recall the following result~\cite[Lemma~3.2]{FaSj11}:
\begin{lemm}\label{l:pseudo} The operator 
$$\Op(\mathbf{A}_m^{(k)})\circ (-i\ml{L}_{V_f}^{(k)})\circ\Op(\mathbf{A}_m^{(k)})^{-1}+i\ml{L}_{V_f}^{(k)}$$
is a pseudodifferential operator in $\Psi^{+0}(M)$ whose symbol in any given system of coordinates is of the form
$$P(x,\xi)=i(X_{H_f}.G_m)(x,\xi)\mathbf{Id}+\ml{O}(S^0)+\ml{O}_m(S^{-1+0}).$$
\end{lemm}

 In this Lemma, the notation $\ml{O}(.)$ means that the remainder is independent of the order function $m$, while the 
notation $\ml{O}_m(.)$ means that it depends on $m$. In particular, this Lemma says that $\widehat{\ml{L}}_{f}^{(k)}$ is an 
element in $\Psi^1(M,\Lambda^k(T^*M))$. Then, combining this Remark to~\cite[Lemma~A.1]{FaSj11} which can be adapted directly to 
the case of operators with values in a vector bundle, one finds that $\widehat{\ml{L}}_{f}^{(k)}$ has a unique closed extension 
as an unbounded operator on $L^2(M,\Lambda^k(T^*M))$. This shows the first part of the Proposition in the case $k=0$.

\begin{rema} The proof of this Lemma was given for $k=0$ in~\cite{FaSj11} and the adaptation to the case $1\leq k\leq n$ follows from the diagonal 
structure of the operators involved. Let us recall that the key idea is to observe, by linearizing the exponential,
\begin{align*} \Op(\mathbf{A}_m^{(k)})\circ (-i\ml{L}_{V_f}^{(k)})\circ\Op(\mathbf{A}_m^{(k)})^{-1} & \simeq & (1+\Op(G_m)+\ldots)
\circ(-i\ml{L}_{V_f}^{(k)})\circ (1-\Op(G_m)+\ldots)\\
 & =& -i\ml{L}_{V_f}^{(k)}+[\Op(G_m\mathbf{Id}),-i\ml{L}_{V_f}^{(k)}]+\ldots,
\end{align*}
which implies via symbolic calculus
 $$\Op(\mathbf{A}_m^{(k)})\circ (-i\ml{L}_{V_f}^{(k)})\circ\Op(\mathbf{A}_m^{(k)})^{-1}
 \simeq -i\ml{L}_{V_f}^{(k)}+i\Op(X_{H_f}.G_m\mathbf{Id})+\ldots.$$
\end{rema}

Then, up to the fact that we have to deal with $L^2(M,\Lambda^k(T^*M))$, the second part of 
Proposition~\ref{p:eigenvalues} is exactly the content of Lemma~3.3 of~\cite{FaSj11} which only 
makes use of the properties of the escape function given in Lemma~\ref{l:escape-function}. We also note 
that they implicitely shows that, 
for every $z$ in $\mathbb{C}$ satisfying $\text{Im} z>C_0$, one has
\begin{equation}\label{e:norm-resolvent}
 \left\|\left(\widehat{\ml{L}}_{f}^{(k)}-z\right)^{-1}\right\|_{L^2(M,\Lambda^k(T^*M))\rightarrow L^2(M,\Lambda^k(T^*M))}\leq\frac{1}{\text{Im}(z)-C_0}.
\end{equation}
\begin{rema}
Combining Proposition~\ref{p:eigenvalues} to the Hille-Yosida 
Theorem~\cite[Cor.~3.6, p.~76]{EnNa00}, one knows that (by conjugation)
\begin{equation}\label{e:semigroup}(\varphi_f^{-t})^*:\ml{H}^m_k(M)\rightarrow \ml{H}_k^m(M),\end{equation}
generates a strongly continuous semigroup which is defined for every $t\geq 0$ and whose norm is bounded by $e^{tC_0}.$
\end{rema}

Finally, the last part of the Proposition is based on results from analytic Fredholm theory. It 
is in fact the only place where things differ with~\cite{FaSj11}. The situation is in fact 
slightly simpler here as we shall now briefly explain it. We write 
$$\widehat{\mathbf{V}}_f^{(k)}:=\frac{i}{2}\left(\left(\widehat{\ml{L}}_{f}^{(k)}\right)^*-\widehat{\ml{L}}_{f}^{(k)}\right).$$
We denote by $\mathbf{V}_f^{(k)}(x,\xi)$ the symbol of this operator. Note that, from~\cite[Lemma~A.1]{FaSj11}, $(\widehat{\ml{L}}_{V_f}^{(k)})^*$ 
also has a unique closed extension to $L^2(M,\Lambda^k(T^*M))$. 
Combining~Lemma~\ref{l:pseudo} to Lemma~\ref{l:escape-function}, one knows that, for every $(x,\xi)$ in $T^*M$,
$$\mathbf{V}_f^{(k)}(x,\xi)\leq (-C_N+C)\mathbf{Id}+\ml{O}_m(S^{-1+0}),$$
for some constant $C>0$ which is independent of $m$ and for the constant $C_N$ defined in Lemma~\ref{l:escape-function}. From the sharp G\aa{}rding inequality, one can deduce that, for every $0<\mu<1$, there exists a constant $C_{\mu,m}>0$ such that, for every $u$ in $\ml{C}^{\infty}(M)$
$$\la(\widehat{\mathbf{V}}_f^{(k)}+C_N-C)u,u\ra_{L^2(M,\Lambda^k(T^*M))}\leq C_{\mu,m}\|u\|_{H^{\frac{\mu-1}{2}}(M,\Lambda^k(T^*M))}^2,$$
where the remainder $\ml{O}_m(S^{-1+0})$ has been absorbed in the RHS thanks to the Calder\'on-Vaillancourt Theorem. 
From this inequality, one can deduce that
$$\la(\widehat{\mathbf{V}}_f^{k)}+C_N-C)u,u\ra_{L^2(M,\Lambda^k(T^*M))}\leq 
\left\la \tilde{C}_{\mu,m}\left(1-\Delta_g^{(k)}\right)^{\frac{\mu-1}{2}}u,u\right\ra_{L^2(M,\Lambda^k(T^*M))},$$
where $\Delta_g^{(k)}$ is the Laplace-Beltrami operator acting on $k$ differential forms. We define then
$$\widehat{\chi}_k:=\tilde{C}_{\mu,m}\left(1-\Delta_g^{(k)}\right)^{\frac{\mu-1}{2}}\in\Psi^{\mu-1}(M,\Lambda^k(T^*M)),$$
which is a compact operator as $\mu-1<0$. Hence, we can rewrite the last inequality as
$$\la(\widehat{\mathbf{V}}_f^{(k)}-\widehat{\chi}_k+C_N-C)u,u\ra_{L^2(M,\Lambda^k(T^*M))}\leq 0,$$
from which one can deduce\footnote{The proof of this fact is similar to the proof of Lemma~3.3 in~\cite{FaSj11}.} that the resolvent
$$\left(\widehat{\ml{L}}_{f}^{(k)}-i\widehat{\chi}_k-z\right)^{-1}$$
defines a bounded operator from $L^2(M,\Lambda^k(T^*M))$ to itself as soon as $\text{Im}(z)>-(C_N-C)$. From the compactness 
of $\widehat{\chi}_k$ we can deduce that 
$\widehat{\chi}_k\left(\widehat{\ml{L}}_{f}^{(k)}-i\widehat{\chi}_k-z\right)^{-1}$ is also a compact operator 
which is exactly the content of Lemma~3.4 in~\cite{FaSj11}. The conclusion then follows by a classical argument from analytic Fredholm theory given in~\cite[Lemma~3.5]{FaSj11}.

\section{Asymptotic expansions}\label{a:asymptotic} 
In this appendix, we review some classical facts on asymptotic expansions (see~\cite[Ch.~1]{Igusa} for a nice review).

\begin{def1}
Let $\ml{I}$ be a discrete\footnote{We mean that it has no accumulation point.} countable subset of $\mathbb{R}$ bounded from below. We call $\ml{I}$ the \textbf{index set}.
Then $h\in C^\infty((0,1],\mathbb{R})$ has polyhomogeneous asymptotic expansion indexed by $\ml{I}$ if
\begin{eqnarray*}
\exists (a_\lambda)_{\lambda\in \ml{I}} \text{ such that }\forall \Lambda\in\mathbb{R}\setminus \ml{I}, \exists C\geqslant 0, \forall s\in (0,1] \\
\left\vert h(s)-\sum_{\lambda\in \ml{I},\lambda\leqslant\Lambda} a_\lambda s^\lambda \right\vert\leqslant Cs^{\lambda_0} 
\end{eqnarray*}
where $\lambda_0=\inf\{\lambda\in \ml{I}\cap[\Lambda,+\infty)\}$.
\end{def1}
The key property we use in this article is that, if such an asymptotic expansion exists, then it is \textbf{unique}.
We also have the following as a consequence of the Taylor formula:
\begin{lemm}\label{l:compo-smooth-becomes-polyhomogeneous}
Let $(\lambda_i)_{i=1}^n$ be a collection of $n$ positive real numbers. Let $\ml{I}$ be the index set defined as 
$$\ml{I}:=\left\{\sum_{j=1}^n k_j\lambda_j:\ \forall 1\leq j\leq n,\ k_j\in\mathbb{N}\right\}\subset \mathbb{R}.$$
Then, for all $\psi\in C^\infty(\mathbb{R}^n)$, the function
\begin{eqnarray*}
s\in(0,1] \longmapsto \psi(s^{\lambda_1},\dots,s^{\lambda_n})
\end{eqnarray*}
has polyhomogeneous asymptotic expansion indexed by $\ml{I}$.
\end{lemm}

A function $h\in C^\infty((0,1],\mathbb{R})$ is said to be \textbf{weakly homogeneous} 
if 
$$\exists C>0, \exists d\in \mathbb{R}, \forall s\in(0,1], \vert h(s) \vert \leqslant C s^d.$$
Recall that the Mellin transform of $h\mathbf{1}_{[0,1]}$ for
$f$ weakly homogeneous is then defined as
\begin{equation}
\ml{M}\left(h\mathbf{1}_{[0,1]}\right)(z)=\int_0^1 h(s)s^z\frac{ds}{s},
\end{equation}
and that it is holomorphic on the half--plane $\text{Re}(z)>-d$. Finally, we note that the following holds:

\begin{lemm}\label{l:mellin-transform-constraint}
Under the above conventions, one has:
\begin{enumerate}
\item For $w$ in $\mathbb{C}$, the Mellin transform
$\ml{M}(s^w\mathbf{1}_{[0,1]}(s))(z)$ equals $\frac{1}{w+z}$ and thus, it extends meromorphically with a simple pole at $z=-w$.
\item For every polyhomogeneous $h$ where $h\sim \sum_{\lambda\in \ml{I}}a_\lambda s^\lambda$, the Mellin transform
$\ml{M}\left(h\mathbf{1}_{[0,1]}\right)(z)$ extends meromorphically
to the complex plane with simple poles at $z\in -\ml{I}$.
\end{enumerate}
\end{lemm}
\begin{proof}
This is a particular case of~\cite[Thm 3.1 p.~11]{Igusa}
\end{proof}

\end{document}